\documentclass{article}
\usepackage{graphicx} 
\usepackage{amsmath,amssymb,amsthm,mathtools}
\usepackage{mathrsfs}
\usepackage{enumitem}
\usepackage{hyperref}
\usepackage{cleveref}
\usepackage{xcolor}
\usepackage{tensor} 

\usepackage{hyperref}
\hypersetup{urlcolor = blue, citecolor=blue, colorlinks=true}

\allowdisplaybreaks[3]

\theoremstyle{plain}
\newtheorem{fthm}{Theorem}[section]
\newtheorem*{fthm*}{Theorem}
\newtheorem{flemma}{Lemma}[section]
\newtheorem*{flemma*}{Lemma}
\newtheorem{fprop}{Proposition}[section]
\newtheorem*{fprop*}{Proposition}
\newtheorem{fcor}{Corollary}[section]
\newtheorem*{fcor*}{Corollary}

\theoremstyle{definition}
\newtheorem{fdefi}{Definition}[section]
\newtheorem*{fdefi*}{Definition}

\newtheorem*{fexmp*}{Example}

\newtheorem*{fex*}{Exercise}

\theoremstyle{remark}
\newtheorem{frmk}{Remark}[section]
\newtheorem*{frmk*}{Remark}

\newtheorem*{fconj*}{Conjecture}

\newtheorem*{fclaim*}{Claim}

\newtheorem*{fquest*}{Question}

\DeclareMathOperator\Ric{Ric}
\DeclareMathOperator\Riem{Riem}
\DeclareMathOperator\Sc{Sec} 
\DeclareMathOperator{\Hess}{Hess}
\DeclareMathOperator{\Id}{Id}
\DeclareMathOperator{\tr}{tr}
\DeclareMathOperator{\Gr}{Gr}
\DeclareMathOperator{\Span}{Span}
\DeclareMathOperator{\supp}{supp}
\DeclareMathOperator{\Vol}{Vol}

\newcommand{\cPac}{\mathcal{P}^{\text{ac}}_{\text{c}}}
\newcommand{\cW}{\mathscr{W}}  
\newcommand{\cH}{\mathcal{H}}  

\newcommand{\tJ}{\textbf{J}} 
\newcommand{\tU}{\textbf{U}} 
\newcommand{\tR}{\textbf{R}}  
\newcommand{\tH}{\textbf{H}} 

\newcommand{\tW}{\textbf{W}} 
\newcommand{\tB}{\textbf{B}} 

\newcommand{\Address}{{
  \bigskip
  \footnotesize

  Chao-Ming Lin, \textsc{Department of Mathematics, National Taiwan University, Taiwan.}\par\nopagebreak
  \textit{E-mail address:} \href{chaominl@ntu.edu.tw}{chaominl@ntu.edu.tw}\par\nopagebreak
  \textit{Personal Website:} \href{https://chaominl.github.io}{https://chaominl.github.io}

  \bigskip
  Kai-Hsiang Wang, \textsc{Faculty of Mathematics, Technion, Israel.}\par\nopagebreak
  \textit{E-mail address:}
  \href{khwang2025@campus.technion.ac.il}{khwang2025@campus.technion.ac.il}\par\nopagebreak
  \textit{Personal Website:}
  \href{https://k-hwang.github.io}{https://k-hwang.github.io}
}}

\title{Lower Bound for Weighted Intermediate Ricci Curvature and Tensorial Entropy Convexity}
\author{Chao-Ming Lin and Kai-Hsiang Wang}
\date{}

\begin{document}
\maketitle

\begin{abstract}
    We introduce a weighted version of intermediate Ricci curvature and establish several equivalent characterizations of its lower bound. As an application, we generalize the results of Aishwarya--Rotem--Shenfeld \cite{aishwarya2025sectionalcurvature} by deriving intrinsic-dimensional evolution variational inequalities and the corresponding Wasserstein contraction estimates for the heat flow.
    We also compare our characterization with that of Ketterer--Mondino \cite{ketterer_sectional_2018} via lower-dimensional optimal transport.
\end{abstract}


\section{Introduction}
One of the most profound achievements of optimal transport theory is
the connection it establishes between the geometry of a space and the
behavior of functionals on its associated Wasserstein space. Let
$(M,g)$ be a smooth complete Riemannian manifold without boundary,
equipped with the induced distance $d$ and volume measure
$\Vol_g$. Associated with $M$ is the Wasserstein space
$(\mathcal P_2(M),\mathscr W_2)$, where $\mathcal P_2(M)$ denotes the
space of probability measures on $M$ with finite second moment, and
$\mathscr W_2$ is the $L^2$-Wasserstein distance defined by
\begin{align*}
\label{eq:1.1}
\mathscr W_2(\mu_0,\mu_1) \coloneqq
\left(
\inf_{\pi\in\Pi(\mu_0,\mu_1)}
\int_{M\times M} d^2(x,y)\,d\pi(x,y)
\right)^{1/2},
\tag{1.1}
\end{align*}
where $\mu_0, \mu_1 \in \mathcal P_2(M)$ and $\Pi(\mu_0,\mu_1)$ denotes the collection of all couplings of
$\mu_0$ and $\mu_1$.

A central object in optimal transport is the \emph{Boltzmann entropy
functional} defined by
\begin{align*}
\label{eq:1.2}
H(\mu) \coloneqq
\begin{cases}
\displaystyle
\int_M
\log\!\left(\frac{d\mu}{d\Vol_g}\right)\,d\mu,
& \text{if }\mu\ll\Vol_g,\\[1em]
+\infty,
& \text{otherwise.}
\end{cases}
\tag{1.2}
\end{align*}
The celebrated works of Otto and Villani
\cite{otto_generalization_2000},
Cordero-Erausquin, McCann, and Schmuckenschläger
\cite{cordero-erausquin_riemannian_2001}, and
von Renesse and Sturm
\cite{von_renesse_transport_2005}
showed that the Boltzmann entropy functional is displacement convex
along Wasserstein geodesics if and only if $(M,g)$ has nonnegative
Ricci curvature.

A natural question is whether stronger curvature conditions admit analogous characterizations in terms of optimal transport and functional inequalities. This question was answered affirmatively by Aishwarya, Rotem, and Shenfeld in \cite{aishwarya2025sectionalcurvature}, where they characterized nonnegative sectional curvature of Riemannian manifolds by introducing
\begin{enumerate}
    \item Matrix displacement convexity in optimal transport theory;
    \item The nonnegativity of a tensorial Bakry--\'Emery operator.
\end{enumerate}
Their results may be viewed as tensorial analogues of the classical characterizations of nonnegative Ricci curvature arising from optimal transport theory \cite{otto_generalization_2000,cordero-erausquin_riemannian_2001,von_renesse_transport_2005} and the Bakry--Émery theory \cite{bakry_diffusions_1985}.

\medskip

In this paper, we extend the results of \cite{aishwarya2025sectionalcurvature} by establishing analogous characterizations of lower bounds for intermediate Ricci curvature, which interpolates between sectional and Ricci curvature and therefore includes both as special cases. Furthermore, our results apply to arbitrary lower curvature bounds, rather than only the nonnegative case. We also work in the more general setting of weighted Riemannian manifolds, with the classical unweighted case corresponding to a constant weight function. To state our main results, we first introduce the necessary terminology.

\begin{fdefi}[Weighted intermediate Ricci curvature]
Let \[(M^n,g,e^{-f}d\Vol_g)\] be a weighted Riemannian manifold,
let $x\in M$, let $k\in\{1,\dots,n\}$,
and let $N>k$.
Given a $k$-dimensional subspace $\Sigma\subset T_xM$ and a vector
$v\in T_xM$, we define the \emph{$N$-Bakry--Émery intermediate $k$-Ricci curvature} by
\begin{align*}
\label{eq:1.3}
\Ric_{k,f}^{N}(\Sigma,v)
\coloneqq
\Ric_k(\Sigma,v)
+
\Hess f(v,v)
-
\frac{1}{N-k}df(v)^2, \tag{1.3}
\end{align*}
where $\Ric_k$ denotes the intermediate $k$-Ricci curvature (see Definition \ref{def:interm_ric}).
We say that $\Ric^N_{k, f} \geq K$ if, for every $x \in M$, every vector $v \in T_x M$, and every $k$-dimensional subspace $\Sigma \subset T_x M$, we have
\begin{align*}
    \Ric^N_{k, f} (\Sigma, v) \geq  K |v|^2.
\end{align*}
\end{fdefi}
In the literature, several weighted curvature notions have been introduced. In particular, Bakry and Émery \cite{bakry_diffusions_1985} introduced the $N$-Bakry--Émery Ricci tensor for $n<N\le\infty$. This notion was subsequently extended to other values of $N$, including negative values; see, for example, \cite{wylie2016geometryriemannianmanifoldsdensity} and the references therein. On the other hand, Kennard, Wylie, and Yeroshkin \cite{kennard_weighted_2019} introduced the notions of positive and negative weighted sectional curvature.

\medskip

Following the approach of Aishwarya--Rotem--Shenfeld \cite{aishwarya2025sectionalcurvature}, we introduce a weighted $k$-Boltzmann entropy tensor associated with optimal transport. Let \((\mu_t)_{t\in[0,1]}\) be a Wasserstein geodesic in
\((\cPac(M),\mathscr W_2)\), where
\begin{align*}
\cPac(M)
\coloneqq
\bigl\{
\mu\in\mathcal P_2(M) \colon
\mu\ll\Vol_g \ \text{ and }
\mu\text{ is compactly supported}
\bigr\}.
\end{align*}
By the Brenier--McCann theorem, there exists a
\(\frac{d^2}{2}\)-concave function
\(-\theta:M\to\mathbb R\cup\{-\infty\}\) such that
\[
\mu_t=(F_t)_\sharp\mu_0,
\qquad
F_t(x)=\exp_x\bigl(t\nabla\theta(x)\bigr),
\qquad
t\in[0,1].
\]For every
\(x\in D(\mu_0,\mu_1)\),
where \(D(\mu_0,\mu_1)\) is the full
\(\mu_0\)-measure set on which
\(\Hess\theta(x)\) exists,
the minimizing geodesic
$\gamma(s)=F_s(x)$
induces a family of Jacobi field matrices
\((\tJ_s(x))_{s\in[0,1]}\) (see Section \ref{sec:2.4.2} for details).
We define
\[
\tU_s(x)
\coloneqq
\dot{\tJ}_s(x)\tJ_s(x)^{-1},
\]
where the dot denotes differentiation with respect to \(s\).
We then define the \emph{Boltzmann entropy tensor} associated with
\((\mu_0,\mu_1)\) by
\begin{align*}
\label{eq:1.4}
\tH_{t}^{\mu_0\rightarrow\mu_1}(x)
\coloneqq
-
\int_0^t
\tU_s(x)\,ds, \tag{1.4}
\end{align*}
and define the corresponding \emph{weighted \(k\)-Boltzmann entropy
tensor} by
\begin{align*}
\label{eq:1.5}
\tH_{t;k,f}^{\mu_0\rightarrow\mu_1}(x)
\coloneqq
\tH_{t}^{\mu_0\rightarrow\mu_1}(x)
+
\frac{
f\bigl(F_t(x)\bigr)-f(x)
}{k}\Id, \tag{1.5}
\end{align*}
for every \(x\in D(\mu_0,\mu_1)\).
We also define the \emph{weighted $k$-Boltzmann entropy functional} by
\begin{align*}
\label{eq:1.6}
    H_{k, f}(\mu) \coloneqq H(\mu) + \frac{n}{k} \int_M f \, d \mu, \tag{1.6}
\end{align*}where $\mu \in \cPac(M)$. The motivation for this definition is that the weighted
\(k\)-Boltzmann entropy functional can be recovered from the trace of
the weighted \(k\)-Boltzmann entropy tensor. More precisely, we have
    \begin{align*}
    \label{eq:1.7}
        &\kern-2em H_{k, f}(\mu_t) \\
        &= H_{k, f}(\mu_0) +  \int_M \tr \bigl( \tH^{\mu_0 \rightarrow \mu_1}_{t; k, f}(x)  \bigr)  \, d \mu_0(x) \\
        &= H_{k, f}(\mu_0)  +  \frac{n}{k} \int_M  \int_{\Gr(k, T_x M)} \tr\Bigl({\tH}^{\mu_0\to\mu_1}_{t;k,f}(x)\big|_\Sigma
\Bigr)  \, d \mu_{\Gr}(\Sigma) \, d \mu_0(x), \tag{1.7}
    \end{align*}where we denote by $\Gr(k, T_x M)$ the Grassmannian of $k$-dimensional subspace of $T_x M$, and \(\mu_{\Gr}\) the Haar probability measure on \(\Gr(k,T_xM)\).

\medskip

Inspired by the classical Bakry--Émery \(\Gamma_2\)-operator, we
introduce its tensorial analogue
\begin{align*}
\label{eq:1.8}
\widetilde{\Gamma}_2(\varphi)
\coloneqq
\frac12 \nabla^2 |\nabla \varphi|^2
-
\nabla_{\nabla \varphi}\nabla^2\varphi, \tag{1.8}
\end{align*}
as well as the tensorial weighted $(N,k)$-Bakry--Émery operator:
\begin{align*}
\label{eq:1.9}
\kern-2em \widetilde{\Gamma}_{2; k, f}^N (\varphi) \coloneqq \frac12 \nabla^2 |\nabla \varphi|^2
-
\nabla_{\nabla \varphi}\nabla^2\varphi + \frac{1}{k} \Bigl(  \Hess f( \nabla \varphi, \nabla \varphi)  - \frac{df(\nabla \varphi)^2}{N- k}  \Bigr) \Id.  \tag{1.9}
\end{align*}

Our main theorem shows that lower bounds on the weighted intermediate
Ricci curvature admit several equivalent characterizations, including tensorial Bochner inequalities and differential and
convexity properties of the weighted \(k\)-Boltzmann entropy tensor.

\hypertarget{T:1.1}{
\begin{fthm}
Let $(M^n,g,e^{-f}d\Vol_g)$ be a smooth complete weighted Riemannian manifold without boundary.
Then the following conditions are equivalent.
\begin{enumerate}
\item
The $N$-Bakry--Émery intermediate $k$-Ricci curvature satisfies
\[
\Ric_{k,f}^N\ge K.
\]
\item
For every $\psi\in C_c^\infty(M)$,
every $x\in M$, and every $k$-dimensional subspace
$\Sigma\subset T_xM$,
\begin{align*}
\label{eq:1.10}
\tr\Bigl(
\widetilde{\Gamma}_{2,k}^N(\psi)(x)\big|_\Sigma
\Bigr)
\ge
K|\nabla\psi(x)|^2.   \tag{1.10}
\end{align*}
\item
For every $\psi\in C_c^\infty(M)$,
every $x\in M$, and every $k$-dimensional subspace
$\Sigma\subset T_xM$,
\begin{align*}
\label{eq:1.11}
\tr\Bigl(
\widetilde{\Gamma}_{2,k}^N(\psi)(x)\big|_\Sigma
\Bigr)
\ge
\tr\Bigl(
\bigl((\nabla^2\psi)(x)\bigr)^2\big|_\Sigma
\Bigr)
+
K|\nabla\psi(x)|^2.  \tag{1.11}
\end{align*}
\item
Let
$\mu_0,\mu_1\in\mathcal P_c^{\mathrm{ac}}(M)$,
and let
$F_t(x)=\exp_x\bigl(t\nabla\theta(x)\bigr)$
be the optimal transport interpolation satisfying
$\mu_1=(F_1)_\sharp\mu_0$.
For any
$x\in D(\mu_0,\mu_1)$,
every $t\in[0,1]$,
and every
$k$-dimensional subspace
$\Sigma\subset T_xM$,
we have
\begin{align*}
\label{eq:1.12}
\tr\Bigl(
\ddot{\tH}^{\mu_0\to\mu_1}_{t;k,f}(x)\big|_\Sigma
\Bigr)
\ge
K|\nabla\theta(x)|^2
+
\frac{1}{N-k}df(\dot\gamma(t))^2,  \tag{1.12}
\end{align*}where $\gamma(t) \coloneqq \exp_x(t \nabla \theta(x))$.
\item
Under the same assumptions as in {\rm(4)},
\begin{align*}
\label{eq:1.13}
\tr\Bigl(
\ddot{\tH}^{\mu_0\to\mu_1}_{t;k,f}(x)\big|_\Sigma
\Bigr)
\ge
\frac1N
\Bigl[
\tr\Bigl(
\dot{\tH}^{\mu_0\to\mu_1}_{t;k,f}(x)\big|_\Sigma
\Bigr)
\Bigr]^2
+
K|\nabla\theta(x)|^2.   \tag{1.13}
\end{align*}
\item
Under the same assumptions as in condition {\rm(4)}, the function
\begin{align*}
\label{eq:1.14}
&\kern-4em t\longmapsto
e^{-\frac1N
\tr\bigl(
\tH^{\mu_0\to\mu_1}_{t;k,f}(x)\big|_\Sigma
\bigr)} \\
&\kern2em +
\frac KN
|\nabla\theta(x)|^2
\int_0^t
(t-s)
e^{-\frac1N
\tr\bigl(
\tH^{\mu_0\to\mu_1}_{s;k,f}(x)\big|_\Sigma
\bigr)}
\,ds  \tag{1.14}
\end{align*}
is concave on $[0,1]$.
\end{enumerate}
\end{fthm}}

As an immediate consequence of our main theorem, we obtain the
following characterization of nonnegative intermediate
\(k\)-Ricci curvature in the unweighted setting.
\begin{fthm}
    Let $(M, g)$ be a smooth complete Riemannian manifold without boundary. Then, the following are equivalent.
    \begin{enumerate}
        \item $M$ has nonnegative intermediate $k$-Ricci curvature.
        \item Let
$\mu_0,\mu_1\in\mathcal P_c^{\mathrm{ac}}(M)$,
and let
$F_t(x)=\exp_x\bigl(t\nabla\theta(x)\bigr)$
be the optimal transport interpolation satisfying
$\mu_1=(F_1)_\sharp\mu_0$.
For any
$x\in D(\mu_0,\mu_1)$,
every $t\in[0,1]$,
and every
$k$-dimensional subspace
$\Sigma\subset T_xM$,
we have
\begin{align*}
\label{eq:1.15}
\tr\Bigl(
\ddot{\tH}^{\mu_0\to\mu_1}_{t}(x)\big|_\Sigma
\Bigr)
\ge
\frac1k
\Bigl[
\tr\Bigl(
\dot{\tH}^{\mu_0\to\mu_1}_{t}(x)\big|_\Sigma
\Bigr)
\Bigr]^2. \tag{1.15}
\end{align*}
    \end{enumerate}
\end{fthm}
In view of the identity \eqref{eq:1.7}, the case \(k=n\) recovers the
classical characterization that the Boltzmann entropy functional is
displacement convex along Wasserstein geodesics if and only if
\((M,g)\) has nonnegative Ricci curvature.

\medskip

In addition, one of the principal applications of
Theorem~\hyperlink{T:1.1}{1.1} is the derivation of functional
inequalities under lower curvature bounds. Let
\((P_t)_{t\ge0}\) denote the heat semigroup generated by the
Laplace--Beltrami operator. A fundamental consequence of the
displacement convexity of the Boltzmann entropy is the
\emph{dimensional evolution variational inequality}, established by
Daneri--Savar\'e \cite{daneri2008eulerian} and
Erbar--Kuwada--Sturm \cite{erbar2015equivalence}. It states that,
when the Ricci curvature is nonnegative, we have
\begin{align*}
\label{eq:1.16}
\frac{d}{d \tau} \cW_2^2(P_\tau \mu_0,\mu_1)
\le
2n \Bigl(1 - \exp \Bigl(-\frac{1}{n}\bigl[H(\mu_1)-H(P_\tau \mu_0)\bigr] \Bigr) \Bigr). \tag{1.16}
\end{align*}

Building on the intrinsic-dimensional refinement of
estimate \eqref{eq:1.16} obtained by
Aishwarya--Rotem--Shenfeld
\cite{aishwarya2025sectionalcurvature},
we extend the dimensional evolution variational inequality to the
setting of weighted intermediate Ricci curvature.

\hypertarget{T:1.3}{
\begin{fthm} 
Let $(M^n,g,e^{-f}d\Vol_g)$ be a smooth compact weighted Riemannian manifold without boundary satisfying $\Ric^N_{k, f} \geq 0$. Let $\mu_0,\mu_1 \in \cPac(M)$, and let $(P_\tau)_{\tau\geq0}$ be the heat semigroup on $M$. Then, for every $T>0$ and for almost every
$\tau\in(0,T)$, we have
\begin{align*}
\label{eq:1.17}
&\kern-2em
\frac{d}{d\tau}\mathscr W_2^2(P_\tau\mu_0,\mu_1)
\\
&\le
\frac{2nN}{k}
-
\frac{2N}{\binom{n-1}{k-1}}
\int_M
\sigma_k \biggl(
\exp\biggl(
-\frac{\tH_{1;k,f}^{P_\tau\mu_0\rightarrow\mu_1}(x)}{N}
\biggr)
\biggr)
\,dP_\tau\mu_0(x) \\
&\kern2em
-
\frac{2n}{k}
\int_M
\left\langle
\nabla f,
\nabla\theta^{P_\tau\mu_0\rightarrow\mu_1}
\right\rangle
\,dP_\tau\mu_0,  \tag{1.17}
\end{align*}where \(\sigma_k\) denotes the \(k\)-th elementary symmetric polynomial.
\end{fthm}}

As is well known, evolution variational inequalities imply
contraction estimates for the Wasserstein distance along the heat flow.
In particular, by the works of
von Renesse--Sturm \cite{von_renesse_transport_2005} and
Bolley--Gentil--Guillin \cite{bolley2015equivalence}, when the Ricci
curvature is nonnegative, we have
\begin{align*}
         \cW^2_2(P_T \mu_0, P_T \mu_1) -   \cW_2^2(\mu_0, \mu_1) \leq   - 8n \int_0^T   \sinh^2 \biggl( \frac{H(P_\tau \mu_1) - H(P_\tau \mu_0) }{2n}    \biggr)    \, d \tau.
    \end{align*}

Building on the intrinsic-dimensional refinement of this estimate due
to Aishwarya--Rotem--Shenfeld
\cite{aishwarya2025sectionalcurvature},
we establish the following Wasserstein contraction estimate under
nonnegative weighted intermediate Ricci curvature.

\hypertarget{C:1.1}{
\begin{fcor}
    Let $(M^n,g,e^{-f}d\Vol_g)$ be a smooth compact weighted Riemannian manifold without boundary satisfying $\Ric^N_{k, f} \geq 0$. Let $\mu_0,\mu_1 \in \cPac(M)$, and let $(P_\tau)_{\tau\geq0}$ be the heat semigroup on $M$. Then, for every $T > 0$, we have
    \begin{align*}
        &\kern-2em \cW^2_2(P_T \mu_0, P_T \mu_1) - \Bigl[ \cW_2(\mu_0, \mu_1) +\frac{2nT}{k} \| \nabla f \|_{L^\infty} \Bigr]^2 \\
        & \leq   - \frac{8N}{\binom{n-1}{k-1}} \int_0^T \int_M  \tr_{\bigwedge^k T_x M} \biggl[  \sinh^2 \biggl( \frac{(\tH_{1; k, f}^{P_\tau \mu_0 \rightarrow P_\tau \mu_1}(x))^{[k]}}{2N}    \biggr) \biggr] \, d P_\tau \mu_0(x) \, d \tau.
    \end{align*}Here, for any self-adjoint operator
\(A:T_xM\to T_xM\),
we write $A^{[k]}:\bigwedge^kT_xM\to\bigwedge^kT_xM$
for the induced endomorphism on the \(k\)-th exterior power (see Section \ref{sec:2.3}).
\end{fcor}}

{\raggedright
\textbf{Organization of paper.}
In Section~\ref{sec:2}, we introduce the basic notions and establish the necessary preliminaries from Riemannian geometry, ordinary differential equations, linear algebra, and optimal transport. In Section~\ref{sec:3}, we prove our main result Theorem~\hyperlink{T:1.1}{1.1}, which characterizes lower bounds for the weighted intermediate Ricci curvature in terms of the weighted
\(k\)-Boltzmann entropy tensor and establishes tensorial Bochner
inequalities.
Finally, in Section~\ref{sec:4}, we establish the
intrinsic-dimensional evolution variational inequality (Theorem~\hyperlink{T:1.3}{1.3}) and the corresponding Wasserstein
contraction estimate for the heat flow
(Corollary~\hyperlink{C:1.1}{1.1}), together with a partial rigidity result
(Proposition~\hyperlink{P:4.1}{4.1}).
}

\medskip

{\raggedright
\textbf{Acknowledgement.}
The authors would like to thank Ming Hsiao, Pak-Yeung Chan, and Shouhei Honda for helpful comments and valuable suggestions. 
The first author is supported by the National Science and Technology Council, Taiwan (Grant No. NSTC 114-2115-M-002-012-MY3).
The research leading to these results is part of a project that has received funding from the European Research Council (ERC) under the European Union's Horizon 2020 research and innovation programme (grant agreement No 101001677).
}

\section{Preliminaries}
\label{sec:2}
In this section, we introduce the basic notions and establish the preliminary results from Riemannian geometry, ODEs, linear algebra, and optimal transport that will be used throughout the paper. These topics are presented in Sections~\ref{sec:2.1}, \ref{sec:2.2}, \ref{sec:2.3}, and \ref{sec:2.4}, respectively.

\subsection{Preliminaries on Riemannian geometry}
\label{sec:2.1}
In this section, we recall the basic notions and formulas from Riemannian geometry that will be needed throughout the paper. In Section~\ref{sec:2.1.1}, we review the notions of
Riemannian curvature, Ricci curvature, sectional curvature, and intermediate Ricci curvature on Riemannian manifolds. We then extend the notion of intermediate Ricci curvature to weighted Riemannian manifolds and introduce the corresponding curvature condition, which we call the $N$-Bakry--Émery intermediate $k$-Ricci curvature. In Section~\ref{sec:2.1.2}, we review Jacobi fields and the Jacobi equation, and derive a Riccati-type equation in the weighted setting. Finally, in Section~\ref{sec:2.1.3}, we present a tensorial version of the Bochner formula and introduce the corresponding Bakry--Émery operator for weighted Riemannian manifolds.

\subsubsection{Riemannian curvature and intermediate Ricci curvature}
\label{sec:2.1.1} 
Let $(M^n,g)$ be an $n$-dimensional smooth complete Riemannian manifold
without boundary.
For $x\in M$, choose local coordinates
$\{x^1,\dots,x^n\}$ in a neighborhood of $x$ and define the metric
coefficients by
\[
g_{ij}
\coloneqq
g\Bigl(
\frac{\partial}{\partial x^i},
\frac{\partial}{\partial x^j}
\Bigr).
\]
We denote by $(g^{ij})$ the inverse matrix of
$(g_{ij})$ and, for convenience, we write
\[
\partial_i
:=
\frac{\partial}{\partial x^i}.
\]
Let $\nabla$ denote the Levi-Civita connection of $g$, namely the unique torsion-free affine connection on the tangent bundle $TM$ that is compatible with the Riemannian metric.
Its Christoffel symbols are defined by
\[
\nabla_{\partial_i}\partial_j
=
\Gamma_{ij}^k \partial_k,
\]where
$\Gamma_{ij}^k
=
\frac12
g^{k l}
\left(
\partial_i g_{j l}
+
\partial_j g_{i l}
-
\partial_l g_{ij}
\right)$.
For smooth vector fields $X,Y\in\Gamma(TM)$, the Riemann curvature tensor is defined by
\begin{align*}
\label{eq:2.1}
\Riem(X,Y)
\coloneqq
\nabla_X\nabla_Y
-
\nabla_Y\nabla_X
-
\nabla_{[X,Y]}. \tag{2.1}
\end{align*}
In local coordinates, we write
\begin{align*}
\label{eq:2.2}
\Riem(\partial_i,\partial_j)\partial_k
=
\tensor{R}{_{ijk}^l} \partial_l,  \tag{2.2}
\end{align*}
where
$\tensor{R}{_{ijk}^l}
=
\partial_i\Gamma_{jk}^{l}
-
\partial_j\Gamma_{ik}^{l}
+
\Gamma_{im}^{l}\Gamma_{jk}^{m}
-
\Gamma_{jm}^{l}\Gamma_{ik}^{m}$.
The corresponding $(0,4)$-curvature tensor is defined by
\begin{align*}
\label{eq:2.3}
R_{ijkl}
\coloneqq
g\bigl(
\Riem(\partial_i,\partial_j)\partial_k,
\partial_l
\bigr). \tag{2.3}
\end{align*}
Equivalently,
$R_{ijkl}
=
g_{ml}   \tensor{R}{_{ijk}^m}$.

The Ricci curvature tensor is obtained by taking the trace of the curvature tensor. In local coordinates, we have
\begin{align*}
\label{eq:2.4}
\Ric_{ij}
\coloneqq
R_{kij}{}^{k}
=
g^{lk}R_{kijl}. \tag{2.4}
\end{align*}
Equivalently, for vector fields $X,Y\in\Gamma(TM)$,
\begin{align*}
\label{eq:2.5}
\Ric(X,Y)
=
\sum_{i=1}^{n}
g\bigl(
\Riem(e_i,X)Y,
e_i
\bigr), \tag{2.5}
\end{align*}
where $\{e_1,\dots,e_n\}$ is any orthonormal frame. Note that the expression \eqref{eq:2.5} does not depend on the choice of orthonormal frame.



We next recall the notion of sectional curvature. Given a point $x\in M$ and a two-dimensional subspace $P\subset T_xM$ spanned by vectors $u,v\in T_xM$, the sectional curvature of $P$ is defined by
\begin{align*}
\label{eq:2.6}
\Sc_x(P)
=
\frac{
g_x\bigl(\Riem_x(u,v)v,u\bigr)
}{
|u|^2|v|^2-g_x(u,v)^2
}. \tag{2.6}
\end{align*}
When $P=\Span\{u,v\}$, we also write
$\Sc_x(u,v)
\coloneqq
\Sc_x(P)$.
In particular, if $u$ and $v$ are orthonormal, then we get
\[
\Sc_x(P)
=
g_x\bigl(\Riem_x(u,v)v,u\bigr).
\]


The Ricci curvature can be expressed as a sum of sectional curvatures. Let $x\in M$ and let $e\in T_xM$ be a unit vector. By choosing vectors
$e_1,\dots,e_{n-1}\in T_xM$ such that
$\{e,e_1,\dots,e_{n-1}\}$
forms an orthonormal basis of $T_xM$, if we set
\[
P_i
=
\Span\{e,e_i\},
\qquad
i=1,\dots,n-1,
\]
then we have
\[
\Ric_x(e,e)
=
\sum_{i=1}^{n-1}\Sc_x(P_i).
\]
Thus, the Ricci curvature in the direction $e$ is obtained by summing the sectional curvatures of all two-planes $P_i$ for $i$ from $1$ to $n-1$. Note that similarly the expression does not depend on the choice of orthonormal frame.  

Motivated by Ricci curvature and sectional curvature, we introduce the notion of intermediate \(k\)-Ricci curvature.

\begin{fdefi}[Intermediate $k$-Ricci curvature]\label{def:interm_ric}
Let $(M^n,g)$ be a Riemannian manifold, let $x\in M$, and let
$k\in\{1,\dots,n\}$.
For a $k$-dimensional subspace $\Sigma\subset T_xM$ and a vector
$v\in T_xM$, the \emph{intermediate $k$-Ricci curvature} of $\Sigma$
in the direction $v$ is defined by
\[
\Ric_k(\Sigma,v)
\coloneqq
\tr \Bigl[
\pi_\Sigma\circ
\bigl(\Riem(\bullet,v)v\bigr)\big|_\Sigma
\Bigr].
\]
Equivalently, if $\{e_1,\dots,e_k\}$ is an orthonormal basis of
$\Sigma$, then
\begin{align*}
\label{eq:2.7}
\Ric_k(\Sigma,v)
=
\sum_{i=1}^k
\Sc_x(v,e_i)
\bigl(
|v|^2-\langle v,e_i\rangle^2
\bigr), \tag{2.7}
\end{align*}
where $\pi_\Sigma:T_xM\to\Sigma$ denotes the orthogonal projection.
\end{fdefi}

\begin{frmk}
Indeed, the quantity $\Ric_k(\Sigma,v)$ does not depend on the choice of orthonormal basis
$\{e_1,\dots,e_k\}$ of $\Sigma$.
\end{frmk}

We next record several special cases illustrating the relationship
between the intermediate $k$-Ricci curvature and the classical
sectional and Ricci curvatures.

\begin{frmk}[Special cases]
For $k =1$ and $k = 2$, the intermediate $k$-Ricci curvature is closely related to sectional curvature. For $k = n-1$ and $k = n$, the intermediate $k$-Ricci curvature recovers the usual Ricci curvature in certain directions. 
\end{frmk}


\bigskip

We now turn to the more general framework of weighted Riemannian manifolds.
Given a smooth function $f\in C^2(M)$, we equip $(M^n,g)$ with the weighted measure
\[
e^{-f}\,d\Vol_g,
\]
where $d\Vol_g$ denotes the Riemannian volume measure induced by $g$.
The resulting triple
\[
(M^n,g,e^{-f}d\Vol_g)
\]
is called a \emph{weighted Riemannian manifold}.
In this setting, the notion of intermediate Ricci curvature admits a natural extension, which we call the \emph{Bakry--Émery intermediate Ricci curvature}.



\begin{fdefi}[Bakry--Émery intermediate $k$-Ricci curvature]
Let \[(M^n,g,e^{-f}d\Vol_g)\] be a weighted Riemannian manifold,
let $x\in M$, let $k\in\{1,\dots,n\}$,
and let $N>k$.
For a $k$-dimensional subspace $\Sigma\subset T_xM$ and a vector
$v\in T_xM$, we define the \emph{Bakry--Émery intermediate $k$-Ricci curvature} by
\begin{align*}    
\label{eq:2.8}
\Ric_{k,f}(\Sigma,v)
\coloneqq
\Ric_k(\Sigma,v)
+
\Hess f(v,v). \tag{2.8}
\end{align*}
In addition, we define
the \emph{$N$-Bakry--Émery intermediate $k$-Ricci curvature} by
\begin{align*}
\label{eq:2.9}
\Ric_{k,f}^{N}(\Sigma,v)
\coloneqq
\Ric_k(\Sigma,v)
+
\Hess f(v,v)
-
\frac{1}{N-k}df(v)^2. \tag{2.9}
\end{align*}
\end{fdefi}




\begin{fdefi}[Intermediate $k$-Ricci upper and lower bounds]
We say that $(M,g, e^{-f} d \Vol_g)$ has Bakry--Émery intermediate $k$-Ricci curvature bounded below
(respectively bounded above) by $K$ if, for any $x\in M$, any vector
$v\in T_xM$, and any $k$-dimensional subspace
$\Sigma\subset T_xM$, we have
\[
\Ric_{k, f}(\Sigma,v)\ge K|v|^2
\qquad
\text{(respectively } \Ric_{k, f}(\Sigma,v)\le K|v|^2\text{)}.
\]
For convenience, we write
\[
\Ric_{k, f} \ge K
\qquad
\text{(respectively } \Ric_{k, f} \le K\text{)}.
\]Similarly, we say that $(M,g, e^{-f} d \Vol_g)$ has $N$-Bakry--Émery intermediate $k$-Ricci curvature bounded below
(respectively bounded above) by $K$ if, for any $x\in M$, any vector
$v\in T_xM$, and any $k$-dimensional subspace
$\Sigma\subset T_xM$, we have
\[
\Ric_{k, f}^N(\Sigma,v)\ge K|v|^2
\qquad
\text{(respectively } \Ric_{k, f}^N(\Sigma,v)\le K|v|^2\text{)}.
\]
For convenience, we write
\[
\Ric_{k, f}^N\ge K
\qquad
\text{(respectively } \Ric_{k, f}^N \le K\text{)}.
\]
\end{fdefi}




The curvature bounds introduced above satisfy a monotonicity
property with respect to the intermediate dimension \(k\), provided
that the synthetic dimension parameter \(N\) and the weight function \(f\) are
rescaled appropriately. We state this observation in the following
remark.

\begin{frmk}
Let $k_1,k_2\in\{1,\dots,n\}$ with $k_1\le k_2$, and assume
$N_1>k_1$. If
$\Ric^{N_1}_{k_1,f}\ge K_1$,
then by defining $N_2 \coloneqq \frac{k_2}{k_1}N_1$, we have
\[
\Ric^{N_2}_{k_2,\frac{k_2}{k_1}f}
\ge
\frac{k_2}{k_1}K_1,
\]

Indeed, let $\Sigma\subset T_xM$ be a $k_2$-dimensional subspace, and let
$\{e_1,\dots,e_{k_2}\}$ be an orthonormal basis of $\Sigma$. Then, we have
\begin{align*}
\Ric^{N_2}_{k_2,\frac{k_2}{k_1}f}(\Sigma,v)
&=
\Ric_{k_2}(\Sigma,v)
+\frac{k_2}{k_1}\Hess f(v,v)
-\frac{\left(\frac{k_2}{k_1}\right)^2 }{N_2-k_2} df(v)^2  \\
&=
\Ric_{k_2}(\Sigma,v)
+\frac{k_2}{k_1}\Hess f(v,v)
-\frac{k_2}{k_1}\frac{1}{N_1-k_1}df(v)^2.
\end{align*}
For each $I\subset\{1,\dots,k_2\}$ with $|I|=k_1$, set
$\Sigma_I:=\operatorname{span}\{e_i:i\in I\}$.
Since each index $i$ appears in exactly
$\binom{k_2-1}{k_1-1}$ such subsets $I$, we have
\begin{align*}
&\kern-2em \Ric^{N_2}_{k_2,\frac{k_2}{k_1}f}(\Sigma,v) \\
&=
\frac{1}{\binom{k_2-1}{k_1-1}}
\sum_{\substack{I\subset\{1,\dots,k_2\}\\ |I|=k_1}}
\Ric^{N_1}_{k_1,f}(\Sigma_I,v) \ge
\frac{1}{\binom{k_2-1}{k_1-1}}
\binom{k_2}{k_1}K_1|v|^2 =
\frac{k_2}{k_1}K_1|v|^2.
\end{align*}
\end{frmk}

\subsubsection{Jacobi fields}
\label{sec:2.1.2}
In this section, we recall the definition of Jacobi fields.
Let $\gamma \colon [0, 1] \rightarrow M$ be a $C^2$-curve, its velocity vector field is defined by
\[\dot{\gamma}(s) \coloneqq \frac{d}{ds}\gamma(s) \in T_{\gamma(s)} M.\]

A $C^2$-curve $\gamma$ is called a \emph{geodesic} if its velocity vector field is parallel along $\gamma$, that is,
\begin{align*}
\label{eq:2.10}
    \nabla_{\dot{\gamma}} \dot{\gamma}(s) = 0 \qquad \text{ for all } \  s \in [0, 1]. \tag{2.10}
\end{align*}
Since the geodesic equation is a second-order system of ordinary
differential equations with smooth coefficients, standard regularity
theory implies that every geodesic is in fact smooth.

Now, let $\gamma$ be a geodesic, and consider a smooth map 
\begin{align*}
    F \colon (- \epsilon, \epsilon) \times [0, 1] &\longrightarrow M, \\
    (r,s) &\longmapsto \gamma_r(s)
\end{align*}such that $\gamma_0 = \gamma$ and for each fixed $r$, the curve $s \mapsto \gamma_r(s)$ is a geodesic.

Since each $\gamma_r$ is a geodesic, by \eqref{eq:2.10}, we have
\begin{align*}
\label{eq:2.11}
\nabla_{\partial_sF}\partial_sF=0,     \tag{2.11}
\end{align*} 
where, for convenience, we write $\partial_s \coloneqq \frac{\partial}{\partial s}$ and $\partial_r \coloneqq \frac{\partial}{\partial r}$.
By taking the covariant derivative $\nabla_{\partial_r F}$ of \eqref{eq:2.11}, we obtain
\begin{align*}
\label{eq:2.12}
\nabla_{\partial_rF}
\nabla_{\partial_sF}
\partial_sF
=
0  \qquad \text{ for all } \ r, s.     \tag{2.12}
\end{align*}
Since the Levi-Civita connection is torsion-free and
$[\partial_r,\partial_s]=0$,
\eqref{eq:2.12} implies that
\begin{align*}
\label{eq:2.13}
\Riem(\partial_rF,\partial_sF)\partial_sF
+
\nabla_{\partial_sF}
\nabla_{\partial_sF}
\partial_rF
=
0.    \tag{2.13}
\end{align*}

The vector field along $\gamma$ defined by
\[
J(s)
\coloneqq
\partial_rF(r,s)
\Big|_{r=0}
\]
is called the \emph{Jacobi field} associated to the variation $F$.
By evaluating \eqref{eq:2.13} at $r=0$, we obtain the \emph{Jacobi equation}.
\begin{align*}
\label{eq:2.14}
\nabla_{\dot\gamma}
\nabla_{\dot\gamma}
J
+
\Riem(J,\dot\gamma)\dot\gamma
=
0.   \tag{2.14} 
\end{align*}

Conversely, the Jacobi equation \eqref{eq:2.14} is a linear second-order ordinary
differential equation along the geodesic $\gamma$. Therefore, for any
prescribed initial data
\[
J(0)
\qquad\text{and}\qquad
\left.
\nabla_{\dot\gamma}J
\right|_{s=0},
\]
there exists a unique Jacobi field along $\gamma$ satisfying these
conditions. Moreover, every Jacobi field arises from a variation of
$\gamma$ through geodesics.

Now, consider a family of Jacobi fields $\{ J^i(s) \}_{i \in \{1, \cdots, n \}}$
along a fixed geodesic $\gamma$.
Let $\{e^i(0) \}_{i=1}^n$ be an orthonormal basis of
$T_{\gamma(0)}M$ such that $e^1(0)$ is parallel to $\dot{\gamma}(0)$
and let $\{e^i(s)\}_{i=1}^n$ denote the parallel transport of
$\{e^i(0) \}_{i=1}^n$ using $\nabla$ along $\gamma$,
that is,
\[
\nabla_{\dot{\gamma}} e^i(s) = 0
\]
for all $s \in [0, 1]$ and $i \in \{1, \cdots, n\}$.

It is convenient to describe the evolution of the Jacobi fields using this parallel frame.
Let $\textbf{J}(s)$ be the $n\times n$ matrix whose $j$-th column consists of the coordinates of the Jacobi field $J^j(s)$ with respect to the basis $\{e^i(s) \}_{i = 1}^n$. 

Then the Jacobi equation \eqref{eq:2.14} can be written in matrix form as
\begin{align*}
\label{eq:2.15}
\ddot{\textbf{J}}(s)+ \textbf{R}(s) \textbf{J}(s)=0,   \tag{2.15}
\end{align*}
where $\dot{\tJ}(s)$ and $\ddot{\tJ}(s)$ denote the first and second derivatives
with respect to $s$, and $\tR(s)$ is the symmetric $n\times n$ matrix whose
entries
\[
\tR(s)= \begin{bmatrix}
    R_{ij}(s)
\end{bmatrix}_{i,j=1}^n
\]
are given by
\begin{align*} 
&\kern-2em R_{ij}(s) \\
&\coloneqq
g_{\gamma(s)}
\bigl(
\Riem
\bigl(
e^j(s), \dot{\gamma}(s)
\bigr)
\dot{\gamma}(s),
e^i(s)
\bigr)
=g_{\gamma(s)}
\bigl(
\Riem
\bigl(
e^i(s), \dot{\gamma}(s)
\bigr)
\dot{\gamma}(s),
e^j(s)
\bigr).
\end{align*}

When $\tJ(s)$ is invertible, we define
\begin{align*} 
\label{eq:2.16}
\tU(s) \coloneqq \dot \tJ(s) \tJ^{-1}(s) \tag{2.16}
\end{align*}and introduce its weighted version
\begin{align*}
\label{eq:2.17}
    \tU_{k, f}(s) \coloneqq \tU(s) - \frac{1}{k} \langle \nabla f, \dot \gamma(s) \rangle \Id. \tag{2.17}
\end{align*}

It follows from \eqref{eq:2.15} that $\tU(s)$ satisfies the matrix Riccati equation
\begin{align*}
\label{eq:2.18}
\dot \tU(s) + \tU(s)^2 + \tR(s) = 0 \tag{2.18}
\end{align*}while $\tU_{k, f}(s)$ satisfies the following weighted matrix Riccati equation 
\begin{align*}
\label{eq:2.19}
    \dot \tU_{k, f}(s) &= \dot \tU(s) - \frac{1}{k} \nabla^2f(\dot \gamma, \dot \gamma) \Id = -\tU(s)^2 - \tR(s) - \frac{1}{k} \nabla^2f(\dot \gamma, \dot \gamma) \Id \\
    &= - \tU_{k, f}(s)^2 -  \frac{2}{k} \langle \nabla f, \dot \gamma(s) \rangle \tU_{k, f}   - \tR(s)  \\
    &\kern2em - \frac{1}{k} \Bigl(   \nabla^2f(\dot \gamma, \dot \gamma) + \frac{1}{k} \langle \nabla f, \dot \gamma(s) \rangle^2 \Bigr) \Id. \tag{2.19}
\end{align*}




\subsubsection{The Bochner formula}
\label{sec:2.1.3}
The Bochner formula is one of the fundamental identities in Riemannian geometry.
It relates the Laplacian of the squared gradient of a smooth function to its
Hessian and the Ricci curvature, thereby revealing how curvature influences the
behavior of second-order differential operators.

For any smooth function $\varphi \in C^\infty(M ; \mathbb R)$ and $x \in M$,
we have 
\begin{align*}
\label{eq:2.20}
-\frac12 \Delta |\nabla \varphi|_g^2
+
g(\nabla \varphi,\nabla \Delta \varphi)
+
\tr \big[(\nabla^2\varphi)^2\big]
+
\Ric(\nabla \varphi,\nabla \varphi)
=
0,\tag{2.20}
\end{align*}
where all terms are evaluated at $x$.

This identity naturally motivates the definition of the Bakry--Émery operator
\begin{align*}
\label{eq:2.21}
\Gamma_2(\varphi)
:=
\frac12 \Delta |\nabla \varphi|^2
-
g(\nabla \varphi,\nabla \Delta \varphi),
\tag{2.21}
\end{align*}
which, by the Bochner formula \eqref{eq:2.20}, admits the decomposition
\begin{align*}\label{eq:2.22}
\Gamma_2(\varphi)
=
\tr \bigl[(\nabla^2\varphi)^2\bigr]
+
\Ric(\nabla\varphi,\nabla\varphi). \tag{2.22}    
\end{align*}
Since the Hessian term is nonnegative, lower bounds on the Ricci curvature
translate directly into lower bounds for $\Gamma_2$. This observation forms the
foundation of Bakry--Émery theory, providing an analytic framework for studying
lower Ricci curvature bounds through diffusion operators. It has led to numerous
applications in geometric analysis, including gradient estimates, heat kernel
bounds, functional inequalities, and, more recently, optimal transport and
synthetic notions of Ricci curvature.

\medskip

Our goal is to develop an analogous framework for intermediate Ricci curvature.
Since intermediate Ricci curvature is defined by taking traces of the Riemann
curvature tensor over $k$-dimensional subspaces, it is more natural to work with
a tensorial version of the Bochner formula before taking traces. This viewpoint
allows the full Riemann curvature tensor, rather than only its trace, to appear
explicitly.



\begin{fprop}
Let $(M, g)$ be a smooth Riemannian manifold. Then, for every smooth function $\varphi:M\to\mathbb R$ and every $x\in M$, we have
\begin{align*}
\label{eq:2.23}
-\frac12 \nabla^2 |\nabla\varphi|_g^2
+
\nabla_{\nabla\varphi}\nabla^2\varphi
+
(\nabla^2\varphi)^2
+
\Riem(\bullet,\nabla\varphi)\nabla\varphi
=
0,    \tag{2.23}
\end{align*}
where all terms are evaluated at $x$. More precisely, the curvature term
is the symmetric $(0,2)$-tensor defined by
\[
\bigl(\Riem(\bullet,\nabla\varphi)\nabla\varphi\bigr)(X,Y)
:=
g\left(
\Riem(X,\nabla\varphi)\nabla\varphi,
Y
\right).
\]
\end{fprop}

\begin{proof}
For $x \in X$, we find a local normal coordinates $\{x^1, \cdots, x^n\}$ at $x$, that is,
\begin{align*}
    g_{ij}(x) \coloneqq g\bigl(\partial_i, \partial_j \bigr), \qquad \partial_k g_{ij} (x) = 0
\end{align*}for all $i ,j ,k \in \{1, \cdots, n\}$, where $\partial_i \coloneqq \frac{\partial}{\partial x^i}$.

\medskip

For convenience, we write
\[
\varphi_i=\nabla_i\varphi,
\qquad
\varphi_{ij}=\nabla_i\nabla_j\varphi,
\qquad
\varphi_{ij;k}=\nabla_k\nabla_i\nabla_j\varphi.
\]
First, we have
\[
|\nabla\varphi|_g^2
=
g^{kl}\varphi_k\varphi_l,
\]which implies that at the point $x$, we obtain 
\begin{align*}
\label{eq:2.24}
\frac12 \nabla_i\nabla_j |\nabla\varphi|^2
&=
\nabla_i\nabla_j
\Bigl(
\frac12 g^{kl}\varphi_k\varphi_l
\Bigr) =
\varphi_{ik}\varphi_{jk}
+
\varphi_k \varphi_{jk;i}. \tag{2.24}
\end{align*}
Second, we have
\begin{align*}
\label{eq:2.25}
    \bigl((\nabla^2\varphi)^2\bigr)_{ij} =
\varphi_{ik}\varphi_{kj}. \tag{2.25}
\end{align*}
Third, we get
\begin{align*}
\label{eq:2.26}
\bigl(\nabla_{\nabla\varphi}\nabla^2\varphi\bigr)_{ij}
=
\varphi_k \varphi_{ij;k}. \tag{2.26}
\end{align*}
By combining equations \eqref{eq:2.24}, \eqref{eq:2.25}, and \eqref{eq:2.26}, we obtain 
\begin{align*}
\label{eq:2.27}
    -\frac12 \nabla_i\nabla_j |\nabla\varphi|^2 + \bigl((\nabla^2\varphi)^2\bigr)_{ij} + \bigl(\nabla_{\nabla\varphi}\nabla^2\varphi\bigr)_{ij} = - \varphi_k \varphi_{jk;i} + \varphi_k \varphi_{ij;k}. \tag{2.27}
\end{align*}
Using local normal coordinates at $x$, we obtain
\begin{align*}
\label{eq:2.28}
&\kern-2em \varphi_{jk;i}-\varphi_{ij;k} \\
&= \bigl ( \partial_i \partial_j \partial_k \varphi - \partial_i \Gamma^l_{jk} \varphi_l \bigr) -
\bigl ( \partial_k \partial_i \partial_j \varphi - \partial_k \Gamma^l_{ij} \varphi_l \bigr)
=
\bigl( \partial_k \Gamma^l_{ij} - \partial_i \Gamma^l_{jk} \bigr) \varphi_l \\
&= \tensor{R}{_{kij}^l} \varphi_l = 
R_{kijl} \varphi_l. \tag{2.28}
\end{align*}
Therefore, by equation \eqref{eq:2.28}, we have
\begin{align*}
\label{eq:2.29}
-\varphi_k\varphi_{ki;j}
+
\varphi_k\varphi_{ij;k}
=
-
\varphi_k\varphi_l R_{kijl} = -\Riem(\bullet, \nabla\varphi)  \nabla\varphi
\bigr)_{ij}. \tag{2.29}
\end{align*}
Thus, by combining equations \eqref{eq:2.27} and \eqref{eq:2.29}, we get
\[
-\frac12 \nabla_i\nabla_j |\nabla\varphi|^2
+
\bigl(\nabla_{\nabla\varphi}\nabla^2\varphi\bigr)_{ij}
+
\bigl((\nabla^2\varphi)^2\bigr)_{ij}
+
\bigl(
\Riem(\bullet, \nabla\varphi)\nabla\varphi
\bigr)_{ij}
=0.
\]
Since this holds for every pair of indices $i,j \in \{1, \cdots, n\}$ and at any arbitrary point
$x\in M$, we obtain
\[
-\frac12 \nabla^2 |\nabla\varphi|^2
+
\nabla_{\nabla\varphi}\nabla^2\varphi
+
(\nabla^2\varphi)^2
+
\Riem(\bullet, \nabla\varphi)\nabla\varphi
=
0.
\]This finishes the proof. 
\end{proof}

Motivated by the classical Bakry--Émery $\Gamma_2$-operator, we introduce
its tensorial analogue
\begin{align*}
\label{eq:2.30}
\widetilde{\Gamma}_2(\varphi)
\coloneqq
\frac12 \nabla^2 |\nabla \varphi|^2
-
\nabla_{\nabla \varphi}\nabla^2\varphi, \tag{2.30}
\end{align*}
the tensorial weighted Bakry--Émery operator:
\begin{align*}
\label{eq:2.31}
    \widetilde{\Gamma}_{2; k, f} (\varphi) &\coloneqq \frac12 \nabla^2 |\nabla \varphi|^2
-
\nabla_{\nabla \varphi}\nabla^2\varphi 
+ \frac{1}{k}   \Hess f( \nabla \varphi, \nabla \varphi)   \Id, \tag{2.31}
\end{align*}and the tensorial weighted $N$-Bakry--Émery operator:
\begin{align*}
\label{eq:2.32}
\kern-2em \widetilde{\Gamma}_{2; k, f}^N (\varphi) \coloneqq \frac12 \nabla^2 |\nabla \varphi|^2
-
\nabla_{\nabla \varphi}\nabla^2\varphi + \frac{1}{k} \Bigl(  \Hess f( \nabla \varphi, \nabla \varphi)  - \frac{df(\nabla \varphi)^2}{N- k}  \Bigr) \Id. \tag{2.32}
\end{align*}

Unlike the classical Bakry--Émery operator, which is a scalar quantity obtained
by taking the trace of the Bochner identity, the above operators are
tensor-valued. Their definitions involve only first-order and second-order
differential operators and do not explicitly involve the curvature tensor.
The tensorial Bochner formula then shows that the Riemann curvature tensor
appears naturally as the obstruction to the commutation of these differential operators.
Consequently, taking traces of these tensorial operators over suitable
subspaces produces quantities governed by intermediate Ricci curvature, making
them the natural analytic objects for our study.

\begin{flemma}
    Let $(M, g)$ be a smooth Riemannian manifold. Then, for every smooth function $\varphi:M\to\mathbb R$ and every $x\in M$, we have
    \begin{align*}\label{eq:2.33}
\kern-2.5em\widetilde{\Gamma}_2(\varphi)
&=(\nabla^2\varphi)^2
+
\operatorname{Riem}(\bullet, \nabla \varphi)\nabla \varphi, \tag{2.33} \\
\label{eq:2.34}
\kern-2.5em\widetilde{\Gamma}_{2; k, f} (\varphi) &=(\nabla^2\varphi)^2
+
\operatorname{Riem}(\bullet, \nabla \varphi)\nabla \varphi 
+ \frac{1}{k}   \Hess f( \nabla \varphi, \nabla \varphi)   \Id, \tag{2.34}  \\
\label{eq:2.35}
\kern-2.5em\widetilde{\Gamma}_{2; k, f}^N (\varphi) 
&=(\nabla^2\varphi)^2
+
\operatorname{Riem}(\bullet, \nabla \varphi)\nabla \varphi 
+ \frac{1}{k} \Bigl(  \Hess f( \nabla \varphi, \nabla \varphi)  - \frac{df(\nabla \varphi)^2}{N- k}  \Bigr) \Id. \tag{2.35}
    \end{align*}
\end{flemma}

\begin{proof}
    By Proposition~\hyperlink{P:2.1}{2.1}, for any smooth function $\varphi:M\to\mathbb R$ and for every $x\in M$, we have
    \begin{align*}
\frac12 \nabla^2 |\nabla\varphi|_g^2
-
\nabla_{\nabla\varphi}\nabla^2\varphi
=
(\nabla^2\varphi)^2
+
\Riem(\bullet,\nabla\varphi)\nabla\varphi.
\end{align*}The identities \eqref{eq:2.33}, \eqref{eq:2.34}, and \eqref{eq:2.35} then follow immediately from the definitions of
$\widetilde{\Gamma}_2$,
$\widetilde{\Gamma}_{2;k,f}$,
and
$\widetilde{\Gamma}_{2;k,f}^{N}$. This finishes the proof. 
\end{proof}

This lemma shows that the tensorial Bakry--Émery operators encode the full Riemann curvature tensor, whereas the classical $\Gamma_2$-operator only detects its trace through the Ricci curvature. This distinction is precisely what makes the tensorial formulation suitable for studying intermediate Ricci curvature.

\subsection{Preliminaries on ODE inequalities}
\label{sec:2.2}
The comparison estimates developed later in this paper ultimately reduce to
one-dimensional differential inequalities along geodesics. These inequalities arise naturally from the
evolution of Jacobi fields, Riccati-type equations, and entropy functionals
along Wasserstein geodesics. In this section, we collect several elementary
ODE inequalities that will be used repeatedly throughout the paper.

\begin{flemma}
Let $T > 0$, $h:[0,T]\to\mathbb R_{\geq 0}$ be a function such that $h^2$ is an absolute continuous function on $[0, T]$.
Assume that $a,b\in L^1([0,T])$ with $a\geq 0$ and $b \geq 0$, and
\[
    (h^2)'(t)\leq a(t)h(t) - b(t)
    \qquad \text{for a.e. }t\in[0,T].
\]
Then, for every $0\leq s\leq t\leq T$, we have
\begin{align*}
\label{eq:2.36}
    h^2(t)
    \leq
        \Biggl[ \sqrt{h^2(s)  + \int_0^s b(r) \, dr      } + \frac{1}{2} \int_s^t a(r)\,dr \Biggr]^2
    -
    \int_0^t b(r) \,dr.    \tag{2.36}
\end{align*}
In particular, by letting $s = 0$, we have 
\begin{align*}
\label{eq:2.37}
h^2(t)
    \leq
        \Bigl( h(0)+ \frac{1}{2} \int_0^t a(s)\,d s \Bigr)^2
    -
    \int_0^t b(r) \,dr.   \tag{2.37}
    \end{align*}
\end{flemma}

\begin{proof}
First, for any $\epsilon > 0$, let 
\begin{align*}
    B(t) \coloneqq \int_0^t b(r) \, dr \quad  \text{ and } \quad H_\epsilon(t) \coloneqq h^2(t) + B(t) + \epsilon.
\end{align*}Then, by the hypothesis, we obtain 
\begin{align*}
    H'_{\epsilon}(t) = (h^2)'(t) + b(t) \leq a(t) h(t) \leq a(t) H_\epsilon^{1/2}(t),
\end{align*}which implies that
\begin{align*}
    \sqrt{H_\epsilon(t)} \leq  \sqrt{H_\epsilon(s)} + \frac{1}{2} \int_s^t a(r) dr
\end{align*}for any $\epsilon > 0$ and for every $0 \leq s \leq t \leq T$. By letting $\epsilon \rightarrow 0$, we obtain 
\begin{align*}
    h^2(t) \leq        \Biggl[ \sqrt{h^2(s)  + \int_0^s b(r) \, dr      } + \frac{1}{2} \int_s^t a(r)\,dr \Biggr]^2
    -
    \int_0^t b(r) \,dr. 
\end{align*}This finishes the proof. 
\end{proof}

\hypertarget{L:2.3}{\begin{flemma}
Let \(h\in C^2([0,1])\) satisfy
$h'' \geq a (h')^2 + b$ 
on \([0,1]\), where \(a,b\in\mathbb R\) with $a \geq 0$.
\begin{enumerate}
\item[(i)] If \(a=0\), then for every \(t\in[0,1]\),
\begin{align*}
\label{eq:2.38}
h(t)
\leq
(1-t)h(0)
+
t h(1)
-
\frac{b}{2}\,t(1-t).
\tag{2.38}
\end{align*}
\item[(ii)] If $a > 0$ and $\pi^2 > ab$, we define
\begin{align*}
\label{eq:2.39}
        \sigma_{ab}^{(t)} = \begin{cases}
            \frac{\sin(t\sqrt{ab})}{\sin(\sqrt{ab})}, &\text{ if }  ab > 0, \\
            t, &\text{ if } ab = 0, \\
            \frac{\sinh(t\sqrt{-ab})}{\sinh(\sqrt{-ab})}, &\text{ if } ab < 0,     
        \end{cases} \tag{2.39}
\end{align*}
then for every \(t\in[0,1]\),
\begin{align*}
\label{eq:2.40}
h(t)
\leq
-\frac1a
\log\bigl(
\sigma^{(1-t)}_{ab}
e^{-a h(0)}
+
\sigma^{(t)}_{ab}
e^{-a h(1)}
\bigr).
\tag{2.40}    
\end{align*}
\end{enumerate}
\end{flemma}}

\begin{proof}
First, consider the case $a = 0$, then by defining 
\begin{align*}
    \tilde{h}(t) = h(t) - \frac{b}{2} t^2,
\end{align*}we observe that $\tilde{h}$ is convex. Thus, we obtain
\begin{align*}
    \tilde{h}(t) \leq (1-t) \tilde{h}(0) + t \tilde{h}(1) = (1-t) h(0) + t h(1) - \frac{bt}{2},
\end{align*}that is,
\begin{align*}
    h(t) \leq (1-t) h(0) + t h(1) - \frac{b}{2}t(1-t).
\end{align*}

Second, consider the case \(a > 0\), and define
$u(t):=e^{-a h(t)}$.
A direct computation gives
$u'
=
-a h' e^{-a h}$
and
$u''
=
-a e^{-a h}\bigl(h''-a(h')^2\bigr)$.
Using the assumption that 
\(
h''\geq a(h')^2+b
\),
we obtain
\[
u''+ab\,u\leq 0.
\]
Let \(w\) be the unique solution of
$w''+ ab  w=0$ with boundary values
\[
w(0)=u(0),
\qquad
w(1)=u(1).
\]
Then the solution is given by
\[
w(t)
=
\sigma_{ab}^{(1-t)} u(0)
+
\sigma_{ab}^{(t)} u(1).
\]

By setting $z \coloneqq u-w$, 
then we get
$z''+ab\,z\leq 0$
and $z(0)=z(1)=0$.
By the one-dimensional maximum principle, we get 
\[
z\geq 0
\qquad\text{on }[0,1].
\]
Hence, we have
\[
u(t)\geq w(t)
=
\sigma_{ab}^{(1-t)} u(0)
+
\sigma_{ab}^{(t)} u(1).
\]
By substituting \(u(t)=e^{-ah(t)}\) back, we obtain 
\[
e^{-a h(t)}
\geq
\sigma_{ab}^{(1-t)}
e^{-a h(0)}
+
\sigma_{ab}^{(t)}
e^{-a h(1)}.
\]
This finishes the proof.
\end{proof}

\hypertarget{P:2.3}{
\begin{fprop}
Let $a>0$, $b\in\mathbb{R}$, and assume $ab<\pi^2$. Suppose that
\[
h(t)
\leq
-\frac1a
\log\Bigl(
\sigma_{ab}^{(1-t)}
e^{-ah(0)}
+
\sigma_{ab}^{(t)}
e^{-ah(1)}
\Bigr)
\]
for every $t\in[0,1]$, where $\sigma^{(t)}_{ab}$ is defined in \eqref{eq:2.39}.
Then the following hold.
\begin{enumerate}
    \item[(1)] For every $t\in[0,1]$,
    \begin{align*}
    \label{eq:2.41}
        h(t)
    \leq
    (1-t)h(0)
    +
    th(1)
    +
    A_t, 
    \tag{2.41}
    \end{align*}
    where $A_t$ is define by the following 
    \begin{align*}
    \label{eq:2.42}
        A_t
    \coloneqq 
    \begin{cases}
    -\frac1a
    \log \biggl[
    \Bigl(
    \frac{\sigma_{ab}^{(1-t)}}{1-t}
    \Bigr)^{1-t}
    \Bigl(
    \frac{\sigma_{ab}^{(t)}}{t}
    \Bigr)^t
    \biggr], & t \in (0, 1), \\
    0, & t = 0, 1.
    \end{cases}
    \tag{2.42}
    \end{align*}
    \item[(2)] In particular, for every $t\in[0,1]$,
    \begin{align*}
    \label{eq:2.43}
     \kern-2em h(t)
    \le
    \begin{cases}
    (1-t)h(0)
    +th(1)
    -\dfrac{b}{2}t(1-t),
    & b>0,\\[3mm]
    (1-t)h(0)
    +th(1),
    & b=0,\\[3mm]
    (1-t)h(0)
    +th(1)
    +\dfrac3a t(1-t)
    \log \biggl(
    \dfrac{\sinh\sqrt{-ab}}{\sqrt{-ab}}
    \biggr),
    & b<0.
    \end{cases} \tag{2.43}
    \end{align*}
\end{enumerate}
\end{fprop}}

\begin{proof}
We first prove $(1)$. Fix \(t\in(0,1)\). By the weighted arithmetic--geometric mean inequality, for all $X, Y > 0$, we have
\[
X+Y
\geq
\biggl(\frac{X}{1-t}\biggr)^{1-t}
\biggl(\frac{Y}{t}\biggr)^t.
\]
By applying this with
$X
=
\sigma_{ab}^{(1-t)}
e^{-ah(0)}$
and $Y
=
\sigma_{ab}^{(t)}
e^{-ah(1)}$,
we obtain
\begin{align*}
\label{eq:2.44}
h(t)
&\leq
-\frac1a
\log\Bigl(
\sigma_{ab}^{(1-t)}
e^{-ah(0)}
+
\sigma_{ab}^{(t)}
e^{-ah(1)}
\Bigr)\\
&\leq
-\frac1a
\log\biggl[
\biggl(
\frac{\sigma_{ab}^{(1-t)}}{1-t}
\biggr)^{1-t}
\biggl(
\frac{\sigma_{ab}^{(t)}}{t}
\biggr)^t
e^{-a\bigl((1-t)h(0)+th(1)\bigr)}
\biggr]\\
&=
(1-t)h(0)
+
th(1)
+
A_t. \tag{2.44}
\end{align*}
One can check directly that $\lim_{t \rightarrow 0} A_t = 0$ and $\lim_{t \rightarrow 1} A_t = 0$, so
the endpoint cases \(t=0\) and \(t=1\) also follow by continuity. This proves $(1)$.

\medskip

To prove $(2)$, if \(b=0\), then
$\sigma_{ab}^{(s)}
=
s$,
and therefore
$A_t
=
0$.
Hence, we obtain 
\[
h(t)
\leq
(1-t)h(0)
+
th(1).
\]

Suppose next that \(b<0\), and let
$r
=
\sqrt{-ab}$.
For \(s\in(0,1]\), define
\[
F(s)
\coloneqq
\frac{\sigma_{ab}^{(s)}}{s}
=
\frac{\sinh(sr)}{s\sinh r}.
\]
Then, we have
\[
A_t
=
-\frac1a
\bigl[
(1-t)\log F(1-t)
+
t\log F(t)
\bigr].
\]

We claim that
\[
\log\biggl(
\frac{\sinh(sr)}{sr}
\biggr)
\geq
s^2
\log\biggl(
\frac{\sinh r}{r}
\biggr) \qquad
\text{for }s\in(0,1].
\]
This follows from the fact that the function
\[
s\longmapsto
\frac{1}{s^2}
\log\biggl(
\frac{\sinh(sr)}{sr}
\biggr)
\]
is nonincreasing on \((0,1]\) and the limit at $s = 0$ is $\frac{r^2}{6}$.
Therefore, we obtain 
\begin{align*}
A_t
&\leq
\frac{3}{a}t(1-t)
\log\biggl(
\frac{\sinh r}{r}
\biggr) = \frac{3}{a}t(1-t)
\log\biggl(
\frac{\sinh\sqrt{-ab}}{\sqrt{-ab}}
\biggr).
\end{align*}

Finally, suppose that \(b>0\), and let
$c = \sqrt{ab}$.
We claim that
\[
\log
\biggl(
\frac{\sin(sc)}{s\sin c}
\biggr)
\geq
\frac{c^2}{6}
\bigl(1-s^2\bigr)
\qquad
\text{for all } s\in(0,1].
\]
Hence, we obtain
\begin{align*}
-aA_t
&=
(1-t)
\log
\biggl(
\frac{\sigma_{ab}^{(1-t)}}{1-t}
\biggr)
+
t
\log
\biggl(
\frac{\sigma_{ab}^{(t)}}{t}
\biggr) \geq
\frac{c^2}{2}
t(1-t).
\end{align*}
Since \(c^2=ab\), we conclude that
\[
A_t
\leq
-\frac{b}{2}
t(1-t).
\]
Substituting these estimates for \(A_t\) into \eqref{eq:2.44} yields \eqref{eq:2.43}. This finishes the proof. 
\end{proof}

\subsection{Exterior powers of self-adjoint operators}
\label{sec:2.3}
In this section, we establish two basic linear algebraic results concerning exterior powers of self-adjoint operators. The first provides an inequality relating the elementary symmetric polynomials of the matrix exponential to the trace of the underlying operator, while the second characterizes lower bounds on the induced operator acting on exterior powers. These results will serve as the main linear algebraic tools in the subsequent sections.

\begin{flemma}
Let $A$ be a symmetric operator on an $n$-dimensional inner product space
$(V,\langle \cdot,\cdot\rangle)$.
Then for any orthonormal basis $\{e_i\}_{i=1}^n$ of $V$ and any
$k\in\{1,\dots,n\}$, we have
\begin{align*}
\label{eq:2.45}
\sigma_k(e^{-A}) \geq \sum_{|I|=k}
\exp \Bigl(
-
\sum_{i\in I}\langle e_i,Ae_i\rangle
\Bigr) \geq \binom{n}{k}  \exp \Bigl(  -\frac{k}{n} \tr A      \Bigr), \tag{2.45}
\end{align*}
where $\exp({-A})$ denotes the matrix exponential of $-A$, and
$\sigma_k$ denotes the $k$-th elementary symmetric polynomial of the
eigenvalues.
\end{flemma}

\begin{proof}
Fix an orthonormal basis $\{e_i\}_{i=1}^n$ of $V$, and let
$d_1,\dots,d_n$ be the diagonal entries of $A$ with respect to this basis.
Then, we have
\[
d_i=\langle e_i,Ae_i\rangle,
\]
and therefore
\[
\sum_{|I|=k}
\exp\left(
-
\sum_{i\in I}\langle e_i,Ae_i\rangle
\right)
=
\sum_{|I|=k}
e^{-\sum_{i\in I} d_i}.
\]

Let $\lambda_1,\dots,\lambda_n$ be the eigenvalues of $A$.
By the Schur--Horn theorem, the vector
$(d_1,\dots,d_n)$ lies in the permutation polytope generated by
$(\lambda_1,\dots,\lambda_n)$.

Now, we consider the function
\[
F(x_1,\dots,x_n)
:=
\sum_{|I|=k}
e^{-\sum_{i\in I}x_i}.
\]
It is straightforward to check $F$ is convex and invariant under permutations of the variables,
hence, $F$ is Schur-convex. Thus, we obtain
\[
F(\lambda_1,\dots,\lambda_n) \geq F(d_1,\dots,d_n),
\]which implies that
\[
\sigma_k(e^{-A}) = \sum_{|I|=k}
e^{-\sum_{i\in I}\lambda_i} \geq \sum_{|I|=k} 
e^{-\sum_{i\in I} d_i}.
\]

For the lower bound, by Jensen's inequality, we obtain 
\begin{align*}
    \frac{1}{\binom{n}{k}} \sum_{|I| = k} \exp \Bigl( - \sum_{i \in I} d_i \Bigr) \geq \exp \Bigl(    - \frac{1}{\binom{n}{k}} \sum_{|I| = k} \sum_{i \in I} d_i \Bigr) = \exp\Bigl( - \frac{k}{n} \tr A \Bigr).
\end{align*}

This finishes the proof.
\end{proof}

\hypertarget{D:2.4}{
\begin{fdefi}
Let $(V,\langle\cdot,\cdot\rangle)$ be an $n$-dimensional inner product space, and let
$A \colon V\to V$
be a self-adjoint operator. We have the following induced operator
\[
A^{[k]} \colon \bigwedge^k V\to \bigwedge^k V,
\]
which is defined by
\[
A^{[k]}(v_1\wedge\cdots\wedge v_k)
:=
\sum_{j=1}^k
v_1\wedge\cdots\wedge Av_j\wedge\cdots\wedge v_k,
\]
for all $v_1,\dots,v_k\in V$.
\end{fdefi}}

\begin{flemma}
    Let $(V,\langle\cdot,\cdot\rangle)$ be an $n$-dimensional inner product space, and let
$A \colon V\to V$
be a self-adjoint operator. Then, for any $1 \leq k \leq n$, we have
    \begin{align*}
    \label{eq:2.46}
        \tr_{\bigwedge^k V} \bigl(  A^{[k]}   \bigr) =  {\binom{n-1}{k-1}} \tr (A). \tag{2.46}
    \end{align*}
\end{flemma}

\begin{proof}
    First, note that if $\lambda_1,\dots,\lambda_n$ are the eigenvalues of $A \colon V \rightarrow V$, then the eigenvalues
of $A^{[k]}$ are
\[
\lambda_{i_1}+\cdots+\lambda_{i_k},
\qquad
1\le i_1<\cdots<i_k\le n.
\]Hence, we have
\begin{align*}
\tr_{\bigwedge^k V} \bigl(  A^{[k]}   \bigr) &= \sum_{1\le i_1<\cdots<i_k\le n} ( \lambda_{i_1}+\cdots+\lambda_{i_k} ) \\
&= \binom{n-1}{k-1} ( \lambda_1 + \cdots + \lambda_n ) = \binom{n-1}{k-1} \tr (A).
\end{align*} 
This finishes the proof. 
\end{proof}






\begin{fdefi}[Decomposable unit $k$-vector]
Let $(V,\langle\cdot,\cdot\rangle)$ be an inner product space. A vector
$\xi\in\bigwedge^kV$ is called a \emph{decomposable unit $k$-vector} if there
exist orthonormal vectors $e_1,\dots,e_k\in V$ such that
\[
\xi=e_1\wedge\cdots\wedge e_k.
\]
Equivalently, $\xi$ is decomposable and satisfies $|\xi|=1$.
\end{fdefi}

\hypertarget{L:2.6}{
\begin{flemma}
Let $(V,\langle\cdot,\cdot\rangle)$ be an $n$-dimensional inner product space,
and let $A:V\to V$ be a self-adjoint operator. Fix
$k\in\{1,\dots,n\}$ and let $K\in\mathbb R$. Then the following are equivalent:
\begin{enumerate}
    \item For every $k$-dimensional subspace $\Sigma\subset V$,
    $\tr(A|_\Sigma)\ge K$.
    \item For every decomposable unit $k$-vector
    $\xi\in\bigwedge^k V$,
    $\langle A^{[k]}\xi,\xi\rangle\ge K$.
    \item The induced operator $A^{[k]}$ satisfies
    $A^{[k]}\geq K \cdot \Id$
    on $\bigwedge^k V$.
\end{enumerate}
\end{flemma}}

\begin{proof}
First, let
$\lambda_1\leq   \cdots \leq \lambda_n$
be the eigenvalues of $A$, and let $\{e_1,\dots,e_n\}$ be an orthonormal
eigenbasis. By the min-max principle, the minimum of $\tr(A|_\Sigma)$ over all
$k$-dimensional subspaces $\Sigma\subset V$ is
\[
\lambda_1+\cdots+\lambda_k.
\]
Hence, condition $(1)$ is equivalent to
\[
\lambda_1+\cdots+\lambda_k\ge K.
\]

We next compare condition $(1)$ with condition $(2)$. Let
$\Sigma\subset V$ be a $k$-dimensional subspace, and let
$\{u_1,\dots,u_k\}$ be an orthonormal basis of $\Sigma$. We set
$\xi \coloneqq u_1\wedge\cdots\wedge u_k$,
then $\xi$ is a decomposable unit $k$-vector. By Definition~\hyperlink{D:2.4}{2.4}, we have
\[
A^{[k]}(u_1\wedge\cdots\wedge u_k)
=
\sum_{j=1}^k
u_1\wedge\cdots\wedge Au_j\wedge\cdots\wedge u_k.
\]
By taking the inner product of $A^{[k]} \xi$ with $\xi$, we obtain
\[
\bigl\langle A^{[k]}\xi,\xi\bigr\rangle
=
\sum_{j=1}^k \langle Au_j,u_j\rangle
=
\tr(A|_\Sigma).
\]
Therefore, condition $(1)$ is equivalent to condition $(2)$.

\medskip

It remains to compare condition $(2)$ with condition $(3)$. We consider the orthonormal basis
\[
e_{i_1}\wedge\cdots\wedge e_{i_k},
\qquad
1\leq i_1<\cdots<i_k\leq n,
\]
of $\bigwedge^k V$, then the operator $A^{[k]}$ becomes diagonal. Indeed, we obtain
\[
A^{[k]}(e_{i_1}\wedge\cdots\wedge e_{i_k})
=
(\lambda_{i_1}+\cdots+\lambda_{i_k})
e_{i_1}\wedge\cdots\wedge e_{i_k}.
\]
Thus,
$A^{[k]}\ge K\Id$
if and only if
$\lambda_{i_1}+\cdots+\lambda_{i_k}\ge K$
for every $1\leq i_1<\cdots<i_k\leq n$. Since
$\lambda_1+\cdots+\lambda_k$
is the smallest one, $A^{[k]}\ge K\Id$ is equivalent to
\[
\lambda_1+\cdots+\lambda_k\ge K.
\]
Hence, condition $(3)$ is also equivalent to condition $(1)$.
Therefore, all three conditions are equivalent. This finishes the proof. 
\end{proof}

\begin{fdefi}
Let $(V,\langle\cdot,\cdot\rangle)$ be an $n$-dimensional inner product space, and let
$A \colon V\to V$
be a self-adjoint operator, and let $k\in\{1,\dots,n\}$, we define
\[
\operatorname{tr}_{\bigwedge^k V}
\left(
\sinh^2\bigl(A^{[k]}\bigr)
\right)
\]
to be the trace of the operator
\[
\sinh^2\bigl(A^{[k]}\bigr)
=
\frac{1}{4}
\left(
e^{A^{[k]}}
-
e^{-A^{[k]}}
\right)^2
\]
acting on $\bigwedge^k V$.
Equivalently, if $\lambda_1,\dots,\lambda_n$ are the eigenvalues of $A$, then
\[
\operatorname{tr}_{\bigwedge^k V}
\left(
\sinh^2\bigl(A^{[k]}\bigr)
\right)
=
\sum_{1\le i_1<\cdots<i_k\le n}
\sinh^2\left(
\lambda_{i_1}+\cdots+\lambda_{i_k}
\right).
\]
\end{fdefi}

\subsection{Preliminaries on optimal transport}
\label{sec:2.4}
In this section, we review the basic notions of optimal transport on Riemannian manifolds that will be needed throughout the paper.
In Section~\ref{sec:2.4.1}, we recall the classical McCann's theorem on the existence and properties of optimal transport maps for the quadratic cost on Riemannian manifolds.
Building on this result, we introduce Wasserstein geodesics and our entropy tensor in Section~\ref{sec:2.4.2}, and explain their relationship with the classical entropy functional.

\medskip

Throughout this section, $(M^n, g)$ denotes a complete Riemannian manifold, and $d(\cdot, \cdot)$ denotes the associated geodesic distance.
Unless otherwise specified, ``a.e." stands for ``almost everywhere with respect to the Riemannian volume measure on $M$".
Given $\mu_0, \mu_1\in \cPac(M)$, the optimal transport considers the minimization problem:
\begin{align*}
\label{eq:2.47}
    \cW_2^2(\mu_0, \mu_1) \coloneqq \inf_\pi \int_{M\times M} d(x,y)^2 d\pi(x,y), \tag{2.47}
\end{align*}where the infimum is taken over all couplings $\pi$ of $\mu_0$ and $\mu_1$, that is, probability measures on $M \times M$ satisfying
\begin{align*}
    (P_1)_\sharp \pi=\mu_0, \quad (P_2)_\sharp \pi=\mu_1,
\end{align*}where $P_i \colon M \times M \rightarrow M$ denotes the projection onto the $i$-th factor for $i \in \{1, 2\}$.
By the Kantorovich existence theorem, the infimum in \eqref{eq:2.47} is attained by an \emph{optimal coupling}, and the quantity $\cW_2$ defines a metric on $\cPac(M)$, called the Wasserstein $2$-distance (see, for example, \cite{villani2009}).

\subsubsection{McCann's theorem}
\label{sec:2.4.1}
One of the fundamental results in optimal transport on Riemannian manifolds is McCann's theorem \cite{mccann_polar_2001}, which characterizes optimal couplings for the quadratic cost \eqref{eq:2.47} in terms of optimal transport maps induced by $\frac{d^2}{2}$-concave functions. To state the theorem, we first recall the following definition. 
\begin{fdefi}
    We say that a function $\varphi \colon M \to \mathbb R \cup \{-\infty\}$
is $\frac{d^2}{2}$-concave if it is not identically $-\infty$ and there
exists a function
\[
\psi \colon M \to \mathbb R \cup \{\pm \infty\}
\]
such that, for every $x\in M$,
\begin{align*}
\label{eq:2.48}
\varphi(x)
=
\inf_{y\in M}
\left\{
\frac12 d^2(x,y)-\psi(y)
\right\}.
\tag{2.48}
\end{align*}
\end{fdefi}

We are now ready to state McCann's theorem, which provides a characterization of optimal transport maps on Riemannian manifolds in terms of $\frac{d^2}{2}$-concave functions.

\begin{fthm}[McCann's theorem, {\cite[Theorems 9]{mccann_polar_2001}}]

Let $\mu_0,\mu_1\in\mathcal P_c^{\mathrm{ac}}(M)$. Then there exists a
$\frac{d^2}{2}$-concave function $-\theta$ such that the map
\[
F:M\to M,
\qquad
F(x)=\exp_x\bigl(\nabla\theta(x)\bigr),
\]
defined for $\mu_0$-almost every $x\in M$, satisfies
\[
F_\sharp\mu_0=\mu_1,
\]
and
\begin{align*}
\label{eq:2.49}
\int_M
\frac12 d^2\bigl(x,F(x)\bigr)\,d\mu_0(x)
=
\frac12\cW_2^2(\mu_0,\mu_1).
\tag{2.49}
\end{align*}
\end{fthm}

\begin{frmk}
By \cite[Theorem 8]{mccann_polar_2001}, the map $F$ is uniquely determined $\mu_0$-almost everywhere by
the measures $\mu_0$ and $\mu_1$. We call $F$ the
\emph{optimal transport map} from $\mu_0$ to $\mu_1$ and the function $-\theta$ a
\emph{Kantorovich potential} associated with $F$.
Such potentials are locally Lipschitz, and thus $F$ is well defined (see \cite[Lemma 4]{mccann_polar_2001}).
\end{frmk}

\subsubsection{Wasserstein geodesic and entropy tensor}
\label{sec:2.4.2}
In \cite{cordero-erausquin_riemannian_2001}, they introduce the family of maps
\begin{align*}
    F_t(x) \coloneqq \exp_x(t\nabla \theta(x)), \qquad t \in [0, 1],
\end{align*}
which induces the Wasserstein geodesic from $\mu_0$ to $\mu_1$ by $\mu_t \coloneqq (F_t)_\sharp \mu_0$; that is, for every $s,t\in [0,1]$, we have 
\begin{align*}
    \cW_2(\mu_s, \mu_t)=|s-t|\cW_2(\mu_0, \mu_1).
\end{align*}
This $\mu_t$ is the so-called the displacement interpolation.
Moreover, since $\theta$ is semiconvex, $\Hess_x\theta$ exists almost everywhere in the Alexandrov sense \@ \cite[Theorem 4.2]{cordero-erausquin_riemannian_2001}; namely,
\begin{align*}
\Hess_x\theta
\end{align*}
is a symmetric linear operator on $T_xM$ satisfying
\begin{align*}
\theta(\exp_x(\varepsilon u))
=
\theta(x)
+
\varepsilon\, g_x(\nabla\theta(x),u)
+
\frac{\varepsilon^2}{2}
g_x(u,\Hess_x\theta\,u)
+
o(\varepsilon^2).
\end{align*}
for every $u\in T_xM$ as $\varepsilon \rightarrow 0$.
As a consequence \cite[Theorem 4.2]{cordero-erausquin_riemannian_2001}, we have
\begin{align*}
F(x)\notin \operatorname{cut}(x),
\end{align*}for almost every $x \in M$,
where $\operatorname{cut}(x)\subseteq M$ denotes the cut locus of $x$, the set of points
in $M$ that cannot be joined to $x$ by a unique minimizing geodesic.

For later use, we introduce the following notation.
\begin{fdefi}
    Given $\mu_0, \mu_1\in \cPac(M)$, let $-\theta$ be a Kantorovich potential function such that $F(x) \coloneqq \exp_x(\nabla \theta(x))$ is the optimal transport map from $\mu_0$ to $\mu_1$.
    We define
    \begin{align*}
    \label{eq:2.50}
        D(\mu_0,\mu_1) \coloneqq \{x\in \supp \mu_0 \mid \, \Hess \theta(x) \text{ exists}\}. \tag{2.50}
    \end{align*}
\end{fdefi}

By the preceding discussion, the set $D(\mu_0, \mu_1)$ has full $\mu_0$-measure.
We now fix $x\in D(\mu_0, \mu_1)$ and let
\begin{align*}
    \gamma(t)=F_t(x)=\exp_x(t\nabla \theta(x)), \qquad t\in [0,1],
\end{align*}be the unique minimizing geodesic from $x$ to $F(x)$. Similar to the construction in Section~\ref{sec:2.1.2}, let 
$\tJ(t)$ denote the Jacobi field matrix along $\gamma(t)$ satisfying the initial conditions
\begin{align*}
\begin{cases}
    \tJ(0)&=\Id;\\
\dot{\tJ}(0)&=\Hess \theta(x).
\end{cases}
\end{align*}
Since $\gamma$ is minimizing, $\tJ(t)$ is invertible for every $t\in[0,1]$. We therefore define
\begin{align*}
    \tU_t (x) \coloneqq \dot{\tJ}(t) \tJ^{-1}(t),
\end{align*}
which can be proved to be symmetric (see e.g.\@ \cite[p. 368]{villani2009}).
Given $f\in C^\infty(M)$ and $k\in \{1, \cdots, n\}$, we further define the weighted endomorphism 
\begin{align*}
\label{eq:2.51}
    \tU_{t; k, f}(x) \coloneqq \tU_t(x) - \frac{1}{k} \langle \nabla f, \dot \gamma(s) \rangle \Id.
    \tag{2.51}
\end{align*}

The following lemma establishes Riccati-type inequalities for the
weighted endomorphism \(\tU_{t;k,f}\), which will serve as a key
ingredient in the subsequent analysis.

\begin{flemma}[Weighted Riccati inequalities]
Let
\(\mu_0,\mu_1\in\cPac(M)\),
let
\(x\in D(\mu_0,\mu_1)\),
and let
\[
\gamma(t)
=
\exp_x\bigl(t\nabla\theta(x)\bigr),
\qquad
t\in[0,1],
\]
be the associated minimizing geodesic.
For each \(t\in[0,1]\), let
\(\Sigma_t\subset T_{\gamma(t)}M\)
be the parallel transport of a \(k\)-dimensional subspace
\(\Sigma\subset T_xM\)
along \(\gamma\). Then the weighted endomorphism
\(\tU_{t;k,f}\)
satisfies
\begin{align*}
\label{eq:2.52}
\tr\bigl(\dot{\tU}_{t;k,f}\big|_{\Sigma}\bigr)
\le
-
\frac{1}{k}
\bigl(
\tr\bigl(\tU_{t;k,f}\big|_{\Sigma}\bigr)
+
\langle\nabla f,\dot\gamma(t)\rangle
\bigr)^2
-
\Ric_{k,f}(\Sigma_t,\dot\gamma(t)). 
\tag{2.52}
\end{align*}
Moreover, if \(N>k\), then we have
\begin{align*}
\label{eq:2.53}
\tr\bigl(\dot{\tU}_{t;k,f}\big|_{\Sigma}\bigr)
\le
-
\frac{1}{N}
\tr\bigl(\tU_{t;k,f}\big|_{\Sigma}\bigr)^2
-
\Ric_{k,f}^{N}(\Sigma_t,\dot\gamma(t)). \tag{2.53}
\end{align*}
\end{flemma}

\begin{proof}
First, similar to the derivation of \eqref{eq:2.18} and \eqref{eq:2.19}, we find that $\tU_s$ satisfies the matrix Riccati equation:
\begin{align*}
\label{eq:2.54}
\dot \tU_s + \tU_s^2 + \tR_s = 0.
\tag{2.54}
\end{align*}Consequently, $\tU_{s; k, f}$ satisfies the following weighted matrix Riccati equation:
\begin{align*}
\label{eq:2.55}
    \dot \tU_{s; k, f} 
    &= - \tU_{s; k, f}^2 -  \frac{2}{k} \langle \nabla f, \dot \gamma(s) \rangle \tU_{s; k, f}   - \tR_s \\
    &\kern2em - \frac{1}{k} \Bigl(  \nabla^2f(\dot \gamma, \dot \gamma) + \frac{1}{k} \langle \nabla f, \dot \gamma(s) \rangle^2 \Bigr) \Id. \tag{2.55}
\end{align*}

Let $\Sigma \subset T_x M$ be a $k$-dimensional subspace, then by taking the trace of \eqref{eq:2.55} over $\Sigma$, we obtain
\begin{align*}
\label{eq:2.56}
    &\kern-2em \tr \bigl( \dot \tU_{s; k, f}  |_\Sigma \bigr) \\
    &=  - \tr \bigl( \tU_s^2 |_\Sigma \bigr) - \Ric_k(\Sigma_s, \dot \gamma(s)) -   \nabla^2f(\dot \gamma, \dot \gamma) \tag{2.56} \\
\label{eq:2.57}
    &= - \tr \bigl( \tU_{s; k, f}^2 |_\Sigma   \bigr) -  \frac{2}{k} \langle \nabla f, \dot \gamma(s) \rangle \tr \bigl( \tU_{s; k, f} |_\Sigma \bigr)   -  \Ric_k(\Sigma_s, \dot \gamma(s)) \\
    &\kern2em - \Bigl(  \nabla^2f(\dot \gamma, \dot \gamma)  + \frac{1}{k} \langle \nabla f, \dot \gamma(s) \rangle^2 \Bigr), \tag{2.57}
\end{align*}where $\Sigma_s$ denotes the parallel transport of $\Sigma$ along $\gamma$.

Since $\tU_{s; k, f}$ is symmetric, the Cauchy--Schwarz inequality implies 
\begin{align*}
\tr\bigl(\tU_{s;k,f}^2\big|_\Sigma\bigr)
\ge
\frac{1}{k}
\tr\bigl(\tU_{s;k,f}\big|_\Sigma\bigr)^2.
\end{align*}
By substituting this estimate into \eqref{eq:2.57}, we obtain
\begin{align*}
\label{eq:2.58}
    &\kern-2em \tr \bigl( \dot \tU_{s; k, f} |_\Sigma \bigr) \\
    &\leq - \frac{1}{k} \tr \bigl( \tU_{s; k, f} |_\Sigma   \bigr)^2 -  \frac{2}{k} \langle \nabla f, \dot \gamma(s) \rangle \tr \bigl( \tU_{k, f} |_\Sigma \bigr)   -  \Ric_k(\Sigma_s, \dot \gamma(s)) \\
    &\kern2em - \Bigl(  \nabla^2f(\dot \gamma, \dot \gamma)  + \frac{1}{k} \langle \nabla f, \dot \gamma(s) \rangle^2 \Bigr) \\
    &= - \frac{1}{k} \bigl(   \tr \bigl( \tU_{s; k, f} |_\Sigma   \bigr) +  \langle \nabla f, \dot \gamma(s) \rangle   \bigr)^2 - \Ric_{k, f}(\Sigma_s, \dot \gamma(s)). \tag{2.58}
\end{align*}
This proves inequality \eqref{eq:2.52}.
Next, we recall the following elementary inequality used in deriving the $CD(K, N)$ condition:
\begin{align*}
\label{eq:2.59}
    \frac{(m+a)^2}{k} \geq \frac{m^2}{N} - \frac{a^2}{N-k}  \tag{2.59}
\end{align*}for $N > k$. By applying \eqref{eq:2.59} with 
\begin{align*}
m = \tr(\tU_{s; k, f} |_\Sigma), \qquad a = \langle \nabla f, \dot \gamma(s) \rangle    
\end{align*}
and combining it with \eqref{eq:2.58}, we obtain
\begin{align*}
    &\kern-2em \tr( \dot \tU_{s; k, f}  |_\Sigma ) \\
    &\leq - \frac{1}{k} \bigl(   \tr ( \tU_{s; k, f}  |_\Sigma   ) +  \langle \nabla f, \dot \gamma(s) \rangle   \bigr)^2 - \Ric_{k, f}(\Sigma_s, \dot \gamma(s)) \\
    &\leq - \frac{ 1 }{N} \tr ( \tU_{s; k, f}  |_\Sigma   )^2 - \Ric_{k, f}(\Sigma_s, \dot \gamma(s)) + \frac{1}{N-k}\langle \nabla f, \dot \gamma(s) \rangle^2 \\
    &=   - \frac{ 1 }{N} \tr ( \tU_{s; k, f} |_\Sigma   )^2 - \Ric^N_{k, f}(\Sigma_s, \dot \gamma(s)).
\end{align*}
This finishes the proof. 
\end{proof}

In the remainder of this section, we introduce the classical
Boltzmann entropy functional and its weighted analogue. We also define
the corresponding Boltzmann entropy tensors, which will play a central
role in the subsequent analysis.

\begin{fdefi}[Boltzmann entropy functional]
    Let \(\mu\in\cPac(M)\) be absolutely continuous with respect to
\(\Vol_g\). The \emph{Boltzmann entropy functional} is defined by
    \begin{align*}
    \label{eq:2.60}
        H (\mu) \coloneqq   \int_M \log \Bigl(  \frac{d \mu}{d \Vol_g}     \Bigr) \, d \mu. \tag{2.60}
    \end{align*}
    Moreover, the \emph{weighted $k$-Boltzmann entropy functional} is defined by
\begin{align*}
\label{eq:2.61}
    H_{k, f}(\mu) \coloneqq H(\mu) + \frac{n}{k} \int_M f \, d \mu = \int_M \log \Bigl(  \frac{d \mu}{d \Vol_g}     \Bigr) \, d \mu + \frac{n}{k} \int_M f \, d \mu. \tag{2.61}
\end{align*}

\end{fdefi}

\hypertarget{D:2.10}{
 \begin{fdefi}[Boltzmann entropy tensor]
 Let $\mu_0, \mu_1\in \cPac(M)$, we define the \emph{Boltzmann entropy tensor} associated with $(\mu_0, \mu_1)$ by
\begin{align*}
\label{eq:2.62}
    \tH^{\mu_0 \rightarrow \mu_1}_t (x) \coloneqq - \int_0^t \tU(s)(x) ds, \qquad   x \in D(\mu_0, \mu_1), \tag{2.62}
\end{align*}
Furthermore, given $f\in C^\infty(M)$ and $k\in \{1, \cdots, n\}$, we define the \emph{weighted $k$-Boltzmann entropy tensor} associated with $(\mu_0, \mu_1)$ by
\begin{align*}
\label{eq:2.63}
    \tH^{\mu_0 \rightarrow \mu_1}_{t; k, f} (x) \coloneqq - \int_0^t \tU_{k, f}(s)(x) ds   , \qquad   x \in D(\mu_0, \mu_1). \tag{2.63}
\end{align*}
\end{fdefi}}

Note that by Definition~\hyperlink{D:2.10}{2.10} and \eqref{eq:2.51}, we have
\begin{align*}
\label{eq:2.64}
    \tH^{\mu_0 \rightarrow \mu_1}_{t; k, f} (x) &= \tH^{\mu_0 \rightarrow \mu_1}_{t} (x) + \int_0^t \frac{1}{k} \frac{d}{ds} f(\gamma(s)) \, ds \cdot \Id   \\
    &= \tH^{\mu_0 \rightarrow \mu_1}_{t} (x) + \frac{ 1 }{k}  \bigl( f(\gamma(t)) - f(\gamma(0)) \bigr)  \Id. \tag{2.64}
\end{align*}

\section{Main results}
\label{sec:3}
Throughout this section, we adopt the convention of writing
$|\bullet|_g$ simply as $|\bullet|$ and
$\tr_g(\bullet)$ simply as $\tr(\bullet)$.
Moreover, unless otherwise specified, we assume that $N>k$ whenever
$f$ is nonconstant. In the special case where $f$ is constant, we also
allow $N=k$, as usual.

\subsection{Equivalent conditions under smooth settings}
\label{sec:3.1}
The main goal of this section is to establish several equivalent
characterizations of the lower bound
$\Ric_{k,f}^N\ge K$, linking this geometric curvature condition with
Bochner-type inequalities and the convexity of the weighted
$k$-Boltzmann entropy tensor along optimal transport. One of our main results is the following.

\hypertarget{T:3.1}{
\begin{fthm}
Let $(M^n,g,e^{-f}d\Vol_g)$ be a smooth complete weighted Riemannian manifold without boundary.
Then the following conditions are equivalent.
\begin{enumerate}
\item
The $N$-Bakry--Émery intermediate $k$-Ricci curvature satisfies
\[
\Ric_{k,f}^N\ge K.
\]
\item
For every $\psi\in C_c^\infty(M)$,
every $x\in M$, and every $k$-dimensional subspace
$\Sigma\subset T_xM$,
\begin{align*}
\label{eq:3.1}
\tr\Bigl(
\widetilde{\Gamma}_{2,k}^N(\psi)(x)\big|_\Sigma
\Bigr)
\ge
K|\nabla\psi(x)|^2.  \tag{3.1} 
\end{align*}
\item
For every $\psi\in C_c^\infty(M)$,
every $x\in M$, and every $k$-dimensional subspace
$\Sigma\subset T_xM$,
\begin{align*}
\label{eq:3.2}
\tr\Bigl(
\widetilde{\Gamma}_{2,k}^N(\psi)(x)\big|_\Sigma
\Bigr)
\ge
\tr\Bigl(
\bigl((\nabla^2\psi)(x)\bigr)^2\big|_\Sigma
\Bigr)
+
K|\nabla\psi(x)|^2. \tag{3.2}
\end{align*}
\item
Let
$\mu_0,\mu_1\in\mathcal P_c^{\mathrm{ac}}(M)$,
and let
$F_t(x)=\exp_x\bigl(t\nabla\theta(x)\bigr)$
be the optimal transport interpolation satisfying
$\mu_1=(F_1)_\sharp\mu_0$.
For any
$x\in D(\mu_0,\mu_1)$,
every $t\in[0,1]$,
and every
$k$-dimensional subspace
$\Sigma\subset T_xM$,
we have
\begin{align*}
\label{eq:3.3}
\tr\Bigl(
\ddot{\tH}^{\mu_0\to\mu_1}_{t;k,f}(x)\big|_\Sigma
\Bigr)
\ge
K|\nabla\theta(x)|^2
+
\frac{1}{N-k}df(\dot\gamma(t))^2, \tag{3.3}
\end{align*}where $\gamma(t) \coloneqq \exp_x(t \nabla \theta(x))$.
\item
Under the same assumptions as in {\rm(4)},
\begin{align*}
\label{eq:3.4}
\tr\Bigl(
\ddot{\tH}^{\mu_0\to\mu_1}_{t;k,f}(x)\big|_\Sigma
\Bigr)
\ge
\frac1N
\Bigl[
\tr\Bigl(
\dot{\tH}^{\mu_0\to\mu_1}_{t;k,f}(x)\big|_\Sigma
\Bigr)
\Bigr]^2
+
K|\nabla\theta(x)|^2. \tag{3.4}
\end{align*}
\item
Under the same assumptions as in condition {\rm(4)}, the function
\begin{align*}
\label{eq:3.5}
&\kern-4em t\longmapsto
e^{-\frac1N
\tr\bigl(
\tH^{\mu_0\to\mu_1}_{t;k,f}(x)\big|_\Sigma
\bigr)} \\
&\kern2em +
\frac KN
|\nabla\theta(x)|^2
\int_0^t
(t-s)
e^{-\frac1N
\tr\bigl(
\tH^{\mu_0\to\mu_1}_{s;k,f}(x)\big|_\Sigma
\bigr)}
\,ds \tag{3.5}
\end{align*}
is concave on $[0,1]$.
\end{enumerate}
\end{fthm}}


\begin{proof}
Clearly, $(3)\Rightarrow(2)$.
To prove $(1)\Rightarrow(3)$, recall the definition of the tensorial weighted $N$-Bakry--Émery operator \eqref{eq:2.32}, we have
\begin{align*}
      \widetilde{\Gamma}_{2,k}^N(\psi) =
    (\nabla^2\psi)^2
    +\Riem(\bullet,\nabla\psi)\nabla\psi
    +\frac1k \Bigl[  \Hess f(\nabla\psi,\nabla\psi) 
    -\frac1{N-k}df(\nabla\psi)^2 \Bigr] \Id.
\end{align*}
If $\nabla\psi(x)=0$, then
$\widetilde{\Gamma}_{2,k}^N(\psi)(x)
=
(\nabla^2\psi)^2(x)$,
and therefore \eqref{eq:3.2} follows immediately.

Now suppose that $\nabla\psi(x)\neq0$.
Let $\Sigma\subset T_xM$ be a $k$-dimensional subspace and let
$\{e_1,\dots,e_k\}$ be an orthonormal basis of $\Sigma$.
Then, at $x \in M$, we have
\begin{align*}
       \tr\Bigl(
\widetilde{\Gamma}_{2,k}^N(\psi)(x)\big|_\Sigma
\Bigr) 
         &= \tr \Bigl( (\nabla^2 \psi)^2(x) \big|_\Sigma \Bigr) + \Ric_{k, f}^N( \Sigma, \nabla \psi(x)). 
\end{align*}By the assumption in condition $(1)$, condition $(3)$ follows. 

\medskip

Next, we prove $(2)\Rightarrow(1)$ by contradiction.
Suppose that there exist
$x\in M$, a nonzero vector $v\in T_xM$, and a $k$-dimensional subspace
$\Sigma\subset T_xM$ such that
\[
\Ric_{k,f}^N(\Sigma,v)
<
K|v|^2.
\]
By the assumption in condition (2), we get
\begin{align*}
\label{eq:3.6}
\tr\Bigl(
\widetilde{\Gamma}_{2,k}^N(\psi)(x)\big|_\Sigma
\Bigr)
\ge
K|\nabla\psi(x)|^2   \tag{3.6}
\end{align*}
for every $\psi\in C_c^\infty(M)$.
By inequality \eqref{eq:3.6} and the definition of $\widetilde{\Gamma}_{2,k}^N$, we have
\begin{align*}
\label{eq:3.7}
\tr \Bigl( (\nabla^2 \psi)^2(x) \big|_\Sigma \Bigr)
+
\Ric_{k,f}^N(\Sigma,\nabla\psi(x))
\ge
K|\nabla\psi(x)|^2.   \tag{3.7}
\end{align*}
Now, by the standard jet construction, we may choose
$\psi\in C_c^\infty(M)$ such that
$\nabla\psi(x)=v$
and $\nabla^2\psi(x)=0$.
Hence, by choosing this $\psi$ and plugging this choice into inequality \eqref{eq:3.7}, we obtain 
\[
\Ric_{k,f}^N(\Sigma,v)
\ge
K|v|^2,
\]
which contradicts the assumption.
Therefore, the following conditions are equivalent:
\[
(1)\Longleftrightarrow(2)\Longleftrightarrow(3).
\]

\medskip

Next, we prove the equivalence of $(1)$ and $(4)$. To prove $(1)\Rightarrow(4)$, by equation \eqref{eq:2.56}, we have 
\begin{align*}
\label{eq:3.8}
&\kern-2em \tr\Bigl(
\ddot{\tH}^{\mu_0\to\mu_1}_{t;k,f}(x)\big|_\Sigma
\Bigr) \\
&=
    -
    \tr\bigl(
    \dot{\tU}_{t; k,f}(x) \big|_\Sigma
    \bigr) = \tr \bigl(  \tU_t^2(x) \big|_\Sigma   \bigr) + \Ric_k(\Sigma_t, \dot \gamma (t)) + \Hess f (\dot \gamma (t), \dot \gamma(t)) \\
    &= \tr \bigl(  \tU_t^2(x) \big|_\Sigma   \bigr) + \Ric^N_{k, f}(\Sigma_t, \dot \gamma(t) ) + \frac{1}{N-k} df(\dot \gamma(t))^2. \tag{3.8} 
\end{align*}where $\Sigma_t$ denotes the parallel transport of $\Sigma$ along $\gamma$.
By the assumption in condition $(1)$ and equation \eqref{eq:3.8}, we obtain 
\begin{align*}
&\kern-2em \tr\Bigl(
\ddot{\tH}^{\mu_0\to\mu_1}_{t;k,f}(x)\big|_\Sigma
\Bigr) \\
    &= \tr \bigl(  \tU_t^2(x) \big|_\Sigma   \bigr) + \Ric^N_{k, f}(\Sigma_t, \dot \gamma(t)) + \frac{1}{N-k}df(\dot \gamma(t))^2 \\
    &\geq K |\dot \gamma (t)|^2 + \frac{1}{N-k} df(\dot \gamma(t))^2 = K|\nabla\theta(x)|^2 + \frac{1}{N-k}df(\dot \gamma(t))^2.
\end{align*}Here, the last equality follows from the fact that $\gamma$ is a geodesic, and hence has constant speed. Therefore, $|\dot\gamma(t)|=|\nabla\theta(x)|$ for any $t \in [0, 1]$. Thus, condition $(4)$ follows.

To prove $(4)\Rightarrow(1)$, we again argue by contradiction.
Suppose that there exist
$x\in M$, a nonzero vector
$v\in T_xM$, and a $k$-dimensional subspace
$\Sigma\subset T_xM$ such that
\[
\Ric_{k,f}^N(\Sigma,v)
<
K|v|^2.
\]

Again, by the standard jet construction, we may choose
$\theta\in C_c^\infty(M)$ such that
$\nabla\theta(x)=v$ and 
$ \nabla^2\theta(x)= 0$.
After scaling $\theta$ (and $v$), we may further assume that
$-\theta$ is $\frac{d^2}{2}$-concave by
\cite[Theorem~13.5]{villani2009}.

Let
\[
F_t(y)
=
\exp_y(t\nabla\theta(y)),
\qquad
\mu_t=(F_t)_\sharp\mu_0,
\]
where
$\mu_0\in\mathcal P_c^{\mathrm{ac}}(M)$
is chosen so that
$x\in\supp\mu_0$.
Since $\theta$ is smooth, we have
$D(\mu_0,\mu_1)
=
\supp\mu_0$. Let $\Sigma_t$ denote the parallel transport of $\Sigma$
along the geodesic 
$\gamma(t)=\exp_x(tv)$, at $t = 0$, the weighted Riccati equation \eqref{eq:3.8} gives
\begin{align*}
\label{eq:3.9}
&\kern-2em \tr\Bigl(
\ddot{\tH}^{\mu_0\to\mu_1}_{0;k,f}(x)\big|_\Sigma
\Bigr) \\
    &= \tr \bigl(  \tU_0^2(x) \big|_\Sigma   \bigr) + \Ric^N_{k, f}(\Sigma_0, \dot \gamma(0) ) + \frac{1}{N-k} df(\dot \gamma(0))^2 \\
    &= \tr \Bigl(  (\nabla^2 \theta)^2(x) \big|_\Sigma   \Bigr)(x) + \Ric^N_{k, f}(\Sigma, v) + \frac{1}{N-k}df(v)^2 \\
    &=   \Ric^N_{k, f}(\Sigma, v) + \frac{1}{N-k}df(v)^2. \tag{3.9}
\end{align*}
On the other hand, by the assumption in condition $(4)$, we have
\begin{align*}
\label{eq:3.10}
&\kern-2em \tr\Bigl(
\ddot{\tH}^{\mu_0\to\mu_1}_{0;k,f}(x)\big|_\Sigma
\Bigr) \\
&\geq K|\nabla\theta(x)|^2
+
\frac{1}{N-k}df(\dot\gamma(0))^2 = K |v|^2 + \frac{1}{N-k}df(v)^2. \tag{3.10}
\end{align*}Equation \eqref{eq:3.9} and inequality \eqref{eq:3.10} imply that $\Ric^N_{k, f}(\Sigma, v) \geq K |v|^2$, which leads to a contradiction. Therefore, the following conditions are equivalent:
\[
(1)\Longleftrightarrow(4).
\]

Next, we prove the equivalence of $(1)$ and $(5)$.
To prove $(1)\Rightarrow(5)$, by taking the trace of equation \eqref{eq:3.9} over $\Sigma$ and using inequality \eqref{eq:2.53}, we obtain
\begin{align*}
\label{eq:3.11}
    &\kern-2em \tr\Bigl(
\ddot{\tH}^{\mu_0\to\mu_1}_{t;k,f}(x)\big|_\Sigma
\Bigr)  \\ 
    &=
    -
    \tr\Bigl(
    \dot{\tU}_{t; k,f}(x)\big|_\Sigma
    \Bigr) \ge
    \frac1N
    \Bigl[
    \tr\Bigl(
    \tU_{t; k,f}(x)\big|_\Sigma
    \Bigr)
    \Bigr]^2
    +
    \Ric_{k,f}^N(\Sigma_t,\dot\gamma(t))\\
    &=
    \frac1N
    \Bigl[
    \tr\Bigl(
    \dot{\mathbf H}^{\mu_0\to\mu_1}_{t;k,f}(x)
    \big|_\Sigma
    \Bigr)
    \Bigr]^2
    +
    \Ric_{k,f}^N(\Sigma_t,\dot\gamma(t)). \tag{3.11}
\end{align*}
By the assumption in condition $(1)$ and the fact that $\gamma$ is a geodesic, inequality \eqref{eq:3.11} becomes 
\[
\tr\Bigl(
\ddot{\tH}^{\mu_0\to\mu_1}_{t;k,f}(x)\big|_\Sigma
\Bigr)
\ge
\frac1N
    \Bigl[
    \tr\Bigl(
    \dot{\mathbf H}^{\mu_0\to\mu_1}_{t;k,f}(x)
    \big|_\Sigma
    \Bigr)
    \Bigr]^2
    +
K|\nabla\theta(x)|^2.
\] Thus, condition $(5)$ follows.

\medskip

To prove $(5)\Rightarrow(1)$, we again argue by contradiction.
Suppose that there exist
$x\in M$, a nonzero vector
$v\in T_xM$, and a $k$-dimensional subspace
$\Sigma\subset T_xM$ such that
\[
\Ric_{k,f}^N(\Sigma,v)
<
K|v|^2.
\]

Again, by the standard jet construction, we may choose
$\theta\in C_c^\infty(M)$ such that
$\nabla\theta(x)=v$ and 
$ 
\nabla^2\theta(x) 
=
-\frac{df(v)}{N-k} \Id$.
After scaling $\theta$ (and $v$), we may further assume that
$-\theta$ is $\frac{d^2}{2}$-concave by
\cite[Theorem~13.5]{villani2009}.

Let
\[
F_t(y)
=
\exp_y(t\nabla\theta(y)),
\qquad
\mu_t=(F_t)_\sharp\mu_0,
\]
where
$\mu_0\in\mathcal P_c^{\mathrm{ac}}(M)$
is chosen so that
$x\in\supp\mu_0$.
Since $\theta$ is smooth, we have
$D(\mu_0,\mu_1)
=
\supp\mu_0$.
 Let $\Sigma_t$ denote the parallel transport of $\Sigma$
along the geodesic 
$\gamma(t)=\exp_x(tv)$, at $t = 0$, the weighted Riccati equation \eqref{eq:3.8} gives
\begin{align*}
\label{eq:3.12}
&\kern-2em \tr\Bigl(
\ddot{\tH}^{\mu_0\to\mu_1}_{0;k,f}(x)\big|_\Sigma
\Bigr) \\
    &= \tr \bigl(  \tU_0^2(x) \big|_\Sigma   \bigr) + \Ric^N_{k, f}(\Sigma_0, \dot \gamma(0) ) + \frac{1}{N-k} df(\dot \gamma(0))^2 \\
    &= \tr \Bigl(  (\nabla^2 \theta)^2(x) \big|_\Sigma   \Bigr)(x) + \Ric^N_{k, f}(\Sigma, v) + \frac{1}{N-k}df(v)^2 \\
    &=   \Ric^N_{k, f}(\Sigma, v) + \frac{N}{(N-k)^2}df(v)^2.  \tag{3.12}
\end{align*}
On the other hand, by the assumption in condition $(5)$, we have
\begin{align*}
\label{eq:3.13}
&\kern-2em \tr\Bigl(
\ddot{\tH}^{\mu_0\to\mu_1}_{0;k,f}(x)\big|_\Sigma
\Bigr) \\
&\geq  
\frac1N
\Bigl[
\tr\Bigl(
\dot{\tH}^{\mu_0\to\mu_1}_{0;k,f}(x)\big|_\Sigma
\Bigr)
\Bigr]^2
+
K|\nabla\theta(x)|^2 =    \frac1N
\Bigl[ -
\tr\Bigl(
{\tU}_{0;k,f}(x)\big|_\Sigma
\Bigr)
\Bigr]^2
+
K|v|^2    \\
&= 
\frac1N
\Bigl[
\tr\Bigl(
\nabla^2\theta(x)\big|_\Sigma
\Bigr)
-
df(v)
\Bigr]^2
+
K|v|^2 = \frac{N}{(N-k)^2}df(v)^2 + K|v|^2. \tag{3.13}
\end{align*}

Equation \eqref{eq:3.12} and inequality \eqref{eq:3.13} imply that $\Ric^N_{k, f}(\Sigma, v) \geq K |v|^2$, which leads to a contradiction. Therefore, the following conditions are equivalent:
\[
(1) \Longleftrightarrow(5).
\]

\medskip

Finally, we prove the equivalence of $(5)$ and $(6)$.
By differentiating function \eqref{eq:3.5} twice with respect to $t$, we obtain
\begin{align*}
&\kern-2em \frac{d^2}{dt^2}
\biggl[
e^{-\frac1N
\tr\bigl(
\tH^{\mu_0\to\mu_1}_{t;k,f}(x)\big|_\Sigma
\bigr)}
+
\frac KN
|\nabla\theta(x)|^2
\int_0^t
(t-s)
e^{-\frac1N
\tr\bigl(
\tH^{\mu_0\to\mu_1}_{s;k,f}(x)\big|_\Sigma
\bigr)}
\,ds
\biggr] \\
&=
-\frac{
e^{-\frac1N
\tr\bigl(
\tH^{\mu_0\to\mu_1}_{t;k,f}(x)\big|_\Sigma
\bigr)}
}{N}
\biggl[
\tr\Bigl(
\ddot{\tH}^{\mu_0\to\mu_1}_{t;k,f}(x)\big|_\Sigma
\Bigr) \\
&\kern14em -
\frac1N
\Bigl[
\tr\Bigl(
\dot{\tH}^{\mu_0\to\mu_1}_{t;k,f}(x)\big|_\Sigma
\Bigr)
\Bigr]^2
-
K|\nabla\theta(x)|^2
\biggr].
\end{align*}
Therefore, the following conditions are equivalent:
\[
(5)\Longleftrightarrow(6).
\]
This finishes the proof.
\end{proof}

By Lemma~\hyperlink{L:2.6}{2.6}, the equivalent
characterizations in Theorem~\hyperlink{T:3.1}{3.1} admit the following
reformulation in terms of the induced endomorphisms on
\(\bigwedge^k TM\), which provides a more natural framework for the subsequent developments.

\hypertarget{C:3.1}{
\begin{fcor}
Let $(M^n,g,e^{-f}d\Vol_g)$ be a smooth complete weighted Riemannian manifold without boundary.
Then the following conditions are equivalent.
\begin{enumerate}
\item
The $N$-Bakry--Émery intermediate $k$-Ricci curvature satisfies
\[
\Ric_{k,f}^N\ge K.
\]
\item
For every $\psi\in C_c^\infty(M)$ and
every $x\in M$,
\begin{align*}
\label{eq:3.14}
\Bigl(
\widetilde{\Gamma}_{2;k,f}^N(\psi)(x)
\Bigr)^{[k]}
\ge
K|\nabla\psi(x)|^2\Id.  \tag{3.14} 
\end{align*}
\item
For every $\psi\in C_c^\infty(M)$ and
every $x\in M$,
\begin{align*}
\label{eq:3.15}
\Bigl(
\widetilde{\Gamma}_{2;k,f}^N(\psi)(x)
\Bigr)^{[k]}
\ge
\Bigl(
(\nabla^2\psi)^2(x)
\Bigr)^{[k]}
+
K|\nabla\psi(x)|^2\Id.  \tag{3.15} 
\end{align*}
\item
Let
$\mu_0,\mu_1\in\mathcal P_c^{\mathrm{ac}}(M)$,
and let
$F_t(x)=\exp_x(t\nabla\theta(x))$
be the optimal transport interpolation satisfying
$\mu_1=(F_1)_\sharp\mu_0$.
For any
$x\in D(\mu_0,\mu_1)$ and every $t\in[0,1]$, we have
\begin{align*}
\label{eq:3.16}
\Bigl(
\ddot{\tH}^{\mu_0\to\mu_1}_{t;k,f}(x)
\Bigr)^{[k]}
\ge
\Bigl(
K|\nabla\theta(x)|^2
+
\frac{1}{N-k}df(\dot\gamma(t))^2
\Bigr)\Id. \tag{3.16}
\end{align*}
\item
Under the same assumptions as in {\rm(4)}, for every decomposable unit $k$-vector
$\xi\in\bigwedge^k T_xM$, we have
\begin{align*}
\label{eq:3.17}
\Bigl\langle
\Bigl(
\ddot{\tH}^{\mu_0\to\mu_1}_{t;k,f}(x)
\Bigr)^{[k]}
\xi,\xi
\Bigr\rangle
\ge
\frac1N
\Bigl\langle
\Bigl(
\dot{\tH}^{\mu_0\to\mu_1}_{t;k,f}(x)
\Bigr)^{[k]}
\xi,\xi
\Bigr\rangle^2
+
K|\nabla\theta(x)|^2. \tag{3.17}
\end{align*}
\item
Under the same assumptions as in {\rm(4)}, for every decomposable unit $k$-vector
$\xi\in\bigwedge^k T_xM$, the function
\begin{align*}
\label{eq:3.18}
&\kern-4em t\longmapsto
 e^{-\frac1N
\bigl\langle
\bigl(
{\tH}^{\mu_0\to\mu_1}_{t;k,f}(x)
\bigr)^{[k]}\xi,\xi
\bigr\rangle}
\\
&+
\frac KN
|\nabla\theta(x)|^2
\int_0^t
(t-s)
e^{-\frac1N
\bigl\langle
\bigl(
{\tH}^{\mu_0\to\mu_1}_{s;k,f}(x)
\bigr)^{[k]}\xi,\xi
\bigr\rangle}
\,ds \tag{3.18}
\end{align*}
is concave on $[0,1]$.
\end{enumerate}
\end{fcor}}

\begin{proof}
    First, by Lemma~\hyperlink{L:2.6}{2.6} and Theorem~\hyperlink{T:3.1}{3.1}, we immediately see that the following conditions are equivalent: 
    \begin{align*}
       (1)\Longleftrightarrow(2)\Longleftrightarrow(3)\Longleftrightarrow(4)\Longleftrightarrow(5).
    \end{align*}

    Finally, we prove the equivalence of $(5)$ and $(6)$.
By differentiating function \eqref{eq:3.18} twice with respect to $t$, we obtain
\begin{align*}
&\kern-2em \frac{d^2}{dt^2}
\biggl[
e^{-\frac1N
\bigl\langle
\bigl(
{\tH}^{\mu_0\to\mu_1}_{t;k,f}(x)
\bigr)^{[k]}\xi,\xi
\bigr\rangle}
+
\frac KN
|\nabla\theta(x)|^2
\int_0^t
(t-s)
e^{-\frac1N
\bigl\langle
\bigl(
{\tH}^{\mu_0\to\mu_1}_{s;k,f}(x)
\bigr)^{[k]}\xi,\xi
\bigr\rangle}
\,ds
\biggr] \\
&=
-\frac{
e^{-\frac1N
\bigl\langle
\bigl(
{\tH}^{\mu_0\to\mu_1}_{t;k,f}(x)
\bigr)^{[k]}\xi,\xi
\bigr\rangle}
}{N}
\biggl[
\Bigl\langle
\Bigl(
\ddot{\tH}^{\mu_0\to\mu_1}_{t;k,f}(x)
\Bigr)^{[k]}
\xi,\xi
\Bigr\rangle
 \\
&\kern15em-
\frac1N
\Bigl\langle
\Bigl(
\dot{\tH}^{\mu_0\to\mu_1}_{t;k,f}(x)
\Bigr)^{[k]}
\xi,\xi
\Bigr\rangle^2 
- K|\nabla\theta(x)|^2
\biggr].
\end{align*}
Therefore, the following conditions are equivalent:
\[
(5)\Longleftrightarrow(6).
\]
This finishes the proof.
\end{proof}


\begin{frmk}
    In the unweighted case, that is, when $f$ is constant, the constant $N$ in both Theorem~\hyperlink{T:3.1}{3.1} and Corollary~\hyperlink{C:3.1}{3.1} can be replaced by $k$. More precisely, $\Ric_{k, f}^N$, $\widetilde{\Gamma}^N_{2; k, f}$, and $\tH_{t; k, f}$ can be replaced by $\Ric_k$, $\widetilde{\Gamma}_2$ and $\tH_t$, respectively.
\end{frmk}

\begin{frmk}
One advantage of the tensorial Bakry--Émery operator is that its lower bound
admits a natural reformulation as an operator inequality on the exterior power
$\bigwedge^k TM$. Indeed, by Lemma~\hyperlink{L:2.6}{2.6}, the inequality
\[
\tr\Bigl(
\widetilde{\Gamma}_{2,k}^N(\psi)(x)\big|_\Sigma
\Bigr)
\ge
\tr\Bigl(
\bigl((\nabla^2\psi)(x)\bigr)^2\big|_\Sigma
\Bigr)
+
K|\nabla\psi(x)|^2
\]
for every $x \in M$ and every $k$-dimensional subspace $\Sigma\subset T_xM$ is equivalent to
\[
\Bigl(
\widetilde{\Gamma}_{2;k,f}^N(\psi)
\Bigr)^{[k]}
\ge
\Bigl(
(\nabla^2\psi)^2
\Bigr)^{[k]}
+
K|\nabla\psi|^2\Id
\]
as an operator on $\bigwedge^k TM$.

In contrast, the second-order differential inequality satisfied by the tensorial entropy
operator contains the nonlinear term
\[
\Bigl[
\tr\Bigl(
\dot{\tH}^{\mu_0\to\mu_1}_{t;k,f}(x)\big|_\Sigma
\Bigr)
\Bigr]^2,
\]
which, under the exterior-power formulation, becomes
\[
\Bigl\langle
\Bigl(
\dot{\tH}^{\mu_0\to\mu_1}_{t;k,f}(x)
\Bigr)^{[k]}
\xi,\xi
\Bigr\rangle^2.
\]
Since this expression is quadratic rather than linear in
$\bigl(\dot{\tH}^{\mu_0\to\mu_1}_{t;k,f}(x)\bigr)^{[k]}$, it cannot be
reformulated as a linear operator inequality on $\bigwedge^k TM$. This
highlights a fundamental distinction between the tensorial Bakry--Émery operator
and the tensorial entropy operator, and illustrates that the former is particularly well
suited to the language of exterior powers.
\end{frmk}

\medskip

\subsection{Displacement convexity of the weighted Boltzmann entropy functional}
In this section, we apply Corollary~\hyperlink{C:3.1}{3.1} to prove that the weighted $k$-Boltzmann entropy
functional $H_{k,f}$ is displacement convex and derive an evolution variational inequality, which in turn yields quantitative estimates for the evolution of the
Wasserstein distance under the heat semigroup. 
In the unweighted setting, our results recover the corresponding
classical results of von Renesse--Sturm
\cite{von_renesse_transport_2005}.

\hypertarget{C:3.2}{
\begin{fcor}
Let
$\mu_0,\mu_1\in\mathcal P_c^{\mathrm{ac}}(M)$,
and let
$F_t(x)=\exp_x(t\nabla\theta(x))$
be the optimal transport interpolation satisfying
$\mu_1=(F_1)_\sharp\mu_0$.
Then, for every
$x\in D(\mu_0,\mu_1)$ and every $t\in[0,1]$, we have
\begin{align*}
\label{eq:3.19}
&\kern-2em
\tr_{\bigwedge^kT_xM}
\Bigl(
\bigl(
{\tH}^{\mu_0\to\mu_1}_{t;k,f}(x)
\bigr)^{[k]}
\Bigr)
\\
&\leq
-\binom{n}{k}N
\log\biggl[
\sigma^{(1-t)}_{\frac{K}{N}|\nabla\theta(x)|^2}
\\
&\kern8em
+
\sigma^{(t)}_{\frac{K}{N}|\nabla\theta(x)|^2}
\exp\biggl(
-\frac{
\tr_{\bigwedge^kT_xM}
\Bigl(
\bigl(
{\tH}^{\mu_0\to\mu_1}_{1;k,f}(x)
\bigr)^{[k]}
\Bigr)
}{
\binom{n}{k}N
}
\biggr)
\biggr]. \tag{3.19}
\end{align*}
where the distortion coefficient
$\sigma_{\bullet}^{(\bullet)}$
is defined in \eqref{eq:2.39}.
Moreover, for every
$x\in D(\mu_0,\mu_1)$ and every
$t\in[0,1]$,
\begin{align*}
\label{eq:3.20}
&\kern-2em
\tr_{\bigwedge^kT_xM}
\Bigl(
\bigl(
{\tH}^{\mu_0\to\mu_1}_{t;k,f}(x)
\bigr)^{[k]}
\Bigr)
-
t
\tr_{\bigwedge^kT_xM}
\Bigl(
\bigl(
{\tH}^{\mu_0\to\mu_1}_{1;k,f}(x)
\bigr)^{[k]}
\Bigr)
\\
&\le
\begin{cases}
-\dfrac{\binom{n}{k}K}{2}
|\nabla\theta(x)|^2
t(1-t),
&
K>0,
\\[3mm]
0,
&
K=0,
\\[3mm]
3\binom{n}{k}N
t(1-t)
\log\Biggl(
\dfrac{
\sinh\Bigl(\sqrt{-\frac{K}{N}}\,|\nabla\theta(x)|\Bigr)
}{
\sqrt{-\frac{K}{N}}\,|\nabla\theta(x)|
}
\Biggr),
&
K<0.
\end{cases}
\tag{3.20}
\end{align*}
\end{fcor}}

\begin{frmk}
When $K<0$ and $\nabla\theta(x)=0$, let $A(x) \coloneqq \sqrt{-\frac{K}{N}}\,|\nabla\theta(x)|$ we interpret
$\frac{\sinh (A(x))}{A(x)}$
by continuity, namely,
$\frac{\sinh (A(x) )}{A(x)}
\coloneqq 1$.
Equivalently, we have
\[
\log \Bigl( \frac{\sinh (A(x))}{A(x)} \Bigr)
\coloneqq 0.
\]
This agrees with the limiting value as $|\nabla\theta(x)|\to0$ and is consistent with the case $K=0$.
\end{frmk}

\begin{proof}
By taking the trace of inequality~\eqref{eq:3.17} over
\(\bigwedge^k T_xM\), we obtain
\begin{align*}
\label{eq:3.21}
&\kern-2em
\operatorname{tr}_{\bigwedge^k T_xM}
\Bigl(
\bigl(
\ddot{\tH}^{\mu_0\to\mu_1}_{t;k,f}(x)
\bigr)^{[k]}
\Bigr)
\\
&\geq
\frac{1}{N}
\operatorname{tr}_{\bigwedge^k T_xM}
\Biggl(
\Bigl(
\bigl(
\dot{\tH}^{\mu_0\to\mu_1}_{t;k,f}(x)
\bigr)^{[k]}
\Bigr)^2
\Biggr)
+
\binom{n}{k}K|\nabla\theta(x)|^2
\\
&\geq
\frac{1}{\binom{n}{k}N}
\Biggl[
\operatorname{tr}_{\bigwedge^k T_xM}
\Bigl(
\bigl(
\dot{\tH}^{\mu_0\to\mu_1}_{t;k,f}(x)
\bigr)^{[k]}
\Bigr)
\Biggr]^2
+
\binom{n}{k}K|\nabla\theta(x)|^2,
\tag{3.21}
\end{align*}
where the second inequality follows from
$\operatorname{tr}(A^2)
\geq
\frac{1}{\binom{n}{k}}
\bigl(\operatorname{tr}A\bigr)^2$
for every self-adjoint endomorphism of
\(\bigwedge^k T_xM\).

By applying Lemma~\hyperlink{L:2.3}{2.3} with
$a=\frac{1}{\binom{n}{k}N}$ and 
$b=\binom{n}{k}K|\nabla\theta(x)|^2$ to inequality \eqref{eq:3.21},
we obtain
\begin{align*}
&\kern-2em
\operatorname{tr}_{\bigwedge^k T_xM}
\Bigl(
\bigl(
\tH^{\mu_0\to\mu_1}_{t;k,f}(x)
\bigr)^{[k]}
\Bigr)
\\
&\leq
-\binom{n}{k}N
\log\Biggl[
\sigma^{(1-t)}_{\frac{K}{N}|\nabla\theta(x)|^2}
\exp\Biggl(
-\frac{
\operatorname{tr}_{\bigwedge^k T_xM}
\Bigl(
\bigl(
\tH^{\mu_0\to\mu_1}_{0;k,f}(x)
\bigr)^{[k]}
\Bigr)
}{
\binom{n}{k}N
}
\Biggr)
\\
&\kern11em
+
\sigma^{(t)}_{\frac{K}{N}|\nabla\theta(x)|^2}
\exp\Biggl(
-\frac{
\operatorname{tr}_{\bigwedge^k T_xM}
\Bigl(
\bigl(
\tH^{\mu_0\to\mu_1}_{1;k,f}(x)
\bigr)^{[k]}
\Bigr)
}{
\binom{n}{k}N
}
\Biggr)
\Biggr].
\end{align*}
Since
$\bigl(
\tH^{\mu_0\to\mu_1}_{0;k,f}(x)
\bigr)^{[k]}=0$,
the above inequality reduces to \eqref{eq:3.19}, that is,
\begin{align*}
&\kern-2em
\operatorname{tr}_{\bigwedge^k T_xM}
\Bigl(
\bigl(
\tH^{\mu_0\to\mu_1}_{t;k,f}(x)
\bigr)^{[k]}
\Bigr)
\\
&\leq
-\binom{n}{k}N
\log\Biggl[
\sigma^{(1-t)}_{\frac{K}{N}|\nabla\theta(x)|^2}
+
\sigma^{(t)}_{\frac{K}{N}|\nabla\theta(x)|^2}
\exp\Biggl(
-\frac{
\operatorname{tr}_{\bigwedge^k T_xM}
\Bigl(
\bigl(
\tH^{\mu_0\to\mu_1}_{1;k,f}(x)
\bigr)^{[k]}
\Bigr)
}{
\binom{n}{k}N
}
\Biggr)
\Biggr].
\end{align*}
Inequality \eqref{eq:3.20} follows from
Proposition~\hyperlink{P:2.3}{2.3}. This finishes the proof. 
\end{proof}

Before turning to the convexity results, we emphasize that the tensorial
entropy functional contains strictly more information than the classical
Boltzmann entropy functional. Indeed, the latter is recovered by taking the
trace of the former.

Since our goal is to study intermediate $k$-Ricci curvature, it is natural
to formulate the framework on the exterior power bundle $\bigwedge^k TM$. The
following trace formula relates the entropy tensor
to its induced endomorphism on $\bigwedge^k TM$. Consequently, the
weighted $k$-Boltzmann entropy functional can be recovered from the induced action of the
entropy tensor on exterior powers.

\begin{flemma}
    Let $\mu_0, \mu_1 \in \cPac(M)$, and let $(\mu_t)_{t\in [0, 1]}$ be the Wasserstein geodesic in $\cPac(M)$ between $\mu_0$ and $\mu_1$. Then, for any $t \in [0, 1]$, we have
    \begin{align*}
    \label{eq:3.22}
        H_{k, f}(\mu_t) &= H_{k, f}(\mu_0) +  \int_M \tr \bigl( \tH^{\mu_0 \rightarrow \mu_1}_{t; k, f}(x)  \bigr)  \, d \mu_0(x)   \\
        &= H_{k, f}(\mu_0) +  \frac{1}{\binom{n-1}{k-1}} \int_M \tr_{\bigwedge^kT_xM} 
        \Bigl( \bigl( \tH_{t;k,f}^{\mu_0\to\mu_1}(x) \bigr)^{[k]} \Bigr) \,d\mu_0(x). \tag{3.22}
    \end{align*}
\end{flemma}

\begin{proof}
    First, we fix $t \in [0, 1]$, recall that there exists a Borel set $K_t \subset M$ of full $\mu_0$-measure such that, for every $x \in K_t$, we have
    \begin{align*}
    \label{eq:3.23}
        \frac{d \mu_0}{d \Vol_g} = \frac{d \mu_t}{d \Vol_g}(F_t(x)) \det(\tJ_t(x)). \tag{3.23}
    \end{align*}By taking logarithms of equation \eqref{eq:3.23} and integrating both sides with respect to $\mu_0$ over $K_t$, we obtain

\begin{align*}
&\kern-2em H_{k,f}(\mu_0) \\
&= H(\mu_0) + \frac{n}{k} \int_M f \, d \mu_0 = 
\int_{K_t}
\log\Bigl(
\frac{d\mu_0}{d\Vol_g}(x)
\Bigr)
\,d\mu_0(x)
+
\frac{n}{k}
\int_M
f\,d\mu_0
\\
&=
\int_{K_t}
\log\Bigl(
\frac{d\mu_t}{d\Vol_g}\bigl(F_t(x)\bigr)
\Bigr)
\,d\mu_0(x)
+
\int_{K_t}
\log\det\bigl(\tJ_t(x)\bigr)
\,d\mu_0(x)
+
\frac{n}{k}
\int_M
f\,d\mu_0.
\end{align*}
Since $K_t$ has full $\mu_0$-measure, we may replace $K_t$ by $M$ to obtain
\begin{align*}
\label{eq:3.24}
&\kern-2em H_{k,f}(\mu_0) \\
&=
\int_M
\log\Bigl(
\frac{d\mu_t}{d\Vol_g}\bigl(F_t(x)\bigr)
\Bigr)
\,d\mu_0(x)
+
\int_M
\log\det\bigl(\tJ_t(x)\bigr)
\,d\mu_0(x)
+
\frac{n}{k}
\int_M
f\,d\mu_0
\\
&=
H_{k,f}(\mu_t)
+
\int_M
\log\det\bigl(\tJ_t(x)\bigr)
\,d\mu_0(x)
+
\frac{n}{k}
\int_M
f\,d\mu_0
-
\frac{n}{k}
\int_M
f\,d\mu_t. \tag{3.24}
\end{align*}

    Next, by the Jacobi formula, we get
    \begin{align*}
    \label{eq:3.25}
        &\kern-2em \frac{d}{dt} \log \det \bigl( \tJ_t(x) \bigr) \\
        &= \tr \Bigl ( \dot \tJ_t(x)  \tJ_t(x)^{-1}  \Bigr) = \tr \bigl( \tU_t(x) \bigr) = - \frac{d}{dt} \tr \Bigl( \tH^{\mu_0 \rightarrow \mu_1}_t(x)  \Bigr). \tag{3.25}
    \end{align*}
    
    Since $\tJ_0(x)=\Id$ and $\tH_0^{\mu_0\to\mu_1}(x)=0$,
integrating identity \eqref{eq:3.22} with respect to $t$ gives
$\log\det(\tJ_t(x))
=
-
\tr\bigl(
\tH_t^{\mu_0\to\mu_1}(x)
\bigr)$. Thus, by the definition of $\tH^{\mu_0\to\mu_1}_{t; k, f}(x)$, we obtain
\begin{align*}
\label{eq:3.26}
    \log\det\bigl(\tJ_t(x)\bigr)
=
-
\tr\Bigl(
\tH_{t;k,f}^{\mu_0\to\mu_1}(x)
\Bigr)
+
\frac{n}{k}
\bigl(
f(F_t(x))-f(x)
\bigr). \tag{3.26}
\end{align*}
    
By combining Lemma~\hyperlink{L:2.6}{2.6} with equalities \eqref{eq:3.24} and \eqref{eq:3.26}, we obtain
    \begin{align*}
H_{k,f}(\mu_t)
&=
H_{k,f}(\mu_0)
-
\int_M
\log\det\bigl(\tJ_t(x)\bigr)
\,d\mu_0(x)
-
\frac{n}{k}
\int_M
f\,d\mu_0
+
\frac{n}{k}
\int_M
f\,d\mu_t
\\
&=
H_{k,f}(\mu_0)
+
\int_M
\tr\bigl(
\tH_{t;k,f}^{\mu_0\to\mu_1}(x)
\bigr)
\,d\mu_0(x)
\\
&=
H_{k,f}(\mu_0)
+
\frac{1}{\binom{n-1}{k-1}}
\int_M
\operatorname{tr}_{\bigwedge^kT_xM}
\Bigl(
\bigl(
\tH_{t;k,f}^{\mu_0\to\mu_1}(x)
\bigr)^{[k]}
\Bigr)
\,d\mu_0(x).
\end{align*}
This finishes the proof.
\end{proof}

As an immediate consequence of Corollary~\hyperlink{C:3.2}{3.2} and Lemma~\hyperlink{L:3.1}{3.1}, we obtain the following
displacement convexity result for the weighted Boltzmann $k$-entropy.

\begin{fcor}
Let $\mu_0, \mu_1 \in \cPac(M)$, and let $(\mu_t)_{t\in [0, 1]}$ be the Wasserstein geodesic in $\cPac(M)$ between $\mu_0$ and $\mu_1$.
Assume that $\Ric_{k,f}^N\ge K$, then for every $t\in[0,1]$, we have
\begin{align*}
\label{eq:3.27}
    &\kern-2em    H_{k, f}(\mu_t)  - (1-t) H_{k, f}(\mu_0) - t H_{k, f}(\mu_1)     \\
    &\leq      \begin{cases}
    -\dfrac{nK}{2k} t(1-t) \cW_2^2(\mu_0, \mu_1),
    & K  >0,\\[3mm]
    0,
    & K   =0,\\[3mm] 
     \frac{3nN}{k} t(1-t)  
    \log \Biggl(
    \dfrac{\sinh\Bigl(\sqrt{-\frac{K}{N}}   \cW_2(\mu_0, \mu_1) \Bigr)  }{\sqrt{-\frac{K}{N}}    \cW_2(\mu_0, \mu_1) }
    \Biggr),
    & K  <0.
    \end{cases}  \tag{3.27}
\end{align*}
\end{fcor}

\begin{frmk}
Assume that $K<0$. For $r \geq 0$, we have the estimate
\[
\log\left(\frac{\sinh r}{r}\right)
\leq
\frac{r^2}{6},
\qquad r\geq0.
\]
With $r = \sqrt{-\frac{K}{N}}
\cW_2(\mu_0,\mu_1)$,
we obtain
\[
\log\Biggl(
\frac{
\sinh\Bigl(
\sqrt{-\frac{K}{N}}
\cW_2(\mu_0,\mu_1)
\Bigr)
}{
\sqrt{-\frac{K}{N}}
\cW_2(\mu_0,\mu_1)
}
\Biggr)
\leq
-\frac{K}{6N}
\cW_2^2(\mu_0,\mu_1).
\]
Consequently, we have the following quadratic upper bound:
\begin{align*}
  \frac{3nN}{k}t(1-t)
\log\Biggl(
\frac{
\sinh\Bigl(
\sqrt{-\frac{K}{N}}
\cW_2(\mu_0,\mu_1)
\Bigr)
}{
\sqrt{-\frac{K}{N}}
\cW_2(\mu_0,\mu_1)
}
\Biggr)
\leq
-\frac{nK}{2k}
t(1-t)
\cW_2^2(\mu_0,\mu_1).
\end{align*}
Thus, for every $K\in\mathbb R$, we obtain the displacement convexity
estimate:
\begin{align*}
\label{eq:3.28}
    H_{k, f}(\mu_t) - t H_{k, f}(\mu_1) - (1-t) H_{k, f}(\mu_0) \leq -\dfrac{nK}{2k} t(1-t) \cW_2^2(\mu_0, \mu_1). \tag{3.28}
\end{align*}
\end{frmk}

\begin{proof}
First, by integrating inequality \eqref{eq:3.20} over $M$ with respect to $\mu_0$, we obtain
\begin{align*}
&\kern-2em \int_M  \tr_{\bigwedge^kT_xM}
\Bigl(
\bigl(
{\tH}^{\mu_0\to\mu_1}_{t;k,f}(x)
\bigr)^{[k]} \, d \mu_0(x)
- t  \int_M  \tr_{\bigwedge^kT_xM}
\Bigl(
\bigl(
{\tH}^{\mu_0\to\mu_1}_{1;k,f}(x)
\bigr)^{[k]} \, d \mu_0(x) \\
&\leq     \begin{cases}
    -\dfrac{\binom{n}{k}K}{2} t(1-t) \int_M |\nabla \theta (x)|^2 \, d \mu_0(x),
    & K  >0,\\[3mm]
    0,
    & K   =0,\\[3mm] 
    3 \binom{n}{k}N t(1-t) \int_M
    \log \Biggl(
    \dfrac{\sinh\Bigl(\sqrt{-\frac{K}{N}}    |\nabla \theta(x)| \Bigr)  }{\sqrt{-\frac{K}{N}}    |\nabla \theta(x)| }
    \Biggr) \, d \mu_0(x),
    & K  <0,
    \end{cases} \\
&=     \begin{cases}
    -\dfrac{\binom{n}{k}K}{2} t(1-t) \cW_2^2(\mu_0, \mu_1),
    & K  >0,\\[3mm]
    0,
    & K   =0,\\[3mm] 
    3 \binom{n}{k}N t(1-t) \int_M
    \log \Biggl(
    \dfrac{\sinh\Bigl(\sqrt{-\frac{K}{N}}    |\nabla \theta(x)| \Bigr)  }{\sqrt{-\frac{K}{N}}    |\nabla \theta(x)| }
    \Biggr) \, d \mu_0(x),
    & K  <0,
    \end{cases}     \\
    &\leq \begin{cases}
    -\dfrac{\binom{n}{k}K}{2} t(1-t) \cW_2^2(\mu_0, \mu_1),
    & K  >0,\\[3mm]
    0,
    & K   =0,\\[3mm] 
    3 \binom{n}{k}N t(1-t)  
    \log \Biggl(
    \dfrac{\sinh\Bigl(\sqrt{-\frac{K}{N}}   \cW_2(\mu_0, \mu_1) \Bigr)  }{\sqrt{-\frac{K}{N}}    \cW_2(\mu_0, \mu_1) }
    \Biggr),
    & K  <0.
    \end{cases} 
\end{align*}Here, the last inequality follows from Jensen's inequality. Since if we define $\Phi(s) \coloneqq \log \bigl( \frac{\sinh(\sqrt{s})}{s} \bigr)$ for $s \geq 0$ with $\Phi(0) = 0$, then one can check that $\Phi$ is increasing and concave on $[0, \infty)$. That is, we have
\begin{align*}
    &\kern-2em \int_M \log \Biggl(
    \dfrac{\sinh\Bigl(\sqrt{-\frac{K}{N}}    |\nabla \theta(x)| \Bigr)  }{\sqrt{-\frac{K}{N}}    |\nabla \theta(x)| }
    \Biggr) \, d \mu_0(x) \\
    &= \int_M \Phi \Bigl(  -\frac{K}{N} |\nabla \theta(x)|^2  \Bigr) \, d \mu_0(x) 
    \leq \Phi \biggl (   - \frac{K}{N}  \int_M |\nabla \theta(x)|^2 \, d \mu_0(x)   \biggr) \\
    &= \log \Biggl(
    \dfrac{\sinh\Bigl(\sqrt{-\frac{K}{N}}   \cW_2(\mu_0, \mu_1) \Bigr)  }{\sqrt{-\frac{K}{N}}    \cW_2(\mu_0, \mu_1) }
    \Biggr).
\end{align*}

\medskip

Therefore, by equality \eqref{eq:3.23}, we obtain 
\begin{align*}
    &\kern-2em  H_{k, f}(\mu_t)  - (1-t) H_{k, f}(\mu_0) - t H_{k, f}(\mu_1)      \\
    &\leq     \begin{cases}
    -\dfrac{nK}{2k} t(1-t) \cW_2^2(\mu_0, \mu_1),
    & K  >0,\\[3mm]
    0,
    & K   =0,\\[3mm] 
    \frac{3nN}{k} t(1-t)  
    \log \Biggl(
    \dfrac{\sinh\Bigl(\sqrt{-\frac{K}{N}}   \cW_2(\mu_0, \mu_1) \Bigr)  }{\sqrt{-\frac{K}{N}}    \cW_2(\mu_0, \mu_1) }
    \Biggr),
    & K  <0.
    \end{cases}  
\end{align*}
This finishes the proof. 
\end{proof}

\begin{frmk}
By Remark~\hyperlink{R:2.2}{2.2}, for any
$1\leq k_1\leq k_2\leq n$, the curvature condition
$\Ric^{N_1}_{k_1,f}\geq K_1$
implies
$\Ric^{\frac{k_2}{k_1}N_1}_{k_2,\frac{k_2}{k_1}f}
\geq
\frac{k_2}{k_1}K_1$.

On the other hand, the weighted Boltzmann $k$-entropy is invariant under this
scaling. Indeed,
\[
H_{k_2,\frac{k_2}{k_1}f}(\mu)
=
H(\mu)
+
\frac{n}{k_2}
\int_M
\frac{k_2}{k_1}f\,d\mu
=
H(\mu)
+
\frac{n}{k_1}
\int_M
f\,d\mu
=
H_{k_1,f}(\mu).
\]
Moreover, we have
\[
\frac{K_1}{N_1}
=
\frac{\frac{k_2}{k_1}K_1}
{\frac{k_2}{k_1}N_1},
\qquad
\frac{k_1}{N_1}
=
\frac{k_2}{\frac{k_2}{k_1}N_1}.
\]
Consequently, the displacement convexity inequality obtained from
$\Ric^{N_1}_{k_1,f}\geq K_1$ coincides exactly with that obtained from its
rescaled consequence
\[
\Ric^{\frac{k_2}{k_1}N_1}_{k_2,\frac{k_2}{k_1}f}
\geq
\frac{k_2}{k_1}K_1.
\]

This shows that the weighted Boltzmann $k$-entropy
$H_{k,f}$ depends only on the normalized weight $f/k$.
Consequently, its displacement convexity cannot distinguish
$\Ric^{N_1}_{k_1,f}\geq K_1$
from its rescaled counterparts for smaller values of $k$.
In particular, the functional $H_{k,f}$ alone cannot detect the stronger
geometric information contained in lower bounds for smaller intermediate
Bakry--Émery Ricci curvature.
\end{frmk}

\begin{flemma}
Let $\mu_0, \mu_1 \in \cPac(M)$, and let $(\mu_t)_{t\in [0, 1]}$ be the Wasserstein geodesic in $\cPac(M)$ between $\mu_0$ and $\mu_1$.
Assume that $\Ric_{k,f}^N\ge K$, then we have
\begin{align}
\label{eq:3.29}
\frac{d^2}{dt^2}H_{k,f}(\mu_t)
\geq
\frac{nK}{k}\cW_2^2(\mu_0,\mu_1)
\tag{3.29}
\end{align}
in the distributional sense on \((0,1)\). More precisely, for every
nonnegative function
\(\varphi\in C_c^\infty((0,1))\),
\begin{align*}
\int_0^1 H_{k,f}(\mu_t)\varphi''(t)\,dt
\geq
\frac{nK}{k}\cW_2^2(\mu_0,\mu_1)
\int_0^1\varphi(t)\,dt.
\end{align*}
Consequently, we have
\begin{align}
\label{eq:3.30}
H_{k,f}(\mu_1)-H_{k,f}(\mu_0)
\geq
\left.
\frac{d^+}{dt}H_{k,f}(\mu_t)
\right|_{t=0}
+
\frac{nK}{2k}\cW_2^2(\mu_0,\mu_1), \tag{3.30}
\end{align}
where the right derivative is understood in the extended sense, that is, we allow the right derivative to be $- \infty$.
\end{flemma}

\begin{proof}
By inequality \eqref{eq:3.28}, for every
\(0\leq s<r\leq 1\) and every \(t\in[s,r]\), we have
\begin{align*}
\label{eq:3.31}
&\kern-2em H_{k, f}(\mu_t) \\
&\leq
\frac{r-t}{r-s} H_{k, f}(\mu_s)
+
\frac{t-s}{r-s} H_{k, f}(\mu_r)
-\frac{nK}{2k}
\frac{(t-s)(r-t)}{(r-s)^2}
\cW_2^2(\mu_s,\mu_r) \\
&=
\frac{r-t}{r-s} H_{k, f}(\mu_s)
+
\frac{t-s}{r-s} H_{k, f}(\mu_r)
-\frac{nK}{2k}
 (t-s)(r-t) 
\cW_2^2(\mu_0,\mu_1). \tag{3.31}
\end{align*}
By \eqref{eq:3.31}, that the function
\[
g(t)\coloneqq H_{k, f}(\mu_t) -\frac{nK}{2k}\cW_2^2(\mu_0, \mu_1) t^2
\]
is convex on \([0,1]\). Since using
\[
\frac{r-t}{r-s}s^2
+
\frac{t-s}{r-s}r^2
-
t^2
=
(t-s)(r-t),
\]
inequality \eqref{eq:3.31} is equivalent to
\[
g(t)
\leq
\frac{r-t}{r-s}g(s)
+
\frac{t-s}{r-s}g(r).
\]

Since \(g\) is convex, its distributional second derivative is
nonnegative. Hence, we obtain 
\[
\frac{d^2}{dt^2}H_{k,f}(\mu_t) \geq 
\frac{nK}{k}\cW_2^2(\mu_0,\mu_1)
\]
in the distributional sense on \((0,1)\), which proves
\eqref{eq:3.29}.

\medskip

Last, by taking \(s=0\) and \(r=1\) in inequality \eqref{eq:3.28},
we obtain
\[
H_{k, f}(\mu_t)
\leq
(1-t)H_{k, f}(\mu_0) + t H_{k, f}(\mu_1)
-\frac{nK}{2k} \cW_2^2(\mu_0, \mu_1)t(1-t).
\]
Thus, for every \(t\in(0,1]\),
\[
\frac{H_{k, f}(\mu_t)-H_{k, f}(\mu_0)}{t}
\leq
H_{k, f}(\mu_1)-H_{k, f}(\mu_0)
-\frac{nK}{2k} \cW_2^2(\mu_0, \mu_1)(1-t).
\]
By letting \(t\downarrow0\), we obtain 
\[
 \frac{d^+}{dt}H_{k, f}(\mu_t) \Big|_{t=0}
\leq
H_{k, f}(\mu_1) -  H_{k, f}(\mu_0) -\frac{nK}{2k}\cW_2^2(\mu_0, \mu_1).
\]
This finishes the proof. 
\end{proof}

The preceding lemma establishes the infinitesimal convexity of
$H_{k,f}$ along Wasserstein geodesics. We now combine this convexity
estimate with the evolution variational inequality for the heat
semigroup. Together with \eqref{eq:3.30}, this yields an evolution
inequality for the Wasserstein distance, leading to a quantitative
estimate for the action of the heat semigroup on Wasserstein distance.


\begin{flemma}
Let $(M,g,e^{-f}\,d\Vol_g)$ be a closed weighted Riemannian
manifold, let $\mu_0,\mu_1\in\cPac(M)$, and let
$(P_\tau)_{\tau\geq0}$ be the heat semigroup on $M$.
For each $\tau>0$, let
$\bigl(\mu_s(\tau)\bigr)_{s\in[0,1]}$
be the Wasserstein geodesic from $P_\tau\mu_0$ to $\mu_1$.
Then, for every $T>0$ and almost every $\tau\in(0,T)$,
\begin{align*}
\label{eq:3.32}
&\kern-2em \frac12
\frac{d}{d\tau}
\cW_2^2(P_\tau\mu_0,\mu_1) \\
&\leq
\frac{d^+}{ds}
H_{k,f}\bigl(\mu_s(\tau)\bigr)
\Big|_{s=0}
-
\frac{n}{k}
\int_M
\Bigl\langle
\nabla f,
\nabla\theta^{P_\tau\mu_0\rightarrow\mu_1}
\Bigr\rangle
\,dP_\tau\mu_0.
\tag{3.32}
\end{align*}
Here, for $\nu_0,\nu_1\in\cPac(M)$,
$-\theta^{\nu_0\rightarrow\nu_1}$ denotes a Kantorovich potential
from $\nu_0$ to $\nu_1$.
Consequently, if $\Ric_{k,f}^N\geq K$, then, for every
$\mu_0,\mu_1\in\cPac(M)$ and every $T \geq0$,
\begin{align*}
\label{eq:3.33}
&\kern-2em \cW_2(P_T \mu_0,P_T \mu_1) \\
&\leq
\begin{cases}
\displaystyle
e^{-\frac{nK}{k}T}
\cW_2(\mu_0,\mu_1)
+
\frac{2}{K}
\|\nabla f\|_{L^\infty(M)}
\Bigl(
1-e^{-\frac{nK}{k}T}
\Bigr),
& K\neq0,
\\[4mm]
\displaystyle
\cW_2(\mu_0,\mu_1)
+
\frac{2n}{k}
\|\nabla f\|_{L^\infty(M)}
 T,
& K=0.
\end{cases}
\tag{3.33}
\end{align*}
\end{flemma}

\begin{proof}
Since the heat flow is the Wasserstein gradient flow of the
Boltzmann entropy $H$, its evolution variational inequality gives
\begin{align*}
\label{eq:3.34}
\frac12
\frac{d^+}{d\tau}
\cW_2^2(P_\tau\mu_0,\mu_1)
\leq
\frac{d^+}{ds}
H\bigl(\mu_s(\tau)\bigr)
\Big|_{s=0}.
\tag{3.34}
\end{align*}
By the relation between $H$ and $H_{k,f}$, we have
\begin{align*}
\frac{d^+}{ds}
H\bigl(\mu_s(\tau)\bigr)
\Big|_{s=0}
&=
\frac{d^+}{ds}
H_{k,f}\bigl(\mu_s(\tau)\bigr)
\Big|_{s=0}
-
\frac{n}{k}
\int_M
\Bigl\langle
\nabla f,
\nabla\theta^{P_\tau\mu_0\rightarrow\mu_1}
\Bigr\rangle
\,dP_\tau\mu_0.
\end{align*}
This proves \eqref{eq:3.32}. Moreover, by the Cauchy--Schwarz
inequality and the identity
\[
\int_M
\bigl|
\nabla\theta^{P_\tau\mu_0\rightarrow\mu_1}
\bigr|^2
\,dP_\tau\mu_0
=
\cW_2^2(P_\tau\mu_0,\mu_1),
\]
we obtain
\begin{align*}
\label{eq:3.35}
&\kern-2em \frac12
\frac{d^+}{d\tau}
\cW_2^2(P_\tau\mu_0,\mu_1) \\
&\leq
\frac{d^+}{ds}
H_{k,f}\bigl(\mu_s(\tau)\bigr)
\Big|_{s=0}
+
\frac{n}{k}
\|\nabla f\|_{L^\infty(M)}
\cW_2(P_\tau\mu_0,\mu_1).
\tag{3.35}
\end{align*}
Since
$\tau\mapsto\mathscr W_2^2(P_\tau\mu_0,\mu_1)$
is absolutely continuous, for every $T>0$ and for almost every
$\tau\in(0,T)$,
\[
\frac{d^+}{d\tau}
\mathscr W_2^2(P_\tau\mu_0,\mu_1)
=
\frac{d}{d\tau}
\mathscr W_2^2(P_\tau\mu_0,\mu_1).
\]
Now, assume that $\Ric_{k,f}^N\geq K$, by combining
\eqref{eq:3.35} with \eqref{eq:3.30}, for almost every $\tau \in (0, T)$, we find
\begin{align*}
\label{eq:3.36}
&\kern-2em \frac12
\frac{d}{d\tau}
\cW_2^2(P_\tau\mu_0,\mu_1) \\
&\leq
H_{k,f}(\mu_1)
-
H_{k,f}(P_\tau\mu_0)
-
\frac{nK}{2k}
\cW_2^2(P_\tau\mu_0,\mu_1)
\\
&\kern2em
+
\frac{n}{k}
\|\nabla f\|_{L^\infty(M)}
\cW_2(P_\tau\mu_0,\mu_1). 
\tag{3.36}
\end{align*}

For $r,s\geq0$, we set
$F(r,s)
\coloneqq
\cW_2^2(P_r\mu_0,P_s\mu_1)$.
By applying \eqref{eq:3.36} with $\mu_1$ replaced by $P_s\mu_1$, for almost every $r \in (0, T)$, we get
\begin{align*}
\label{eq:3.37}
&\kern-2em \frac12
\frac{\partial}{\partial r}F(r,s) \\
&\leq
H_{k,f}(P_s\mu_1)
-
H_{k,f}(P_r\mu_0)
-
\frac{nK}{2k}F(r,s) \\
&\kern2em +
\frac{n}{k}
\|\nabla f\|_{L^\infty(M)}
\sqrt{F(r,s)}.
\tag{3.37}
\end{align*}
By interchanging the roles of $\mu_0$ and $\mu_1$, for almost every $s \in (0, T)$, we similarly
obtain
\begin{align*}
\label{eq:3.38}
&\kern-2em \frac12
\frac{\partial}{\partial s}F(r,s) \\
&\leq
H_{k,f}(P_r\mu_0)
-
H_{k,f}(P_s\mu_1)
-
\frac{nK}{2k}F(r,s) \\
&\kern2em 
+
\frac{n}{k}
\|\nabla f\|_{L^\infty(M)}
\sqrt{F(r,s)}.
\tag{3.38}
\end{align*}

Since the map $(r,s)\mapsto F(r,s)$ is absolutely continuous
along the diagonal, for almost every $\tau \in (0, T)$,
\[
\frac{d}{d\tau}F(\tau,\tau)
=
\frac{\partial}{\partial r}F(r,s)\Big|_{r=s=\tau}
+
\frac{\partial}{\partial s}F(r,s)\Big|_{r=s=\tau}.
\]
By adding inequalities \eqref{eq:3.37} and \eqref{eq:3.38} and then setting
$r=s=\tau$, for almost every $\tau \in (0, T)$, we obtain
\begin{align*}
\label{eq:3.39}
&\kern-2em \frac12
\frac{d}{d\tau}
\cW_2^2(P_\tau\mu_0,P_\tau\mu_1) \\
&\leq
-\frac{nK}{k}
\cW_2^2(P_\tau\mu_0,P_\tau\mu_1)
+
\frac{2n}{k}
\|\nabla f\|_{L^\infty(M)}
\cW_2(P_\tau\mu_0,P_\tau\mu_1).
\tag{3.39}
\end{align*}

Since $\cW_2(P_\tau\mu_0,P_\tau\mu_1)$ is nonnegative and absolutely continuous, one can check that
inequality \eqref{eq:3.39} implies, in the almost-everywhere
sense,
\[
\frac{d}{d \tau} \cW_2(P_\tau\mu_0,P_\tau\mu_1)
\leq
-\frac{nK}{k} \cW_2(P_\tau\mu_0,P_\tau\mu_1)
+
\frac{2n}{k}
\|\nabla f\|_{L^\infty(M)}.
\]

If $K\neq0$, Grönwall's inequality gives
\begin{align*}
\cW_2(P_T \mu_0,P_T \mu_1)
&\leq
e^{-\frac{nK}{k}T}
\cW_2(\mu_0,\mu_1)
+
\frac{2}{K}
\|\nabla f\|_{L^\infty(M)}
\Bigl(
1-e^{-\frac{nK}{k}T}
\Bigr).
\end{align*}
For $K=0$, integration yields
\[
\cW_2(P_T \mu_0,P_T \mu_1)
\leq
\cW_2(\mu_0,\mu_1)
+
\frac{2n}{k}
\|\nabla f\|_{L^\infty(M)}
T.
\]
This finishes the proof.
\end{proof}

\section{Applications}
\label{sec:4}
In this section, we present several applications of our main results. In particular, we generalize the results of Aishwarya--Rotem--Shenfeld \cite{aishwarya2025sectionalcurvature} by deriving intrinsic-dimensional evolution variational inequalities and the corresponding Wasserstein contraction estimates for the heat flow under our weighted intermediate Ricci curvature condition. We then compare our framework with that of Ketterer--Mondino \cite{ketterer_sectional_2018}, prove that the two formulations are equivalent in the unweighted setting, and recover their results as a special case.

\subsection{Contraction of Wasserstein distance along heat flow}
\label{sec:4.1}

\begin{flemma}
Let $(M^n,g,e^{-f}d\Vol_g)$ be a smooth complete weighted Riemannian manifold without boundary satisfying $\Ric^N_{k, f} \geq K$. Let
$\mu_0,\mu_1 \in \cPac(M)$, and let
$(P_\tau)_{\tau\geq0}$ be the heat semigroup on $M$. For each $\tau \geq 0$, let $(\mu_s(\tau))_{s\in[0,1]}$ denote the Wasserstein geodesic from
$P_\tau\mu_0$ to $\mu_1$.
Then, for every $T>0$ and for almost every $\tau \in (0,T)$, we have
\begin{align*}
\label{eq:4.1}
&\kern-2em\frac12
\frac{d}{d\tau}
 \mathscr W_2^2(P_\tau \mu_0,\mu_1) \\
&\leq
\frac{1}{\binom{n-1}{k-1}} \int_M  \tr_{\bigwedge^k T_x M} \Bigl(     \Bigl( \dot \tH^{P_\tau \mu_0 \rightarrow \mu_1 }_{0; k, f} (x) \Bigr)^{[k]} \Bigr)  \, d P_\tau \mu_0 (x)  \\
&\kern2em -\frac{n}{k} \int_M \langle \nabla f, \nabla \theta^{P_\tau \mu_0  \rightarrow \mu_1 } \rangle \, d P_\tau \mu_0, \tag{4.1}
\end{align*}where $-\theta^{P_\tau \mu_0 \rightarrow \mu_1}$ denotes the Kantorovich potential from $P_\tau \mu_0$ to $\mu_1$.
\end{flemma}

\begin{proof}
For notational simplicity, for $u \in [0, 1]$, we write
\[
A_u(x)
\coloneqq
\Bigl(
\tH_{u;k,f}^{P_\tau\mu_0\rightarrow\mu_1}(x)
\Bigr)^{[k]}.
\]First, since the heat flow is the Wasserstein gradient flow of the
Boltzmann entropy $H$, its evolution variational inequality yields
\begin{align*}
\label{eq:4.2}
\frac12
\frac{d^+}{d\tau}
\mathscr W_2^2(P_\tau\mu_0,\mu_1)
\le
\frac{d^+}{ds}
H(\mu_s(\tau))
\Big|_{s=0}, \tag{4.2}
\end{align*}
where $\frac{d^+}{dt}$ denotes the upper right derivative.
Since
$\tau\mapsto\mathscr W_2^2(P_\tau\mu_0,\mu_1)$
is absolutely continuous, for every $T>0$ and for almost every
$\tau\in(0,T)$,
\[
\frac{d^+}{d\tau}
\mathscr W_2^2(P_\tau\mu_0,\mu_1)
=
\frac{d}{d\tau}
\mathscr W_2^2(P_\tau\mu_0,\mu_1).
\]

On the other hand, using the relation between $H$ and $H_{k, f}$ together with the first variation formula for $\int_M f d \mu_s(\tau)$, we obtain
\begin{align*}
\label{eq:4.3}
&\kern-2em \frac{d^+}{ds}
H_{k,f}(\mu_s(\tau))
\Big|_{s=0} \\
&=
\frac{d^+}{ds}
\Bigl(
H(\mu_s(\tau))
+
\frac nk
\int_M
f\,d\mu_s(\tau)
\Bigr)
\Big|_{s=0}
\\
&=
\frac{d^+}{ds}
H(\mu_s(\tau))
\Big|_{s=0}
+
\frac nk
\frac{d}{ds}
\int_M
f\,d\mu_s(\tau)
\Big|_{s=0}
\\
&=
\frac{d^+}{ds}
H(\mu_s(\tau))
\Big|_{s=0}
+
\frac nk
\int_M
\bigl\langle
\nabla f,
\nabla\theta^{P_\tau\mu_0\rightarrow\mu_1}
\bigr\rangle
\,dP_\tau\mu_0. \tag{4.3}
\end{align*}

By the definition of upper right derivative and Lemma~\hyperlink{L:3.1}{3.1}, we have
\begin{align*}
\label{eq:4.4}
&\kern-2em
\frac{d^+}{ds}
H_{k,f}(\mu_s(\tau))
\Big|_{s=0}
\\
&=
\limsup_{\delta\downarrow0}
\frac{
H_{k,f}(\mu_\delta(\tau))
-
H_{k,f}(\mu_0(\tau))
}{\delta}
\\
&=
\limsup_{\delta\downarrow0}
\frac1{\binom{n-1}{k-1}}
\int_M
\frac{
\tr_{\bigwedge^kT_xM}
\bigl(
A_\delta(x)
\bigr)
-
\tr_{\bigwedge^kT_xM}
\bigl(
A_0(x)
\bigr)
}{\delta}
\,dP_\tau\mu_0(x). \tag{4.4}
\end{align*}

Furthermore, by Corollary~\hyperlink{C:3.1}{3.1}, we have
\begin{align*}
&\kern-2em \frac{d^2}{dt^2}
\biggl[
A_t(x)
-
\frac K2
t^2
|\nabla\theta^{P_\tau \mu_0 \rightarrow \mu_1}(x)|^2
\Id
\biggr] \\
&=
\Bigl(
\ddot{\tH}_{t;k,f}^{P_\tau\mu_0\rightarrow\mu_1}(x)
\Bigr)^{[k]}
-
K
|\nabla\theta^{P_\tau \mu_0 \rightarrow \mu_1}(x)|^2
\Id
\succeq \frac{1}{N-k} df(\dot \gamma(t))^2 \Id \succeq 0,
\end{align*}
where $\succeq$ denotes the Loewner order.
Hence, we obtain 
\begin{align*}
&\kern-2em
\tr_{\bigwedge^kT_xM}
\bigl(
A_\delta(x)
\bigr)
\\
&\le
\frac{K\delta^2}{2}
\binom nk
|\nabla\theta^{P_\tau \mu_0 \rightarrow \mu_1}(x)|^2
+
(1-\delta)
\tr_{\bigwedge^kT_xM}
\bigl(
A_0(x)
\bigr)
\\
&\kern2em +
\delta
\biggl[
\tr_{\bigwedge^kT_xM}
\bigl(
A_1(x)
\bigr)
-
\frac K2
\binom nk
|\nabla\theta^{P_\tau \mu_0 \rightarrow \mu_1}(x)|^2
\biggr] \\
&= \frac{K\delta^2}{2}
\binom nk
|\nabla\theta^{P_\tau \mu_0 \rightarrow \mu_1}(x)|^2 \\
&\kern2em +
\delta
\biggl[
\tr_{\bigwedge^kT_xM}
\bigl(
A_1(x)
\bigr)
-
\frac K2
\binom nk
|\nabla\theta^{P_\tau \mu_0 \rightarrow \mu_1}(x)|^2
\biggr].
\end{align*}
Since the terms $|\nabla \theta^{P_\tau \mu_0 \rightarrow \mu_1}|^2$ and
\begin{align*}
&\kern-2em \tr_{\bigwedge^kT_xM}
\bigl(
A_1(x)
\bigr)
-
\frac K2
\binom nk
|\nabla\theta^{P_\tau \mu_0 \rightarrow \mu_1}(x)|^2 \\
&=
\tr_{\bigwedge^kT_xM}
\Bigl(
\bigl(
\tH_{1;k,f}^{P_\tau\mu_0\rightarrow\mu_1}(x)
\bigr)^{[k]}
\Bigr)
-
\frac K2
\binom nk
|\nabla\theta^{P_\tau \mu_0 \rightarrow \mu_1}(x)|^2
\end{align*}
are $P_\tau \mu_0$-integrable, by reverse Fatou's lemma, equation \eqref{eq:4.4} becomes
\begin{align*}
\label{eq:4.5}
&\kern-2em
\frac{d^+}{ds}
H_{k,f}(\mu_s(\tau))
\Big|_{s=0}
\\
&\le
\frac1{\binom{n-1}{k-1}}
\int_M
\limsup_{\delta\downarrow0}
\frac{
\tr_{\bigwedge^kT_xM}
\bigl(
A_\delta(x)
\bigr)
-
\tr_{\bigwedge^kT_xM}
\bigl(
A_0(x)
\bigr)
}{\delta}
\,dP_\tau\mu_0(x)
\\
&=
\frac1{\binom{n-1}{k-1}}
\int_M
\tr_{\bigwedge^kT_xM}
\Bigl(
\bigl(
\dot{\tH}_{0;k,f}^{P_\tau\mu_0\rightarrow\mu_1}(x)
\bigr)^{[k]}
\Bigr)
\,dP_\tau\mu_0(x). \tag{4.5}
\end{align*}
By combining inequalities \eqref{eq:4.2}, \eqref{eq:4.3}, and \eqref{eq:4.5}, this finishes the proof.
\end{proof}

\begin{frmk}
Assume that \(\Ric_{k,f}^N\ge K\). Then the estimate
\eqref{eq:4.5} is stronger than \eqref{eq:3.30}. Indeed, by
integrating \eqref{eq:3.16} twice with respect to the interpolation
parameter \(t\), taking the trace over \(\bigwedge^kT_xM\), and then
integrating with respect to the initial measure \(P_\tau\mu_0\), we
obtain
\begin{align*}
&\kern-2em \frac{1}{\binom{n-1}{k-1}}
\int_M
\tr_{\bigwedge^kT_xM}
\Bigl(
\bigl(
\dot{\tH}^{P_\tau \mu_0\rightarrow\mu_1}_{0;k,f}(x)
\bigr)^{[k]}
\Bigr)
\,dP_\tau \mu_0(x) \\
&\kern2em\le
H_{k,f}(\mu_1)
-
H_{k,f}(P_\tau \mu_0)
-
\frac{nK}{2k}
\cW_2^2(P_\tau \mu_0,\mu_1).
\end{align*}
Substituting this estimate into \eqref{eq:4.5} immediately recovers
\eqref{eq:3.30}. Thus, the tensorial formulation implies the scalar
displacement convexity inequality as a direct consequence.
\end{frmk}

From now on, we consider the case that $M$ is a smooth compact manifold without boundary and the case $K = 0$, that is, $\Ric^N_{k, f} \geq 0$. We obtain the following intrinsic dimensional evolution inequality.

\hypertarget{T:4.1}{
\begin{fthm} 
Let $(M^n,g,e^{-f}d\Vol_g)$ be a smooth compact weighted Riemannian manifold without boundary satisfying $\Ric^N_{k, f} \geq 0$. Let $\mu_0,\mu_1 \in \cPac(M)$, and let $(P_\tau)_{\tau\geq0}$ be the heat semigroup on $M$. Then, for every $T>0$ and for almost every
$\tau\in(0,T)$, we have
\begin{align*}
\label{eq:4.6}
&\kern-2em
\frac{d}{d\tau}\mathscr W_2^2(P_\tau\mu_0,\mu_1)
\\
&\le
\frac{2nN}{k}
-
\frac{2N}{\binom{n-1}{k-1}}
\int_M
\sigma_k \biggl(
\exp\biggl(
-\frac{\tH_{1;k,f}^{P_\tau\mu_0\rightarrow\mu_1}(x)}{N}
\biggr)
\biggr)
\,dP_\tau\mu_0(x)
\\
&\kern2em
-
\frac{2n}{k}
\int_M
\left\langle
\nabla f,
\nabla\theta^{P_\tau\mu_0\rightarrow\mu_1}
\right\rangle
\,dP_\tau\mu_0. \tag{4.6}
\end{align*}
\end{fthm}}

\begin{proof}
For notational simplicity, for $u \in [0, 1]$, we write
\[
A_u(x)
\coloneqq
\Bigl(
\tH_{u;k,f}^{P_\tau\mu_0\rightarrow\mu_1}(x)
\Bigr)^{[k]}.
\]
For each $x \in M$, we choose an orthonormal basis $\{e_i(x) \}_{i =1}^n$ of $T_x M$. For each multi-index
$I \coloneqq \{i_1<\cdots<i_k\} \subset\{1,\dots,n\}$, we let
\[
\xi_I(x)
\coloneqq
e_{i_1}(x)\wedge\cdots\wedge e_{i_k}(x).
\]
Then, by Lemma~\hyperlink{L:4.1}{4.1}, for almost every $\tau \in (0, T)$, we have
\begin{align*}
\label{eq:4.7}
&\kern-2em
\frac12\frac{d}{d\tau}
\mathscr W_2^2(P_\tau\mu_0,\mu_1)
\\
&\le
\frac1{\binom{n-1}{k-1}}
\int_M
\sum_{|I|=k}
\bigl\langle
\dot A_0(x)\xi_I(x),
\xi_I(x)
\bigr\rangle
\,dP_\tau\mu_0(x) \\
&\kern2em -
\frac nk
\int_M
\bigl\langle
\nabla f,
\nabla\theta^{P_\tau\mu_0\rightarrow\mu_1}
\bigr\rangle
\,dP_\tau\mu_0. \tag{4.7}
\end{align*}

Fix $x \in M$ and a decomposable unit $k$-vector $\xi=\xi_I(x)$, and define
\[
c_\xi(s)
\coloneqq
\exp \Bigl (   -\frac1N\langle A_s(x)\xi,\xi\rangle      \Bigr)
\]for $s \in [0, 1]$.
By Corollary~\hyperlink{C:3.1}{3.1}, the function $c_\xi$ is concave on $[0, 1]$. Hence, we obtain 
\[
c_\xi'(0)\ge c_\xi(1)-c_\xi(0).
\]
Since $A_0(x)=0$, we have $c_\xi(0)=1$, consequently, we obtain 
\begin{align*}
\label{eq:4.8}
1
-
\exp \Bigl ( {-\frac1N\langle A_1(x)\xi,\xi\rangle} \Bigr)
\ge
\frac1N
\bigl\langle
\dot A_0(x)\xi,
\xi
\bigr\rangle. \tag{4.8}
\end{align*}

Applying estimate \eqref{eq:4.8} to each decomposable unit $k$-vector $\xi_I$ in inequality \eqref{eq:4.7} yields
\begin{align*}
\label{eq:4.9}
&\kern-2em
\frac12\frac{d}{d\tau}
\mathscr W_2^2(P_\tau\mu_0,\mu_1)
\\
&\le
\frac N{\binom{n-1}{k-1}}
\int_M
\sum_{|I|=k}
\Bigl(
1-
\exp \Bigl( {-\frac1N\langle A_1(x)\xi_I,\xi_I\rangle} \Bigr)
\Bigr)
\,dP_\tau\mu_0(x)
\\
&\kern2em
-
\frac nk
\int_M
\bigl\langle
\nabla f,
\nabla\theta^{P_\tau\mu_0\rightarrow\mu_1}
\bigr\rangle
\,dP_\tau\mu_0 \tag{4.9}
\end{align*}for almost every $\tau \in (0, T)$. For each $x \in M$, we now choose $\{e_i(x)\}_{i=1}^n$ to be an orthonormal eigenbasis of
$\tH_{1;k,f}^{P_\tau\mu_0\rightarrow\mu_1}(x)$, with corresponding
eigenvalues $\lambda_1(x),\dots,\lambda_n(x)$. Then, we have
\[
\sum_{|I|=k}
\exp \Bigl({-\frac1N\langle A_1(x)\xi_I(x),\xi_I(x)\rangle} \Bigr)
=
\sigma_k\Bigl(
\exp\Bigl(
-\frac{\tH_{1;k,f}^{P_\tau\mu_0\rightarrow\mu_1}(x) }N
\Bigr)
\Bigr),
\]
and hence we get
\begin{align*}
\label{eq:4.10}
&\kern-2em \sum_{|I|=k}
\Bigl(
1-
\exp \Bigl( {-\frac1N\langle A_1(x)\xi_I(x),\xi_I(x)\rangle} \Bigr)
\Bigr) \\
&=
\binom nk
-
\sigma_k\Bigl(
\exp\Bigl(
-\frac{\tH_{1;k,f}^{P_\tau\mu_0\rightarrow\mu_1}(x)}N
\Bigr)
\Bigr). \tag{4.10}
\end{align*}
By substituting \eqref{eq:4.9} into inequality \eqref{eq:4.10}, for almost every $\tau \in (0, T)$, we obtain 
\begin{align*}
&\kern-2em
\frac12\frac{d}{d\tau}
\mathscr W_2^2(P_\tau\mu_0,\mu_1)
\\
&\le
\frac{nN}{k}
-
\frac N{\binom{n-1}{k-1}}
\int_M
\sigma_k\biggl(
\exp\biggl(
-\frac{\tH_{1;k,f}^{P_\tau\mu_0\rightarrow\mu_1}(x)}N
\biggr)
\biggr)
\,dP_\tau\mu_0(x)
\\
&\kern2em
-
\frac nk
\int_M
\bigl\langle
\nabla f,
\nabla\theta^{P_\tau\mu_0\rightarrow\mu_1}
\bigr\rangle
\,dP_\tau\mu_0.
\end{align*}
This finishes the proof.
\end{proof}

The preceding theorem gives a one-sided differential estimate for the Wasserstein distance between the heat flow $P_\tau\mu_0$ and a fixed target measure $\mu_1$. To obtain a corresponding estimate for the distance between two heat flows, we will apply this result in both transport directions. We therefore first establish the following time-reversal relation between the weighted transport Hessians associated with the forward and reversed Wasserstein geodesics.

\hypertarget{L:4.2}{
\begin{flemma}
Fix $\nu_0,\nu_1\in\cPac(M)$, and let
\[
F_t^{\nu_0\to\nu_1}(x)
\coloneqq
\exp_x\bigl(t\nabla\theta^{\nu_0\to\nu_1}(x)\bigr),
\qquad t\in[0,1],
\]
be the Wasserstein geodesic map from $\nu_0$ to $\nu_1$. Define
\[
\widetilde D(\nu_0,\nu_1)
\coloneqq
\Bigl\{
x\in D(\nu_0,\nu_1)
:
F_1^{\nu_0\to\nu_1}(x)\in D(\nu_1,\nu_0)
\Bigr\}.
\]
Then, after choosing compatible orthonormal parallel frames along the
forward and reversed geodesics, for $\nu_0$-almost every
$x\in \widetilde D(\nu_0,\nu_1)$ and every $t\in[0,1]$, we have
\[
\tH_{t;k,f}^{\nu_1\to\nu_0}
\Bigl(
F_1^{\nu_0\to\nu_1}(x)
\Bigr)
=
\tH_{1-t;k,f}^{\nu_0\to\nu_1}(x)
-
\tH_{1;k,f}^{\nu_0\to\nu_1}(x).
\]
In particular, by taking $t=1$, we have
\begin{align*}
\label{eq:4.11}
\tH_{1;k,f}^{\nu_1\to\nu_0}
\Bigl(
F_1^{\nu_0\to\nu_1}(x)
\Bigr)
=
-
\tH_{1;k,f}^{\nu_0\to\nu_1}(x). \tag{4.11}
\end{align*}
\end{flemma}}

\begin{proof}
Fix $x\in \widetilde D(\nu_0,\nu_1)$ and set
$y\coloneqq F_1^{\nu_0\to\nu_1}(x)$.
For $\nu_0$-almost every such $x$, the optimal transport maps
$F_1^{\nu_0\to\nu_1}$ and $F_1^{\nu_1\to\nu_0}$ are inverse to each other.
Hence, for every $t\in[0,1]$, we have
\begin{align*}
\label{eq:4.12}
F_t^{\nu_0\to\nu_1}(x)
=
F_{1-t}^{\nu_1\to\nu_0}(y). \tag{4.12}
\end{align*}

Now, we choose an orthonormal basis
$e_0^{\nu_0\to\nu_1}(x)$ of $T_xM$ and parallel transport it along
$F_t^{\nu_0\to\nu_1}(x)$ to obtain
$e_t^{\nu_0\to\nu_1}(x)$. For the reversed geodesic, we choose the initial
frame at $y$ by
\[
e_0^{\nu_1\to\nu_0}(y)
\coloneqq
e_1^{\nu_0\to\nu_1}(x),
\]
and parallel transport it along
$F_t^{\nu_1\to\nu_0}(y)$. Then, we have
$e_t^{\nu_1\to\nu_0}(y)
=
e_{1-t}^{\nu_0\to\nu_1}(x)$.

Let $\tJ_t^{\nu_0\to\nu_1}(x)$ and
$\tJ_t^{\nu_1\to\nu_0}(y)$ be the corresponding Jacobi matrices in these
parallel frames. Since the two geodesics are the same curve with opposite
orientation, the curvature matrices satisfy
\[
\tR_t^{\nu_1\to\nu_0}(y)
=
\tR_{1-t}^{\nu_0\to\nu_1}(x).
\]
Moreover, by Claim~3.7 of \cite{aishwarya2025sectionalcurvature}, for any $x \in D(\nu_0, \nu_1)$ satisfying $F_1^{\nu_0 \rightarrow \nu_1}(x) \in D(\nu_1, \nu_0)$ and for all $t \in [0, 1]$, we have

\[
\tJ_t^{\nu_1\to\nu_0}(y)
=
\tJ_{1-t}^{\nu_0\to\nu_1}(x)
\tJ_1^{\nu_1\to\nu_0}(y).
\]

Therefore, we have
\begin{align*}
\label{eq:4.13}
\tU_t^{\nu_1\to\nu_0}(y)
&=
\dot{\tJ}_t^{\nu_1\to\nu_0}(y)
\Bigl[
\tJ_t^{\nu_1\to\nu_0}(y)
\Bigr]^{-1}
\\
&=
-
\dot{\tJ}_{1-t}^{\nu_0\to\nu_1}(x)
\tJ_1^{\nu_1\to\nu_0}(y)
\Bigl[
\tJ_{1-t}^{\nu_0\to\nu_1}(x)
\tJ_1^{\nu_1\to\nu_0}(y)
\Bigr]^{-1}
\\
&=
-
\dot{\tJ}_{1-t}^{\nu_0\to\nu_1}(x)
\Bigl[
\tJ_{1-t}^{\nu_0\to\nu_1}(x)
\Bigr]^{-1}
=
-\tU_{1-t}^{\nu_0\to\nu_1}(x).
\tag{4.13}
\end{align*}

Using the convention
$\tH_t^{\nu_0\to\nu_1}(x)
=
-\int_0^t
\tU_s^{\nu_0\to\nu_1}(x)
\,ds$,
equation \eqref{eq:4.13} gives
\begin{align*}
\label{eq:4.14}
\tH_t^{\nu_1\to\nu_0}(y)
&=
-\int_0^t
\tU_s^{\nu_1\to\nu_0}(y)
\,ds
=
\int_0^t
\tU_{1-s}^{\nu_0\to\nu_1}(x)
\,ds =
\int_{1-t}^1
\tU_r^{\nu_0\to\nu_1}(x)
\,dr
\\
&=
\tH_{1-t}^{\nu_0\to\nu_1}(x)
-
\tH_1^{\nu_0\to\nu_1}(x).
\tag{4.14}
\end{align*}

Finally, by definition of the weighted quantity, we have
\begin{align*}
\label{eq:4.15}
\tH_{t;k,f}^{\nu_0\to\nu_1}(x)
=
\tH_t^{\nu_0\to\nu_1}(x)
+
\frac{
f\bigl(F_t^{\nu_0\to\nu_1}(x)\bigr)-f(x)
}{k}\Id. \tag{4.15}
\end{align*}
By equations \eqref{eq:4.12}, \eqref{eq:4.14}, and \eqref{eq:4.15},
we get
\begin{align*}
\tH_{t;k,f}^{\nu_1\to\nu_0}(y)
&=
\tH_t^{\nu_1\to\nu_0}(y)
+
\frac{
f\bigl(F_t^{\nu_1\to\nu_0}(y)\bigr)-f(y)
}{k}\Id
\\
&=
\tH_{1-t}^{\nu_0\to\nu_1}(x)
-
\tH_1^{\nu_0\to\nu_1}(x)
+
\frac{
f\bigl(F_{1-t}^{\nu_0\to\nu_1}(x)\bigr)
-
f\bigl(F_1^{\nu_0\to\nu_1}(x)\bigr)
}{k}\Id
\\
&=
\tH_{1-t;k,f}^{\nu_0\to\nu_1}(x)
-
\tH_{1;k,f}^{\nu_0\to\nu_1}(x).
\end{align*}
This finishes the proof.
\end{proof}

We are now ready to combine Theorem~\hyperlink{T:4.1}{4.1} and Lemma~\hyperlink{L:4.2}{4.2} to derive the following intrinsic-dimensional Wasserstein contraction along heat flows, which reflects the effective dimension underlying the transport of probability measures more precisely than the classical formulation in \cite{von_renesse_transport_2005, bolley2015equivalence}.

\hypertarget{C:4.1}{
\begin{fcor}
    Let $(M^n,g,e^{-f}d\Vol_g)$ be a smooth compact weighted Riemannian manifold without boundary satisfying $\Ric^N_{k, f} \geq 0$. Let $\mu_0,\mu_1 \in \cPac(M)$, and let $(P_\tau)_{\tau\geq0}$ be the heat semigroup on $M$. Then, for every $T > 0$, we have
    \begin{align*}
    \label{eq:4.16}
        &\kern-2.5em \cW^2_2(P_T \mu_0, P_T \mu_1) - \Bigl[ \cW_2(\mu_0, \mu_1) +\frac{2nT}{k} \| \nabla f \|_{L^\infty} \Bigr]^2 \\
        &\kern-0.5em \leq   - \frac{8N}{\binom{n-1}{k-1}} \int_0^T \int_M  \tr_{\bigwedge^k T_x M} \biggl[  \sinh^2 \biggl( \frac{(\tH_{1; k, f}^{P_\tau \mu_0 \rightarrow P_\tau \mu_1}(x))^{[k]}}{2N}    \biggr) \biggr] \, d P_\tau \mu_0(x) \, d \tau. \tag{4.16}
    \end{align*}
\end{fcor}}


\begin{proof}
First, we apply Theorem~\hyperlink{T:4.1}{4.1} and Lemma~\hyperlink{L:4.2}{4.2} to the probability measures $(\mu_1, \mu_0)$, for almost every $\tau \in (0, T)$, we obtain
\begin{align*}
\label{eq:4.17}
&\kern-2em
\frac{d}{d\tau}
\mathscr W_2^2(P_\tau\mu_1,\mu_0)
\\
&\leq  \frac{2nN}{k} - \frac{2N}{\binom{n-1}{k-1}} \int_M   \sigma_k\biggl(\exp \biggl({- \frac{\tH_{1; k, f}^{P_\tau \mu_1 \rightarrow \mu_0}(x)}{N}} \biggr)\biggr)  \,  d P_\tau \mu_1(x)  \\
        &\kern2em -\frac{2n}{k} \int_M \langle \nabla f, \nabla \theta^{P_\tau \mu_1 \rightarrow \mu_0} \rangle  \, d P_\tau \mu_1 \\
&=
\frac{2nN}{k}- \frac{2N}{\binom{n-1}{k-1}}
\int_M
\sigma_k\biggl(
\exp\biggl(
\frac{
\tH_{1;k,f}^{\mu_0\rightarrow P_\tau\mu_1}(x)
}{N}
\biggr)
\biggr)
\,d\mu_0(x)
\\
&\kern2em -
\frac{2n}{k}
\int_M
\bigl\langle
\nabla f,
\nabla\theta^{P_\tau\mu_1\rightarrow\mu_0}
\bigr\rangle
\,dP_\tau\mu_1
\\
&\le
\frac{2nN}{k}- \frac{2N}{\binom{n-1}{k-1}}
\int_M
\sigma_k\biggl(
\exp\biggl(
\frac{
\tH_{1;k,f}^{\mu_0\rightarrow P_\tau\mu_1}(x)
}{N}
\biggr)
\biggr)
\,d\mu_0(x)
\\
&\kern2em
+
\frac{2n}{k}
\|\nabla f\|_{L^\infty}
\mathscr W_2(P_\tau\mu_1,\mu_0). \tag{4.17}
\end{align*}
Replacing $\mu_0$ by $P_s \mu_0$ in inequality \eqref{eq:4.17}, we obtain 
\begin{align*}
\label{eq:4.18}
     &\kern-2em \frac{d}{d\tau} \cW_2^2 (P_\tau \mu_1, P_s \mu_0) \\ 
    &\leq\frac{2nN}{k}- \frac{2N}{\binom{n-1}{k-1}}
\int_M
\sigma_k\biggl(
\exp\biggl(
\frac{
\tH_{1;k,f}^{P_s \mu_0\rightarrow P_\tau\mu_1}(x)
}{N}
\biggr)
\biggr)
\,dP_s \mu_0(x)
\\
&\kern2em
+
\frac{2n}{k}
\|\nabla f\|_{L^\infty}
\mathscr W_2(P_\tau\mu_1, P_s\mu_0). \tag{4.18}
\end{align*}By interchanging the roles of $\mu_0$ and $\mu_1$ and using the symmetry of $\cW_2$, inequality \eqref{eq:4.18} becomes
\begin{align*}
\label{eq:4.19}
     &\kern-2em \frac{d}{d\tau} \cW_2^2 (P_\tau \mu_0, P_s \mu_1) \\ 
    &\leq\frac{2nN}{k}- \frac{2N}{\binom{n-1}{k-1}}
\int_M
\sigma_k\biggl(
\exp\biggl(
\frac{
\tH_{1;k,f}^{P_s \mu_1\rightarrow P_\tau\mu_0}(x)
}{N}
\biggr)
\biggr)
\,dP_s \mu_1(x)
\\
&\kern2em
+
\frac{2n}{k}
\|\nabla f\|_{L^\infty}
\mathscr W_2(P_\tau\mu_0, P_s\mu_1). \tag{4.19}
\end{align*}
By summing inequalities \eqref{eq:4.18} and \eqref{eq:4.19}, we obtain
\begin{align*}
\label{eq:4.20}
    &\kern-2em \frac{d}{d\tau} \cW_2^2 (P_\tau \mu_1, P_s \mu_0) +   \frac{d}{d\tau} \cW_2^2 (P_\tau \mu_0,  P_s \mu_0) \\
    &= \frac{d}{d\tau} \cW_2^2 (P_s \mu_0, P_\tau \mu_1 ) +   \frac{d}{d\tau} \cW_2^2 (P_\tau \mu_0,  P_s \mu_0) \\
    &\leq \frac{4nN}{k}
- \frac{2N}{\binom{n-1}{k-1}}
\int_M
\sigma_k\biggl(
\exp\biggl(
\frac{
\tH_{1;k,f}^{P_s \mu_0\rightarrow P_\tau\mu_1}(x)
}{N}
\biggr)
\biggr)
\,dP_s \mu_0(x) \\
&\kern2em - \frac{2N}{\binom{n-1}{k-1}}
\int_M
\sigma_k\biggl(
\exp\biggl(
\frac{
\tH_{1;k,f}^{P_s \mu_1\rightarrow P_\tau\mu_0}(x)
}{N}
\biggr)
\biggr)
\,dP_s \mu_1(x)
\\
        &\kern2em +\frac{2n}{k} \| \nabla f \|_{L^\infty} \bigl( \cW_2(P_\tau \mu_1,  P_s \mu_0)  + \cW_2(P_\tau \mu_0, P_s \mu_1)\bigr). \tag{4.20}
\end{align*}
By setting $s = \tau$ and applying the chain rule, inequality \eqref{eq:4.20} becomes
\begin{align*}
    &\kern-2em \frac{d}{d\tau} \cW_2^2 (P_\tau \mu_1, P_\tau \mu_0)  \\
     &\leq \frac{4nN}{k}- \frac{2N}{\binom{n-1}{k-1}}
\int_M
\sigma_k\biggl(
\exp\biggl(
\frac{
\tH_{1;k,f}^{P_\tau \mu_1\rightarrow P_\tau\mu_0}(x)
}{N}
\biggr)
\biggr)
\,dP_\tau \mu_1(x)
\\
&\kern2em - \frac{2N}{\binom{n-1}{k-1}}
\int_M
\sigma_k\biggl(
\exp\biggl(
\frac{
\tH_{1;k,f}^{P_\tau \mu_0\rightarrow P_\tau\mu_1}(x)
}{N}
\biggr)
\biggr)
\,dP_\tau \mu_0(x)
\\
        &\kern2em +\frac{4n}{k} \| \nabla f \|_{L^\infty}  \cW_2(P_\tau \mu_0, P_\tau \mu_1) \\
&= \frac{4nN}{k}- \frac{2N}{\binom{n-1}{k-1}}
\int_M
\sigma_k\biggl(
\exp\biggl(
- \frac{
\tH_{1;k,f}^{P_\tau \mu_0 \rightarrow P_\tau\mu_1}(x)
}{N}
\biggr)
\biggr)
\,dP_\tau \mu_0(x)
\\
&\kern2em - \frac{2N}{\binom{n-1}{k-1}}
\int_M
\sigma_k\biggl(
\exp\biggl(
\frac{
\tH_{1;k,f}^{P_\tau \mu_0\rightarrow P_\tau\mu_1}(x)
}{N}
\biggr)
\biggr)
\,dP_\tau \mu_0(x) \\
&\kern2em +\frac{4n}{k} \| \nabla f \|_{L^\infty}  \cW_2(P_\tau \mu_0, P_\tau \mu_1) \\ 
&= -  \frac{8N}{\binom{n-1}{k-1}}  \int_M  \tr_{\bigwedge^k T_x M} \biggl[  \sinh^2 \biggl( \frac{(\tH_{1; k, f}^{P_\tau \mu_0 \rightarrow P_\tau \mu_1}(x))^{[k]} }{2N}    \biggr) \biggr] \, d P_\tau \mu_0(x)  \\
&\kern2em +\frac{4n}{k} \| \nabla f \|_{L^\infty}  \cW_2(P_\tau \mu_0, P_\tau \mu_1).
\end{align*}


Finally, since the function
$\tau\mapsto\cW_2(P_\tau\mu_0,P_\tau\mu_1)$
is absolutely continuous, 
Lemma~\hyperlink{L:2.2}{2.2}
applied to the above differential inequality yields
\begin{align*}
        &\kern-2em \cW^2_2(P_T \mu_0, P_T \mu_1) \\
        &\leq \Bigl[ \cW_2(\mu_0, \mu_1) +\frac{2nT}{k} \| \nabla f \|_{L^\infty} \Bigr]^2 \\
        &\kern2em - \frac{8N}{\binom{n-1}{k-1}} \int_0^T \int_M  \tr_{\bigwedge^k T_x M} \biggl[  \sinh^2 \biggl( \frac{(\tH_{1; k, f}^{P_\tau \mu_0 \rightarrow P_\tau \mu_1}(x))^{[k]} }{2N}    \biggr) \biggr] \, d P_\tau \mu_0(x) \, d \tau.
    \end{align*}This finishes the proof.
\end{proof}


\hypertarget{R:4.1}{
\begin{frmk}
When $f$ is constant, the drift term in Corollary~\hyperlink{C:4.1}{4.1} vanishes, and the evolution variational inequality reduces to the following contraction estimate for the heat flow:
\begin{align*}
    &\kern-2em
    \cW_2^2(P_T\mu_0,P_T\mu_1)
    -
    \cW_2^2(\mu_0,\mu_1)
    \\
    &\le
    -
    \frac{8k}{\binom{n-1}{k-1}}
    \int_0^T
    \int_M
    \operatorname{tr}_{\bigwedge^kT_xM}
    \biggl[
    \sinh^2
    \biggl(
    \frac{
    \bigl(\tH_{1}^{P_\tau\mu_0\rightarrow P_\tau\mu_1} (x) \bigr)^{[k]}
    }{2k}
    \biggr)
    \biggr]
    \,dP_\tau\mu_0(x)\,d\tau.
\end{align*}
In particular, the Wasserstein distance between the two heat flows is nonincreasing in time. Moreover, the right-hand side provides a quantitative estimate for the dissipation of the Wasserstein distance. The rate of contraction is quantified by the tensor
\[
\bigl(\tH_{1;k,f}^{P_\tau\mu_0\rightarrow P_\tau\mu_1}\bigr)^{[k]},
\]
which is naturally associated with the optimal transport from $P_\tau\mu_0$ to $P_\tau\mu_1$.
\end{frmk}}

\hypertarget{P:4.1}{
\begin{fprop}
Let $(M^n,g)$ be a smooth compact Riemannian manifold without boundary, and
assume that
$\Ric_k\ge0$
for some $k\in\{1,\ldots,n\}$. Suppose that there exist
$\mu_0,\mu_1\in\cPac(M)$ and $T>0$ such that the following hold:
\begin{enumerate}
\item[(i)]
For almost every $\tau\in(0,T)$, if
$-\theta_\tau$ is a Kantorovich potential inducing the optimal transport
from $P_\tau\mu_0$ to $P_\tau\mu_1$, then
\[
\bigl(\Hess\theta_\tau(x)\bigr)^{[k]}\le0
\]
for $P_\tau\mu_0$-almost every
$x\in D(P_\tau\mu_0,P_\tau\mu_1)$.
\item[(ii)]
$\cW_2(P_T\mu_0,P_T\mu_1)
=
\cW_2(\mu_0,\mu_1)$.
\end{enumerate}
Then, for almost every $\tau\in(0,T)$, for
$P_\tau\mu_0$-almost every $x$, for every
$k$-dimensional subspace
$\Sigma\subset T_xM$, and every
$s\in[0,1]$,
\[
\Ric_k\bigl(\Sigma_s,\dot\gamma_x^\tau(s)\bigr)=0,
\]
where
\[
\gamma_x^\tau(s)
\coloneqq
F_s^{P_\tau\mu_0\rightarrow P_\tau\mu_1}(x),
\]
and $\Sigma_s$ denotes the parallel transport of $\Sigma$ along
$\gamma_x^\tau$.
\end{fprop}}

\begin{proof}
By Remark~\hyperlink{R:4.1}{4.1}, with $N=k$ in the unweighted case, we
have
\begin{align*}
    &\kern-2em
    \cW_2^2(P_T\mu_0,P_T\mu_1)
    -
    \cW_2^2(\mu_0,\mu_1)
    \\
    &\le
    -
    \frac{8k}{\binom{n-1}{k-1}}
    \int_0^T
    \int_M
    \operatorname{tr}_{\bigwedge^kT_xM}
    \biggl[
    \sinh^2
    \biggl(
    \frac{
    \bigl(
    \tH_1^{P_\tau\mu_0\rightarrow P_\tau\mu_1}(x)
    \bigr)^{[k]}
    }{2k}
    \biggr)
    \biggr]
    \,dP_\tau\mu_0(x)\,d\tau.
\end{align*}
The integrand on the right-hand side is nonnegative. Since
\[
\cW_2(P_T\mu_0,P_T\mu_1)
=
\cW_2(\mu_0,\mu_1),
\]
the preceding inequality implies that
\begin{align*}
    \int_0^T
    \int_M
    \operatorname{tr}_{\bigwedge^kT_xM}
    \biggl[
    \sinh^2
    \biggl(
    \frac{
    \bigl(
    \tH_1^{P_\tau\mu_0\rightarrow P_\tau\mu_1}(x)
    \bigr)^{[k]}
    }{2k}
    \biggr)
    \biggr]
    \,dP_\tau\mu_0(x)\,d\tau
    =
    0.
\end{align*}
Consequently, for almost every $\tau\in(0,T)$,
\begin{align*}
    \int_M
    \operatorname{tr}_{\bigwedge^kT_xM}
    \biggl[
    \sinh^2
    \biggl(
    \frac{
    \bigl(
    \tH_1^{P_\tau\mu_0\rightarrow P_\tau\mu_1}(x)
    \bigr)^{[k]}
    }{2k}
    \biggr)
    \biggr]
    \,dP_\tau\mu_0(x)
    =
    0.
\end{align*}
For each such $\tau$, the nonnegativity of the integrand gives
\[
\operatorname{tr}_{\bigwedge^kT_xM}
\biggl[
\sinh^2
\biggl(
\frac{
\bigl(
\tH_1^{P_\tau\mu_0\rightarrow P_\tau\mu_1}(x)
\bigr)^{[k]}
}{2k}
\biggr)
\biggr]
=
0
\]
for $P_\tau\mu_0$-almost every $x$. Since the operator inside the trace
is nonnegative and $\sinh t=0$ if and only if $t=0$, it follows that
\[
\bigl(
\tH_1^{P_\tau\mu_0\rightarrow P_\tau\mu_1}(x)
\bigr)^{[k]}
=
0
\]
for $P_\tau\mu_0$-almost every $x$.

Fix such a time $\tau$ and a point $x$ belonging to the corresponding
full-measure set. For simplicity, we write
$\gamma(s)\coloneqq\gamma_x^\tau(s)$
and
$\tU_s(x)
\coloneqq
\tU_s^{P_\tau\mu_0\rightarrow P_\tau\mu_1}(x)$.
By the definition of $\tH_1$, we have
\[
\tH_1^{P_\tau\mu_0\rightarrow P_\tau\mu_1}(x)
=
-\int_0^1\tU_s(x)\,ds.
\]

Let $\Sigma\subset T_xM$ be a $k$-dimensional subspace, and let
$\xi_\Sigma\in\bigwedge^kT_xM$ be a unit decomposable $k$-vector
representing $\Sigma$. The identity
$\bigl(
\tH_1^{P_\tau\mu_0\rightarrow P_\tau\mu_1}(x)
\bigr)^{[k]}
=
0$
implies that
\begin{align*}
\label{eq:4.21}
&\kern-2em \int_0^1
\tr \bigl(
\tU_s(x)\big|_\Sigma
\bigr)
\,ds \\
&= -  \tr \bigl(
\tH_1^{P_\tau\mu_0\rightarrow P_\tau\mu_1}(x)
\big|_\Sigma
\bigr) = - \bigl\langle
\bigl(
\tH_1^{P_\tau\mu_0\rightarrow P_\tau\mu_1}(x)
\bigr)^{[k]}
\xi_\Sigma,
\xi_\Sigma
\bigr\rangle = 0. \tag{4.21}
\end{align*}

Define
$m_\Sigma(s)
\coloneqq
\operatorname{tr}
\bigl(
\tU_s(x)\big|_\Sigma
\bigr)$.
Let $\Sigma_s$ denotes the parallel transport of $\Sigma$ along $\gamma$,
then by taking the trace of the Riccati equation
$\dot{\tU}_s+\tU_s^2+\tR_s=0$
over $\Sigma$, we obtain
\[
m_\Sigma'(s)
=
-
\operatorname{tr}
\bigl(
\tU_s(x)^2\big|_\Sigma
\bigr)
-
\Ric_k\bigl(\Sigma_s,\dot\gamma(s)\bigr).
\]
By the Cauchy--Schwarz inequality, we have
\[
\operatorname{tr}
\bigl(
\tU_s(x)^2\big|_\Sigma
\bigr)
\ge
\frac{1}{k}
\biggl[
\operatorname{tr}
\bigl(
\tU_s(x)\big|_\Sigma
\bigr)
\biggr]^2
=
\frac{1}{k}m_\Sigma(s)^2.
\]
Since $\Ric_k\ge 0$, it follows that
\[
m_\Sigma'(s)
\le
-\frac{1}{k}m_\Sigma(s)^2
-
\Ric_k\bigl(\Sigma_s,\dot\gamma(s)\bigr)
\le 0.
\]
Thus, $m_\Sigma$ is nonincreasing.

Moreover, by the hypothesis on the Kantorovich potential $-\theta_\tau$,
\[
m_\Sigma(0)
=
\operatorname{tr}
\bigl(
\Hess\theta_\tau(x)\big|_\Sigma
\bigr)
\le 0.
\]
Therefore,
\[
m_\Sigma(s)\le 0
\qquad
\text{for every }s\in[0,1].
\]
On the other hand, \eqref{eq:4.21} gives
\[
\int_0^1m_\Sigma(s)\,ds=0.
\]
Since $m_\Sigma$ is nonpositive, we conclude that
\[
m_\Sigma(s)=0
\qquad
\text{for every }s\in[0,1].
\]
Substituting this identity into the traced Riccati equation yields
\[
0
=
-
\operatorname{tr}
\bigl(
\tU_s(x)^2\big|_\Sigma
\bigr)
-
\Ric_k\bigl(\Sigma_s,\dot\gamma(s)\bigr).
\]
Both terms on the right-hand side are nonpositive, while
$\tr\bigl(
\tU_s(x)^2\big|_\Sigma
\bigr)\ge 0$
and
$\Ric_k\bigl(\Sigma_s,\dot\gamma(s)\bigr)\ge 0$.
Hence, both terms vanish, and in particular
\[
\Ric_k\bigl(\Sigma_s,\dot\gamma(s)\bigr)=0
\qquad
\text{for every }s\in[0,1].
\]
This finishes the proof.
\end{proof}

\begin{frmk}
Suppose that, for every $p\in M$ and every $v\in T_pM$, the family of
transported $k$-dimensional subspaces arising in Proposition~\hyperlink{P:4.1}{4.1} is
sufficiently rich to determine symmetric endomorphisms of $T_pM$; to be more precise,
whenever a symmetric endomorphism
$A:T_pM\to T_pM$ satisfies
\[
\operatorname{tr}
\bigl(
\pi_\Sigma\circ A|_\Sigma
\bigr)
=
0
\]
for every such transported $k$-dimensional subspace $\Sigma$, then necessarily
$A=0$.
Since
\[
\Ric_k(\Sigma,v)
=
\operatorname{tr}
\bigl(
\pi_\Sigma\circ(\Riem(\bullet,v)v)\big|_\Sigma
\bigr),
\]
the preceding proposition then implies that
$\Ric_k\equiv0$.
\end{frmk}

\subsection{Comparison with Ketterer--Mondino}
\label{sec:4.2}
In this section, we compare our results with those of Ketterer and Mondino \cite{ketterer_sectional_2018} in the unweighted setting.
They considered a countably $\mathcal{H}^k$-rectifiable Wasserstein geodesic 
\begin{align*}
\mu^{(k)}_t=\rho^{(k)}_t \mathcal{H}^k\llcorner\Sigma_t\in \mathcal{P}_c(M,\mathcal{H}^k), \qquad t \in [0, 1],
\end{align*}where each $\Sigma_t$ is a countably $\mathcal{H}^k$-rectifiable set. The geodesic
$(\mu^{(k)}_t)_{t \in [0, 1]}$ is induced by a family of transport maps $T_{s,t}$ satisfying 
\begin{align*}
\mu^{(k)}_s=(T_{s,t})_\sharp\mu^{(k)}_t.
\end{align*}
After modifying on a null set, one may assume that $T_{s,t}$ is Lipschitz.
Moreover, for $\cH^k$-almost every $x\in \Sigma_0$, the curve $\gamma_x(t) \coloneqq T_{t,0}(x)$ is a minimizing geodesic.

\medskip

In the following, let $\mu_0, \mu_1\in \cPac(M)$ and let $-\theta$ be a Kantorovich potential such that $F \coloneqq \exp(\nabla \theta)$ is the optimal transport map from $\mu_0$ to $\mu_1$.
For $t \in [0, 1]$, we define $F_t=\exp(t\nabla \theta)$ and $\mu_t \coloneqq (F_t)_\sharp \mu_0$,
then $( \mu_t )_{t \in [0, 1]}$ is the Wasserstein geodesic from $\mu_0$ to $\mu_1$.
For the comparison, let $\Sigma_0\subset \supp \mu_0$ be a countably $\mathcal{H}^k$-rectifiable set and define
\begin{align*}
\label{eq:4.22}
    \mu^{(k)}_0&\coloneqq \frac{1}{\int_{\Sigma_0}\rho_0 d\mathcal{H}^k}\rho_0 \, \mathcal{H}^k\llcorner\Sigma_0, \tag{4.22}\\
    \label{eq:4.23}
    \Sigma_t &\coloneqq F_t(\Sigma_0), \tag{4.23} \\
    \label{eq:4.24}
    \mu^{(k)}_t&\coloneqq (F_t)_\sharp\mu^{(k)}_0. \tag{4.24}
\end{align*}
Without loss of generality, we may assume that $(\mu^{(k)}_t)$ is then a countably $\mathcal{H}^k$-rectifiable Wasserstein geodesic.
The key point is that the Wasserstein geodesics $(\mu_t)_{t \in [0, 1]}$ and $(\mu^{(k)}_t)_{t \in [0, 1]}$ are induced by the same family of transport maps $(F_t)_{t \in [0, 1]}$. We emphasize, however, that unlike the notation used in the previous sections, the sets $\Sigma_t$ here are not obtained by parallel transport of a fixed $k$-dimensional subspace along the geodesic.

\medskip

For our computation, we define $T_{t} \coloneqq F_t|_{\Sigma_0}$. Fix a point $x\in \Sigma_0$, and consider the minimizing geodesic $\gamma(t)=F_t(x)=T_t(x)$ for $t\in [0,1]$.
Let $\{e_\alpha(t)\}_{\alpha=1}^n$ be a parallel orthonormal frame along $\gamma$, and
let $\{E_i(t)\}_{i=1}^k$ be an orthonormal frame of $T_{\gamma(t)}\Sigma_t$ satisfying $\nabla_t E_i(t)\perp T_{\gamma(t)}\Sigma_t$.
For each $t\in[0,1]$, define the $n \times k$ matrix $\tW(t)$ by
\begin{align*}
\label{eq:4.25}
   \tW_{\alpha i}(t) = g(E_i(t), e_\alpha(t)), \tag{4.25}
\end{align*}
so that $\tW(t)$ represents the inclusion map $T_{\gamma(t)}\Sigma_t\hookrightarrow T_{\gamma(t)}M$ with respect to the bases $\{E_i(t)\}_{i = 1}^k$ and $\{e_\alpha(t)\}_{\alpha =1}^n$.
Consequently, the transpose $\tW^t(t)\in M_{k\times n}$ is the matrix representation of the orthogonal projection from $T_{\gamma(t)}M$ onto $T_{\gamma(t)}\Sigma_t$.
A direct computation shows that for any $t \in [0, 1]$, we have
\begin{align*}
\begin{cases}
    \tW^t(t) \tW(t) &= \Id_k,\\
    \dot{\tW}^t(t) \tW(t)&=0, \\
    \tW^t(t)\dot{\tW}(t) &= 0. 
\end{cases}
\end{align*}
Next, let $J(t) \colon T_xM\to T_{\gamma(t)}M$ denote the Jacobi field map. Namely, for each $v\in T_xM$, the vector field $J(t)v$ is the unique Jacobi field along $\gamma$ satisfying 
\begin{align*}
    (Jv)(0)=v \quad \text{ and } \quad (Jv)'(0)=\Hess \theta(v).
\end{align*}
We use $\tJ(t)$ to denote the corresponding $n\times n$ matrix with respect to the bases $\{e_\alpha(0)\}_{\alpha = 1}^n$ and $\{e_\alpha(t) \}_{\alpha =1}^n$, as in the previous sections.
Since $\gamma$ is minimizing, $\tJ(t)$ is invertible for every $t \in [0, 1]$. Following \cite{ketterer_sectional_2018}, we define
\begin{align*}
\label{eq:4.26}
    \tB(t) \coloneqq \tJ(t) \tW_0\in M_{n\times k}, \tag{4.26}
\end{align*}
where $\tW_0 \coloneqq \tW(0)$. By the definition of $\Sigma_t$, the image of $\tB(t)$ coincides with the image of $\tW(t)$.

\medskip

From now on, for notational simplicity, we write $\tW=\tW(t)$, $\tJ=\tJ(t)$, and $\tB=\tB(t)$ for each $t \in [0, 1]$. Unless otherwise specified, we suppress the dependence on $t$ whenever no confusion can arise. 

\hypertarget{L:4.3}{
\begin{flemma}
    The $k \times k$ matrix 
    \begin{align*}
        \tW^{\, t} \tB=\tW^{ \, t} \tJ \tW_0\in M_{k \times k}(\mathbb R)
    \end{align*} is invertible. Moreover, we have
    \begin{align*}
        (\tW^{\, t} \tB)^{-1}= \tW_0^t \tJ^{\, -1} \tW.
    \end{align*}
\end{flemma}}

\begin{proof}
    First, since the image of $\tB(t)$ coincides with the image of $\tW(t)$ and $\tW \tW^t$ is the orthogonal projection onto the image of $\tW(t)$, we have 
    \begin{align*}
    \label{eq:4.27}
    \tW \tW^t \tB = \tB.      
    \tag{4.27}
    \end{align*}
    Thus, by definition \eqref{eq:4.26}, we obtain 
    \begin{align*}
        \tW_0^t \, \tJ^{-1} \tW \tW^t \tB=\tW_0^t \, \tJ^{-1} \tB= \tW^t_0 \, \tJ^{-1} \tJ \tW_0=\Id_k.
    \end{align*}Hence, $\tW_0^t \, \tJ^{-1} \tW$ is the inverse of the square matrix $\tW^t \tB$. This finishes the proof. 
\end{proof}

We next define the $n\times k$ matrix
\begin{align*}
\label{eq:4.28}
\mathcal{U}
\coloneqq
\dot{\tB} \, (\tW^t \tB)^{-1}
=
\dot{\tJ} \tW_0 \tW_0^t \, \tJ^{-1} \tW.
\tag{4.28}
\end{align*}

\hypertarget{L:4.4}{
\begin{flemma}
    We have 
    \begin{align*}
    \label{eq:4.29}
        \mathcal{U}=\dot{\tJ} \tJ^{-1} \tW.
        \tag{4.29}
    \end{align*}
\end{flemma}}

\begin{proof}
By equation \eqref{eq:4.27} and Lemma~\hyperlink{L:4.3}{4.3}, we have
\begin{align*}
\tJ \tW_0 \tW_0^t \, \tJ^{-1} \tW
=
\tB \, (\tW^t \tB)^{-1}
=
 \tW \tW^t \tB \, (\tW^t \tB)^{-1}
=
\tW.
\end{align*}
Multiplying both sides on the left by $\tJ^{-1}$, we obtain
\begin{align*}
\label{eq:4.30}
\tW_0 \tW_0^t \, \tJ^{-1} \tW
=
\tJ^{-1}\tW.
\tag{4.30}
\end{align*}
Substituting identity \eqref{eq:4.30} into the definition \eqref{eq:4.28} of $\mathcal U$ gives
\[
\mathcal U
=
\dot{\tJ} \tW_0 \tW_0^t \, \tJ^{-1} \tW
=
\dot{\tJ} \tJ^{-1} \tW
\]
This finishes the proof. 
\end{proof}

Our goal is to compute the derivative of $\mathcal{U}$, we first have the following lemma.

\begin{flemma}
We have
\begin{align*}
\label{eq:4.31}
\dot{\tW}
=
(\Id_n- \tW \tW^{\, t}) \, \mathcal U
=
\mathcal U^\perp,
\tag{4.31}
\end{align*}where $\mathcal U^\perp$ is the normal component of $\mathcal U$ with respect to the orthogonal decomposition
$T_{\gamma(t)}M
=
T_{\gamma(t)}\Sigma_t
\oplus
T_{\gamma(t)}\Sigma_t^\perp$.
Consequently, we have
\begin{align*}
\label{eq:4.32}
\dot{\tW}^{\, t} \mathcal U
=
\dot{\tW}^{\, t} \dot{\tW}.
\tag{4.32}
\end{align*}
\end{flemma}


\begin{proof}
By differentiating the identity \eqref{eq:4.27}, we obtain
\begin{align*}
\label{eq:4.33}
\dot{\tW} \tW^t \tB
+
\tW \dot{\tW}^t \tB
+
\tW \tW^t \dot{\tB}
=
\dot{\tB}.
\tag{4.33}
\end{align*}
Since
$\tB= \tW \tW^t \tB$,
the second term in \eqref{eq:4.33} can be written as
$\tW \dot{\tW}^t \tW \tW^t \tB$,
which vanishes because
$\dot{\tW}^t \tW=0$.
Therefore, we get 
\begin{align*}
\label{eq:4.34}
\dot{\tW} \tW^t \tB
+
\tW \tW^t \dot{\tB}
=
\dot{\tB}.    
\tag{4.34}
\end{align*}
Since $\tW^t \tB$ is invertible, multiplying equation \eqref{eq:4.34} on the right by
$(\tW^t \tB)^{-1}$ yields
\[
\dot{\tW}
=
(\Id_n- \tW \tW^t) \, \dot{\tB} \, (\tW^t \tB)^{-1}
=
(\Id_n-\tW \tW^t) \, \mathcal U.
\]
Since
$\Id_n- \tW \tW^t$
is the orthogonal projection onto the normal space
$T_{\gamma(t)}\Sigma_t^\perp
\subset
T_{\gamma(t)}M$,
the first identity \eqref{eq:4.31} follows.

\medskip

Finally, by \eqref{eq:4.31}, we have 
\begin{align*}
\label{eq:4.35}
\mathcal U
=
\tW \tW^t\mathcal U
+
(\Id_n- \tW \tW^t)\, \mathcal U
=
\tW \tW^t\mathcal U+\dot{\tW}.
\tag{4.35}
\end{align*}
By multiplying equation \eqref{eq:4.35} on the left by $\dot{\tW}^t$ and using
$\dot{\tW}^t \tW=0$,
we obtain
\[
\dot{\tW}^t \mathcal U
=
\dot{\tW}^t \tW \tW^t \mathcal U
+
\dot{\tW}^t\dot{\tW}
=
\dot{\tW}^t \dot{\tW}.
\]
This finishes the proof.
\end{proof}

We are now ready to compute the derivative of $\mathcal U$.

\hypertarget{L:4.6}{
\begin{flemma}
    We have
    \begin{align*}
    \label{eq:4.36}
        \dot{\mathcal{U}}=-\tR \tW-\mathcal{U} \tW^{\, t} \mathcal{U},
        \tag{4.36}
    \end{align*}
    where $\tR$ is the $n \times n$ curvature matrix defined in Section~2.1.2.
\end{flemma}}

\begin{proof}
    By Lemma~\hyperlink{L:4.4}{4.4}, we have
    \begin{align*}
    \label{eq:4.37}
        \dot{\mathcal{U}}
        =\ddot{\tJ} \tJ^{-1} \tW-\dot{\tJ} \tJ^{-1}\dot{\tJ} \tJ^{-1}\tW+\dot{\tJ}\tJ^{-1}\dot{\tW}.
        \tag{4.37}
    \end{align*}
By the Jacobi field equation \eqref{eq:2.15}, we get $\ddot{\tJ} \tJ^{-1} \tW=-\tR \tJ \tJ^{-1} \tW=-\tR \tW$.
Moreover, by Lemma~\hyperlink{L:4.4}{4.4} again and \eqref{eq:4.31}, we obtain
\begin{align*}
    &\kern-2em \dot{\tJ} \tJ^{-1}(\dot{\tJ} \tJ^{-1} \tW-\dot{\tW}) \\
    &=\dot{\tJ} \tJ^{-1}(\mathcal{U}-\mathcal{U}^\perp)
    =\dot{\tJ} \tJ^{-1}\mathcal{U}^\top
    =\dot{\tJ} \tJ^{-1} \tW \tW^t\mathcal{U}
    =\mathcal{U} \tW^t\mathcal{U}.
\end{align*}
Substituting these identities into \eqref{eq:4.37} yields \eqref{eq:4.36}.
This finishes the proof. 
\end{proof}

From Lemma~\hyperlink{L:4.6}{4.6}, we obtain
\begin{align*}
\label{eq:4.38}
\tW^t \dot{\mathcal U}
=
-\tW^t \tR \tW
-
(\tW^t \mathcal U)^2.
\tag{4.38}
\end{align*}
Combining identity \eqref{eq:4.38} with Lemma~\hyperlink{L:4.5}{4.5}, we conclude that
\begin{align*}
\label{eq:4.39}
\bigl(\tr(\tW^t\mathcal U)\bigr)'
&=
\tr(\dot{\tW}^t \dot{\tW})
-
\tr(\tW^t \tR \tW)
-
\tr\bigl((\tW^t\mathcal U)^2\bigr) \tag{4.39} \\
\label{eq:4.40}
&=
\|\mathcal U^\perp\|^2
-
\Ric_k(T_{\gamma(t)}\Sigma_t,\dot\gamma)
-
\tr\bigl((\tW^t\mathcal U)^2\bigr). \tag{4.40}
\end{align*}
Equations \eqref{eq:4.39} and \eqref{eq:4.40} are precisely the equations obtained in \cite[Proposition~4.4]{ketterer_sectional_2018}. Notice that all the matrices involved are of size \(k \times k \).

By contrast, our approach considers the full \(n\times n\) matrix
$\tU=\dot{\tJ} \tJ^{-1}$,
which satisfies
\begin{align*}
\dot{\tU}
&=
-\tR- \tU^2,\\
\bigl(\tr_{\widetilde{\Sigma}(t)} \tU \bigr)'
&=
-\tr_{\widetilde{\Sigma}(t)} \tR
-\tr_{\widetilde{\Sigma}(t)}(\tU^2).
\end{align*}
Another fundamental difference is that our formulation is based on parallel \(k\)-planes \(\widetilde{\Sigma}(t)\) along \(\gamma\), whereas in \cite{ketterer_sectional_2018}, Ketterer and Mondino considered the tangent \(k\)-planes
$T_{\gamma(t)}\Sigma_t$,
which are generally not parallel along \(\gamma\).


\newpage
\bibliographystyle{alpha}
\bibliography{ref.bib}

\Address

\end{document}